\documentclass[preprint,11pt]{elsarticle}

\usepackage[a4paper,width=16cm]{geometry}
\usepackage{graphicx}		
\usepackage{amssymb}
\usepackage{amsmath}
\usepackage{amsthm}
\usepackage{bm}

\usepackage{multirow}
\usepackage{subcaption}
\usepackage{color}

\newcommand{\ud}{\textrm{d}}
\newcommand{\dd}[2]{\frac{\ud #1}{\ud #2}}
\newcommand{\df}[2]{\frac{\partial #1}{\partial #2}}
\newcommand{\con}{\bm{q}}
\newcommand{\prim}{\bm{v}}
\DeclareMathOperator{\ee}{e}

\newcommand{\imh}{{i - \frac{1}{2}}}
\newcommand{\iph}{{i + \frac{1}{2}}}
\newcommand{\jmh}{{j - \frac{1}{2}}}
\newcommand{\jph}{{j + \frac{1}{2}}}

\newcommand{\fl}{\bm{f}}
\newcommand{\gl}{\bm{g}}
\newcommand{\bw}{\bm{w}}

\newcommand{\nfl}{\hat{\fl}}
\newcommand{\ngl}{\hat{\gl}}
\newcommand{\bs}{\bm{s}}

\newcommand{\half}{\frac{1}{2}}
\newcommand{\ohalf}{\frac{3}{2}}
\newcommand{\const}{\textrm{const}}

\newcommand{\intem}{\varepsilon}

\newcommand{\up}{\tilde{u}}

\newcommand{\sign}{\textrm{sign}}
\newcommand{\ue}{\bar{u}}
\newcommand{\pe}{\bar{p}}
\newcommand{\Te}{\bar{T}}

\newcommand{\phe}{\tilde{p}}

\newcommand{\thetae}{\bar{\theta}}
\newcommand{\rhoe}{\bar{\rho}}
\newcommand{\rhohe}{\tilde{\rho}}
\newcommand{\minmod}{\mathcal{M}}

\newcommand{\nx}{{N_x}}
\newcommand{\ny}{{N_y}}

\newcommand{\km}{\textrm{ Km}}

\newtheorem{theorem}{Theorem}
\newtheorem{definition}{Definition}
\newtheorem{remark}{Remark}

\journal{Computers \& Fluids}
\begin{document}

\begin{frontmatter}

\title{A second-order, discretely well-balanced finite volume scheme for Euler equations with gravity}

\author[label2]{Deepak Varma}
\author[label1]{Praveen Chandrashekar}

\address[label2]{ TIFR Center for Applicable Mathematics, Bangalore, India ({\tt deepakvarmars@gmail.com})}
\address[label1]{ TIFR Center for Applicable Mathematics, Bangalore, India ({\tt praveen@math.tifrbng.res.in})}

\begin{abstract}
We present a well-balanced, second order, Godunov-type finite volume scheme for compressible Euler equations with gravity. By construction, the scheme admits a discrete stationary solution which is a second order accurate approximation to the exact stationary solution. Such a scheme is useful for problems involving complex equations of state and/or hydrostatic solutions which are not known in closed form expression. No \'a priori knowledge of the hydrostatic solution is required to achieve the well-balanced property. The performance of the scheme is demonstrated on several test cases in terms of preservation of hydrostatic solution and computation of small perturbations around a hydrostatic solution.
\end{abstract}


\end{frontmatter}

\pagestyle{myheadings}

\section{Introduction}
The Euler equations model the flow of compressible fluids and are a system of non-linear, coupled partial differential equations which mathematically embody the principles of conservation of mass, momentum and energy. In the presence of an external gravitational field, the fluid elements experience an additional force due to gravity which contributes to a source term in the momentum and energy conservation equations. The presence of these source terms leads to non-trivial steady state solutions of the equations. An important class of such solutions are the so called {\em hydrostatic solutions} which are steady solutions with zero velocity. In many problems like weather prediction in the Earth's atmosphere or in a star, the system may be in a hydrostatic state atleast locally and for very long periods of time. Accurately simulating such stationary solutions with a numerical scheme can be a challenging task since commonly used algorithms may not have the property to maintain hydrostatic solutions. {\em A numerical scheme that can exactly maintain a hydrostatic initial condition for all times and on any grid is said to be well-balanced}. Schemes which are not exactly well-balanced have to rely on their consistency to approximately maintain hydrostatic solutions. It becomes necessary to use very high order schemes and/or large grids in order to keep the error to small levels. This becomes especially critical if we are interested in computing small perturbations around the hydrostatic solution, as the truncation errors can dominate the small perturbations leading to very large errors in the solutions. 

In the hydrostatic state, the pressure gradient is balanced by the gravitational force. This balance must be achieved in the numerical scheme also for it to be well-balanced. The hydrostatic state is governed by a first order differential equation whose solution requires assuming additional properties on the solution, e.g., isothermal, polytropic, etc. Since there is a wide variety of hydrostatic solutions depending on the equation of state and other conditions of a particular physical problem, it is not possible to construct a scheme that is well-balanced for a wide variety of problems. The well-balanced property is usually achieved by exploiting the structure of the hydrostatic solution which is used in the reconstruction process of finite volume schemes, and also in discretizing the gravitational source terms. Most of the existing schemes have been constructed to be well-balanced for ideal gas equation of state together with either isothermal or polytropic solutions. The scheme in~\cite{Kappeli2014} is based on isentropic assumption and uses constancy of enthalpy to perform the reconstruction, while~\cite{Kappeli2017} constructs a scheme for isothermal case by using the constancy of Gibbs free energy. The scheme by Chandrashekar and Klingenberg~\cite{Chandrashekar2015} was well-balanced for both isothermal and polytropic solutions and did not require \'a priori knowledge of the temperature (for isothermal case) or polytropic index. Schemes which are well-balanced for arbitrary equations of state have been proposed in \cite{Ghosh2015}, \cite{Ghosh2016}, \cite{Berberich2016} but they require \'a priori knowledge of the hydrostatic solution which is used to rewrite the source term to achieve the well-balanced property. A few well-balanced discontinuous Galerkin methods have also been proposed in recent years. A well-balanced DG scheme for isothermal case was proposed in~\cite{Li2016} and extensions were made in~\cite{Li2018} for polytropic case. Chandrashekar and Zenk~\cite{Chandrashekar2017} proposed well-balanced nodal discontinuous Galerkin methods for isothermal and polytropic solutions but these schemes require \'a priori knowledge of the hydrostatic solution which is used in the source term approximation.  A finite volume scheme that preserves a discrete hydrostatic solution has been constructed in~\cite{Kappeli2016} which does not require \'a priori knowledge of the type of hydrostatic solution that may occur in a particular problem. Note that hydrostatic solutions are not polynomial functions and hence they cannot be exactly represented in a finite element method. What is preserved by the well-balanced finite element method is a projection of the hydrostatic solution onto the finite element space, e.g., \cite{Chandrashekar2017} preserves the interpolation of the hydrostatic solution onto the Lagrange polynomials. Many schemes require the use of particular numerical flux functions to achieve well-balanced property, e.g., \cite{Xing2013}, \cite{Ghosh2016}, \cite{Li2017} make use of a modified local Lax-Friedrich flux where the dissipation operator is altered and \cite{Desveaux2014}, \cite{Desveaux2016} are built using relaxation solvers, while other schemes like in \cite{Chandrashekar2015}, \cite{Chandrashekar2017}, \cite{Kappeli2014}, \cite{Kappeli2017} are well-balanced for any numerical flux function. The latter type of schemes have greater applicability since practitioners can make use of the best available numerical flux functions for their particular problem. Among other approaches, we may mention the wave propagation scheme in~\cite{LeVeque1999} and central schemes in~\cite{Chertock2018}.

In finite volume and finite difference schemes, the basic unknowns are values associated with each volume or point. A well-balanced finite volume scheme may be designed to preserve the interpolation of the exact hydrostatic solution on the mesh. In this work, we do not try to exactly preserve the exact hydrostatic solution, which may not even be known in an explicit form for general equations of state and for realistic problems occuring in astrophysics. Instead, {\em we construct a scheme which admits a hydrostatic solution which is a second order accurate approximation to the exact hydrostatic solution}. This allows us to make the scheme well-balanced for general equations of state and for any type of hydrostatic solution, and we do not need \'a priori knowledge of which type of hydrostatic solution may be realized in a particular problem. Note that our scheme is well-balanced for the discrete hydrostatic solution which can be found numerically. This is in contrast to the scheme in~\cite{Chandrashekar2015} which well-balances the exact hydrostatic solution but only for ideal gas and certain types of hydrostatic solutions, namely isothermal and polytropic. In the Appendix, we give some arguments to show that such a discretely well-balanced scheme can give consistent and accurate approximation to small perturbations around a hydrostatic solution. Any consistent numerical flux function that exactly resolves stationary contact discontinuities can be used which makes this method of interest to practitioners who want to use their own favourite numerical flux function. We construct this scheme on Cartesian grids but the extension to general curvilinear grids can also be made. The well-balanced property is achieved by using two main techniques: (1) a hydrostatic reconstruction scheme and (2) a discretization of gravitational source terms in terms of a local hydrostatic solution.  These ingredients are similar to what is used in~\cite{Chandrashekar2015} but with some modifications that allow us to deal with general hydrostatic solutions, whereas the previous work was restricted to ideal gas equation of state. Both of these techniques are based on exploiting a particular structure of the hydrostatic solution. A restriction of this scheme and also the one in~\cite{Kappeli2016}, is that the well-balanced property requires the gravitational force to be aligned with the grid lines. In atmospheric or stellar problems where the gravity is predominantly in the radial direction, this requirement can be satisfied with the use of polar or spherical grids. However we show in numerical solutions that in practice, the hydrostatic state is quite well maintained even otherwise.

The rest of the paper is organized as follows. In section~(\ref{sec:1d}), we introduce the 1-D Euler equations and some standard equations of state that are of common interest. In section~(\ref{sec:fvm1d}), we give a full description of our approach to well-balanced scheme and prove this as a theorem. The extension to two dimensions is made in section~(\ref{sec:fvm2d}) where we also discuss the treatment of some boundary conditions. Sections~(\ref{sec:res1d}) and~(\ref{sec:res2d}) present several numerical results in one and two dimensions to show the performance of the current scheme. We end the paper with some conclusions in section~(\ref{sec:end}).
\section{1-D Euler equations with gravity}
\label{sec:1d}
Consider the system of compressible Euler equations in one dimension which models conservation of mass, momentum and energy and are given by
\begin{eqnarray*}
\df{\rho}{t} + \df{}{x}(\rho u) &=& 0 \\
\df{}{t}(\rho u) + \df{}{x}(p+\rho u^2) &=& -\rho \df{\phi}{x} \\
\df{E}{t} + \df{}{x}(Eu + pu) &=& -\rho u \df{\phi}{x}
\end{eqnarray*}
Here $\rho$ is the density, $u$ is the velocity, $p$ is the pressure, $E$ is the energy per unit volume excluding the gravitational energy and $\phi$ is the gravitational potential. The  energy $E$ is given by
\[
E = \rho \intem + \half \rho u^2
\]
where $\intem$ is the internal energy per unit mass. We can write the above set of coupled equations in a compact notation as
\[
\df{\con}{t} + \df{\fl}{x} = \begin{bmatrix}
0 \\ s \\ s u \end{bmatrix}, \qquad s = -\rho \df{\phi}{x}
\]
where $\con=[\rho, \ \rho u, \ E]^\top$ is the set of conserved variables and $\fl = [\rho u, \ p + \rho u^2, \ (E+p)u]^\top$ is the corresponding flux vector. In the case of a self gravitating system, the gravitational potential $\phi$ is governed by a Poisson-type equation whose details are not relevant for the present discussion. We will consider the case of static gravitational potential which is assumed to be given as a function of the spatial coordinates. However the scheme we propose can be used for time dependent potential also. To simplify some of the later notation, we also denote the set of primitive variables by $\prim=[\rho, \ u, \ p]^\top$.
\subsection{Hydrostatic states}
Consider the hydrostatic stationary solution, i.e., the solution for which the velocity is zero
\[
\ue = 0
\]
We will denote the quantities in a hydrostatic solution with an overbar. In this case, the mass and energy conservation equations are automatically satisfied. The momentum equation becomes an ordinary differential equation given by
\begin{equation}
\dd{\pe}{x} = -\rhoe \dd{\phi}{x}
\label{eq:hydro}
\end{equation}
We have only one equation but two thermodynamic quantities, $\rhoe$, $\pe$, to be determined. The pressure, density and temperature are related by an equation of state. We will write the equation of state in the form
\begin{equation}
p = \rho \theta
\label{eq:eostheta}
\end{equation}
where $\theta$ is in general a function of the thermodynamic variables, e.g., $\theta = \theta(p,T)$ or $\theta = \theta(\rho,T)$, with $T$ being the temperature. This form of equation of state is commonly used in astrophysical applications where the function $\theta$ may be given in a tabular form. Integrating the hydrostatic equation~(\ref{eq:hydro}) we obtain
\begin{equation}
\pe(x) = p_0 \ee^{\psi(x)}, \qquad \psi(x) = -\int_{x_0}^x\frac{\phi'(s)}{\theta(\pe(s), \Te(s))}\ud s
\label{eq:hydrostruct}
\end{equation}
In the above equation, $p_0$ is the pressure at some reference position $x_0$. This is a {\em general relation that holds for any hydrostatic solution and any equation of state, and is the crucial building block for our scheme}. The above integral can be explicitly evaluated in some simple situations. We next give several examples to illustrate that the equations of state can be put in the form~(\ref{eq:eostheta}) and show the hydrostatic relation in some simple situations.
\subsubsection{Ideal gas}
For an ideal gas model we have $\theta = RT$ where $R$ is the gas constant. Assuming some equilibrium temperature profile $\Te(x)$, the pressure can be computed using~(\ref{eq:hydrostruct}) as
\[
\pe(x) = p_0 \exp\left(-\int_{x_0}^x\frac{\phi'(s)}{R\Te(s)}\ud s\right)
\]
If the hydrostatic state is {\em isothermal}, i.e., $\Te(x) = T_0 = \const$, then
\begin{equation}
\pe(x) \exp\left(\frac{\phi(x)}{R T_0}\right) = \const
\label{eq:isot}
\end{equation}
If the hydrostatic solution is {\em polytropic}~\cite{Chandrasekhar1967}, then we have the following relations satisfied
\begin{equation}
\pe \rhoe^{-\nu} = \const, \qquad \pe \Te^{-\frac{\nu}{\nu-1}} = \const, \qquad \rhoe \Te^{-\frac{1}{\nu-1}} = \const
\label{eq:isen}
\end{equation}
where $\nu > 1$ is some constant. Using these polytropic relations in the hydrostatic equation~(\ref{eq:hydro}) and performing an integration, we obtain
\begin{equation}
\frac{\nu R \Te(x)}{\nu - 1} + \phi(x) = \const
\label{eq:poly}
\end{equation}
In this case the temperature profile is determined once the potential is fixed. If $\nu = \gamma$ then we have an isentropic solution.
\subsubsection{van der Waals equation of state}
\label{sec:vdWeos}
Unlike the ideal gas which is assumed to be made of point particles, the van der Waals equation of state accounts for the volume occupied by the molecules composing the gas and the inter-molecular forces, and is useful in modeling problems with phase transitions. The equation of state is given by~\cite{VANDERWAALS1873}
\[
p = \frac{\rho R_u T}{M - \rho b} - a\left(\frac{\rho}{M}\right)^2
\]
where $R_u = 8.314 J/kgK$ is the universal gas constant, $a$ and $b$ are characteristic values for an individual gas and $M$ is its molar mass. We can write this in the form~(\ref{eq:eostheta}) where
\[
\theta = \theta(\rho,T) = \frac{R_u T}{M - \rho b} - \frac{a\rho}{M^2}
\]
If we assume a condition of polytropic van der Waals gas, i.e., the specific heat at constant volume does not change with temperature, the total energy per unit volume for the system can be written as \cite{guardone2002}
	\[
	E = \frac{\rho R_u T}{M(\gamma - 1)} + \frac{1}{2}\rho u^2  - a \left( \frac{\rho}{M} \right) ^2
	\]
The acoustic speed for a polytropic van der Waals condition is derived to be 
\[
c =\sqrt{\frac{\gamma p M + a\rho^2}{\rho(M - \rho b)} - 2a\frac{\rho}{M}}
\]
which is used in the HLLC numerical flux function. Since we may not have explicitly known hydrostatic solutions, it has to be calculated using a numerical quadrature of the hydrostatic equation.
\subsubsection{Ideal gas with radiation pressure}
The pressure exerted by radiation becomes significant inside stellar atmospheres due to the presense of high temperatures and hence it must be included in the equation of state. The equation of state is given by~\cite{Chandrasekhar1967}
\[
p = \rho R T + \frac{1}{3} aT^4
\]
where $a$ is the Stefan-Boltzmann constant. This equation can be written in the form~(\ref{eq:eostheta}) where
\[
\theta = \theta(p,T) = \frac{pRT}{p - \frac{1}{3} aT^4}
\]
and the corresponding internal energy is given by
\[
\rho \intem = \frac{\rho R T}{\gamma-1} + a T^4
\]
In this case, we may not have explictly known hydrostatic solutions available to us. The hydrostatic solution has to be computed through some numerical quadrature of the hydrostatic equation.

\section{1-D finite volume scheme}
\label{sec:fvm1d}
Let us divide the domain into $N$ finite volumes each of size $\Delta x$. The $i$'th cell is given by the interval $(x_\imh, x_\iph)$ and let $x_i = \half(x_\imh + x_\iph)$ be the cell center. Consider the semi-discrete finite volume scheme for the $i$'th cell
\begin{equation}
\dd{\con_i}{t} + \frac{\nfl_\iph - \nfl_\imh }{\Delta x} = 
\begin{bmatrix}
0 \\
s_i \\
s_i u_i \end{bmatrix}
\label{eq:semifvm}
\end{equation}
where $s_i$ is a consistent approximation of the gravitational source term and $\nfl_\iph = \nfl(\prim_\iph^L, \prim_\iph^R)$ is the numerical flux function which we write as a function of primitive variables. The scheme will be completely specified once the source term and numerical flux schemes are explained. For later use, let us define a piecewise linear interpolation $\phi_h(x)$ of the gravitational potential 
\[
\phi_h(x) = \frac{x_{i+1} - x}{x_{i+1} - x_i} \phi_i + \frac{x - x_i}{x_{i+1} - x_i} \phi_{i+1}, \qquad x \in [x_i, x_{i+1}]
\]
and let $\theta_h(x)$ be a piecewise constant function defined as
\[
\theta_h(x) = \theta_i = \theta(p_i, T_i), \qquad x \in (x_\imh, x_\iph)
\]
Motivated by the structure of the hydrostatic solution given by equation~(\ref{eq:hydrostruct}), define
\begin{equation}
\psi_h(x) = -\int^x \frac{\phi'_h(s)}{\theta_h(s)} \ud s
\label{eq:psih}
\end{equation}
where the lower limit of the integral is arbitrary and not relevant for the definition of the scheme since only differences of the function $\psi_h$ are actually used in the scheme. Note that $\psi_h(x)$ is a continuous function. This scheme differs from the scheme in \cite{Chandrashekar2015} which was specific to ideal gas so that $\theta = \theta(T) = T$ and we used a logarithmic average for the definition of $T$ in a piecewise manner over the intervals $[x_i,x_{i+1}]$. The current scheme is much simpler since it does not involve logarithmic averages and the value of $\theta$ is defined to be piecewise constant over each cell.
\subsection{Numerical flux}
The numerical flux $\nfl_\iph = \nfl(\prim_\iph^L, \prim_\iph^R)$ is computed using two states at the interface $x_\iph$ which are obtained by a reconstruction process. We will assume that the function $\nfl(\prim^L, \prim^R)$ is consistent in the sense that $\nfl(\prim,\prim) = \fl(\prim)$. Moreover, we assume the following property is also satisfied.

\noindent
{\bf Contact Property} {\em The numerical flux $\nfl$ is said to satisfy contact property if for any two states $\prim^L = [\rho^L, 0, p]$ and $\prim^R = [\rho^R, 0, p]$ we have
\[
\nfl(\prim^L,\prim^R) = [0, \ p, \ 0]^\top
\]}
The states $\prim^L$, $\prim^R$ in the above definition correspond to a stationary contact discontinuity. The above property is equivalent to the ability of a numerical flux to exactly support a stationary contact discontinuity. A few important examples of numerical fluxes which satisfy this property are the Roe flux~\cite{Roe1981} and the HLLC flux~\cite{Toro1994}.
\subsection{Reconstruction scheme}
\label{sec:recon}
To obtain the states $\prim_\iph^L, \prim_\iph^R$ at the cell boundary which are required to calculate the numerical flux $\nfl_\iph$, we will reconstruct the following set of variables
\[
\bw = \left[ \rho \ee^{-\psi_h}, \ u, \ p \ee^{-\psi_h} \right]^\top
\]
which is motivated by the structure of the hydrostatic solution as shown in equation~(\ref{eq:hydrostruct}). To perform the reconstruction at $x_\iph$, define $\psi_h(x)$ in equation~(\ref{eq:psih}) relative to $x_\iph$, and compute the quantities $\psi_i = \psi_h(x_i)$, $\psi_{i \pm 1} = \psi_h(x_{i \pm 1})$ and $\psi_{i + 2} = \psi_h(x_{i + 2})$ from the definition of $\psi_h$ as
\[
\psi_i = -\int^{x_i}_{x_\iph} \frac{\phi'_h(s)}{\theta_h(s)} \ud s = - \frac{\phi_i - \phi_\iph}{\theta(p_i, T_i)}, \quad \psi_{i+1} = -\int^{x_{i+1}}_{x_\iph} \frac{\phi'_h(s)}{\theta_h(s)} \ud s = - \frac{\phi_{i+1} - \phi_\iph}{\theta(p_{i+1}, T_{i+1})}
\]
\[
\psi_{i-1} = -\int^{x_{i-1}}_{x_\iph} \frac{\phi'_h(s)}{\theta_h(s)} \ud s = -\int^{x_{i-1}}_{x_\imh} \frac{\phi'_h(s)}{\theta_h(s)} \ud s - \int^{x_\imh}_{x_\iph} \frac{\phi'_h(s)}{\theta_h(s)} \ud s = - \frac{\phi_{i-1} - \phi_\imh}{\theta(p_{i-1}, T_{i-1})} - \frac{\phi_\imh - \phi_\iph}{\theta(p_i, T_i)}
\]
\[
\psi_{i+2} = -\int^{x_{i+2}}_{x_\iph} \frac{\phi'_h(s)}{\theta_h(s)} \ud s = -\int^{x_{i+2}}_{x_{i+\frac{3}{2}}} \frac{\phi'_h(s)}{\theta_h(s)} \ud s - \int^{x_{i+\frac{3}{2}}}_{x_\iph} \frac{\phi'_h(s)}{\theta_h(s)} \ud s = -\frac{\phi_{i+2} - \phi_{i+\frac{3}{2}}}{\theta(p_{i+2}, T_{i+2})} - \frac{\phi_{i+\frac{3}{2}} - \phi_\iph}{\theta(p_{i+1},T_{i+1})}
\]
The values of the potential at the faces are obtained by the piecewise linear interpolation so that $\phi_\iph = \half(\phi_i + \phi_{i+1})$, etc. The $\bw$ variables at the cell centers are defined as
\[
\bw_j = \left[ \rho_j \ee^{-\psi_j}, \ u_j, \ p_j \ee^{-\psi_j} \right]^\top, \qquad j=i-1,i,i+1,i+2
\]
Using these values, we can use any reconstruction scheme to obtain $\bw_\iph^L$ and $\bw_\iph^R$. E.g., the minmod scheme is given by
\begin{eqnarray*}
\bw_\iph^L &=& \bw_i + \half \minmod( \theta(\bw_i - \bw_{i-1}), (\bw_{i+1}-\bw_{i-1})/2, \theta(\bw_{i+1} - \bw_i) ) \\
\bw_\iph^R &=& \bw_{i+1} - \half \minmod( \theta(\bw_{i+1} - \bw_{i}), (\bw_{i+2}-\bw_{i+1})/2, \theta(\bw_{i+2} - \bw_{i+1}) )
\end{eqnarray*}
where $\theta \in [1,2]$ and $\minmod(\cdot,\cdot,\cdot)$ is the minmod limiter function given by
\[
\minmod(a,b,c) = \begin{cases}
s \min(|a|, |b|, |c|) & \textrm{if } s = \sign(a) = \sign(b) = \sign(c) \\
0 & \textrm{otherwise}
\end{cases}
\]
The primitive variables $\prim$ are obtained from the $\bw$ variables leading to
\[
\rho_\iph^L = \ee^{\psi_\iph} (w_1)_\iph^L, \qquad u_\iph^L = (w_2)_\iph^L, \qquad p_\iph^L = \ee^{\psi_\iph} (w_3)_\iph^L, \quad \textrm{etc.}
\]
where $\psi_\iph = \psi_h(x_\iph)$. Using these values, we can now compute the numerical flux. Note that $\psi_\iph = 0$ by definition and hence the exponential factors can be dropped when we convert to the primitive variables.
\begin{remark}
We have explained the scheme assuming that the equation of state has the form $\theta = \theta(p,T)$. In some applications, the function $\theta$ may be given in the form $\theta= \theta(\rho,T)$ and the above approach can be used in this case also. If $p$ is given as some function of $\rho,T$, then we can define $\theta(\rho,T) = p(\rho,T)/\rho$.
\end{remark}
\subsection{Approximation of source term}
The source term could be calculated in a straight-forward way by using a finite difference approximation of the potential gradient. However for achieving well balanced property, it is necessary to approximate it in a different but equivalent manner. Inside the $i$'th cell, using the pressure at the cell center, we will estimate the pressure at the faces assuming that the solution is in a hydrostatic state, i.e., using equation~(\ref{eq:hydrostruct}). Thus the local hydrostatic pressures at the faces inside the $i$'th cell are given by
\begin{equation}
\pe_\iph^L = p_i \ee^{\psi_\iph - \psi_i} = p_i \ee^{ - \frac{\phi_{i+1} - \phi_i}{2\theta_i}}, \qquad \pe_\imh^R = p_i \ee^{\psi_\imh - \psi_i} = p_i \ee^{  \frac{\phi_{i} - \phi_{i-1}}{2\theta_i}}
\label{eq:hydrop}
\end{equation}
Then the source term in the momentum equation is approximated as
\begin{equation}
s_i = \frac{\pe_\iph^L - \pe_\imh^R}{\Delta x}
\label{eq:src}
\end{equation}
In the next theorem we show that this is a consistent approximation. The source term discretization makes use of solution values only within the cell in contrast to the scheme used in~\cite{Chandrashekar2015}. This scheme is hence expected to be better for discontinuous hydrostatic solutions, e.g., where the density is discontinuous.
\begin{theorem}
The source term discretization given by (\ref{eq:hydrop}), (\ref{eq:src}) is second order accurate.
\end{theorem}
\begin{proof}
The source term discretization is given by
\[
s_i = p_i \frac{\exp\left(\frac{\phi_i - \phi_{i+1}}{2\theta_i}\right) - \exp\left(\frac{\phi_i - \phi_{i-1}}{2\theta_i}\right)}{\Delta x}
\]
Performing a Taylor expansion of the potential around $x_i$ we get
\begin{eqnarray*}
\ee^{\frac{\phi_{i} - \phi_{i+1}}{2\theta_i}} - \ee^{\frac{\phi_{i} - \phi_{i-1}}{2\theta_i}} 
&=& \ee^{\frac{1}{2\theta_i}(-\phi'_i \Delta x - \phi''_i \Delta x^2 + O(\Delta x^3))} - \ee^{\frac{1}{2\theta_i}(+\phi'_i \Delta x - \phi''_i \Delta x^2 + O(\Delta x^3))} \\
&=& \left[ 1 + \frac{1}{2\theta_i}(-\phi'_i \Delta x - \phi''_i \Delta x^2) + \frac{1}{2(2\theta_i)^2} (\phi'_i \Delta x)^2 + O(\Delta x^3) \right] \\
&& - \left[ 1 + \frac{1}{2\theta_i}(\phi'_i \Delta x - \phi''_i \Delta x^2) + \frac{1}{2(2\theta_i)^2} (\phi'_i \Delta x)^2 + O(\Delta x^3) \right] \\
&=& -\frac{1}{\theta_i} \phi'(x_i) \Delta x + O(\Delta x^3)
\end{eqnarray*}
Therefore
\[
s_i = -\frac{p_i}{\theta_i} \phi'(x_i) + O(\Delta x^2) = -\rho_i \phi'(x_i) + O(\Delta x^2)
\]
Hence the source term discretization is second order accurate.
\end{proof}
\subsection{Well-balanced property}
We now state the basic result on the well-balanced property. We essentially show that the finite volume schemes admit a stationary solution which is a discrete analogue of~(\ref{eq:hydrostruct}).
\begin{theorem}
The finite volume scheme (\ref{eq:semifvm}) together with a numerical flux which satisfies contact property and reconstruction of $\bw$ variables is well-balanced in the sense that the initial condition given by
\begin{equation}
u_i = 0, \qquad p_i \exp(-\psi_i) = \const, \qquad \forall \ i
\label{eq:ic}
\end{equation}
is preserved by the numerical scheme.
\end{theorem}
\begin{proof}
Let us start the computations with an initial condition that satisfies~(\ref{eq:ic}).  Since we reconstruct the variables $\bw$ and by our assumption the second and third components are constant, at any interface $\iph$ we have
\[
(w_2)_\iph^L = (w_2)_\iph^R = 0, \qquad (w_3)_\iph^L = (w_3)_\iph^R = p_i \exp(-\psi_i)
\]
and hence, since $\psi_h(x)$ is a continuous function, we get
\[
u_\iph^L = u_\iph^R = 0, \qquad p_\iph^L = p_\iph^R = p_i \exp(\psi_\iph - \psi_i) =: p_\iph
\]
and similarly at $\imh$
\[
u_\imh^L = u_\imh^R = 0, \qquad p_\imh^L = p_\imh^R = p_i \exp(\psi_\imh - \psi_i) =: p_\imh
\]
The above equations together with~(\ref{eq:hydrop}) imply that $\pe_\iph^L = p_\iph$ and $\pe_\imh^R = p_\imh$. In general, the density may not be continuous, i.e., $\rho_\iph^L \ne \rho_\iph^R$. But since the numerical flux satisfies contact property, we have $\nfl_\imh = [0, p_\imh, 0]^\top$ and $\nfl_\iph = [0, p_\iph, 0]^\top$. 
The flux in mass and energy equations are zero and the gravitational source term in the energy equation is also zero. Hence the mass and energy equations are already well balanced, i.e., $\dd{\con_i^{(1)}}{t}=0$ and $\dd{\con_i^{(3)}}{t}=0$. 
 It remains to check the momentum equation. On the left of the momentum equation, we have
\[
\frac{\nfl_\iph^{(2)} - \nfl_\imh^{(2)} }{\Delta x}  = \frac{p_\iph - p_\imh}{\Delta x} 
\]
while on the right, the source term takes the form
\[
s_i = \frac{\pe_\iph^L - \pe_\imh^R}{\Delta x}
 = \frac{p_\iph - p_\imh}{\Delta x}
\]
and hence $\dd{\con_i^{(2)}}{t} = 0$.  This proves that the initial condition satisfying~(\ref{eq:ic}) is preserved under any time integration scheme.
\end{proof}
\begin{remark}
Note that the scheme preserves a state that satisfies the conditions given in~(\ref{eq:ic}). The exact hydrostatic solution may not satisfy these conditions since the quantity $\psi_i$ is only known approximately in general. See the next theorem for a special case where we obtain exact well-balanced property. In general, the scheme well-balances an approximate hydrostatic solution which is second order accurate approximation to the true hydrostatic solution. We provide some theoretical justification in the Appendix for why such a scheme can still give good approximations.
\end{remark}
\begin{remark}
It is possible to reconstruct density and still retain the result of the previous theorem. In the isothermal case, the quantity $\rho \ee^{-\psi}$ is constant and we can expect the reconstruction of density to be more accurate if we scale the density as in the $\bw$ variables.
\end{remark}

\begin{definition}[Exact hydrostatic solution]
\label{def:exhydro}
Let $\pe(x)$, $\rhoe(x)$ be a hydrostatic solution which is available in closed form expression. Then the interpolation of this solution on the grid
\[
\pe_i = \pe(x_i), \qquad \rhoe_i = \rhoe(x_i)
\]
will be refered to as the exact hydrostatic solution.
\end{definition}
\begin{remark}
The above theorem says that the finite volume scheme admits hydrostatic solutions. This does not mean that it can preserve every exact hydrostatic solution. The next theorem shows that the finite volume scheme is well-balanced for the exact hydrostatic solution in a special situation.
\end{remark}
\begin{theorem}
\label{thm:iso}
Any hydrostatic solution for ideal gas EOS which is isothermal is exactly preserved by the finite volume scheme (\ref{eq:semifvm}).
\end{theorem}

\begin{proof}
Assume that the initial condition is taken to be the interpolation of an isothermal hydrostatic solution as in definition~(\ref{def:exhydro}). We have to verify that the initial condition satisfes equation~(\ref{eq:ic}). If the initial condition is isothermal, then $T_i = T_0 = \const$, we obtain
\[
\frac{p_{i+1} \ee^{-\psi_{i+1}}}{p_i \ee^{-\psi_i}} = \frac{p_{i+1}}{p_i} \ee^{\psi_{i} - \psi_{i+1}} = \frac{p_{i+1}}{p_i} \exp\left( \frac{\phi_{i+1} - \phi_{i}}{RT_0} \right) = \frac{p_{i+1} \exp(\phi_{i+1}/RT_0)}{p_i \exp(\phi_i/RT_0)} = 1
\]
where the last equality follows from~(\ref{eq:isot}).
\end{proof}

\subsection{Computation of discrete hydrostatic solution}
\label{sec:dischydro}
The new scheme is well-balanced at a state that satisfies conditions~(\ref{eq:ic}) which is essentially a numerical quadrature of the hydrostatic equation~(\ref{eq:hydro}). Let us discuss how to compute such a solution. We will assume that the equilibrium temperature $\Te(x)$ and the corresponding potential $\phi(x)$ are known at the grid points.  We will use equation~(\ref{eq:ic}) to determine the pressure and density at all the grid points. The hydrostatic equation is a first order ordinary differential equation. Hence it requires an initial condition; assume that we are given the pressure $\phe_1$ at the first grid point.
If the pressure at $i-1$ is known then the pressure at $i$ is given by
\[
\phe_{i-1} \ee^{-\psi_{i-1}} = \phe_i \ee^{-\psi_i}, \quad \mbox{i.e.,} \quad \phe_i = \phe_{i-1} \ee^{\psi_i - \psi_{i-1}}
\]
By the definition of $\psi_i = \psi_h(x_i)$, we get
\begin{eqnarray*}
\psi_i - \psi_{i-1} &=& -\int_{x_{i-1}}^{x_\imh} \frac{\phi'_h(s)}{\theta(\phe_{i-1},\Te_{i-1})} \ud s - \int_{x_\imh}^{x_i} \frac{\phi'_h(s)}{\theta(\phe_i,\Te_i)} \ud s \\
&=& - \frac{1}{\theta(\phe_{i-1},\Te_{i-1})} \int_{x_{i-1}}^{x_\imh} \phi'_h(s) \ud s - \frac{1}{\theta(\phe_i,\Te_i)} \int_{x_\imh}^{x_i} \phi'_h(s) \ud s \\
&=& - \frac{1}{\theta(\phe_{i-1},\Te_{i-1})} (\phi_\imh - \phi_i) - \frac{1}{\theta(\phe_i,\Te_i)} (\phi_i - \phi_\imh) \\
&=& -\frac{1}{2} (\phi_i - \phi_{i-1}) \left( \frac{1}{\theta(\phe_{i-1},\Te_{i-1})} + \frac{1}{\theta(\phe_i,\Te_i)} \right) \qquad \textrm{using} \quad \phi_\imh = \half (\phi_{i-1} + \phi_i)
\end{eqnarray*}
Define
\[
f(p) = p - \phe_{i-1} \exp\left[ -\frac{1}{2} (\phi_i - \phi_{i-1}) \left( \frac{1}{\theta(\phe_{i-1},\Te_{i-1})} + \frac{1}{\theta(p,\Te_i)} \right) \right]
\]
The pressure $\phe_i$ is the root of the function $f$. The root finding problem is solved using a Newton method for which a good initial guess for the root is given by $\phe_{i-1}$. Once the pressure has been computed the density is obtained from the equation of state $\rhohe_i = \phe_i/\theta(\phe_i, \Te_i)$ where $\Te_i = \Te(x_i)$. For equations of state with $\theta$ = $\theta(\rho, T)$ such as the van der Waals equation, we define a root-finding problem for which the root is the density $\tilde{\rho_i}$ of the function
	\[
	f(\rho) = p(\rho, \bar{T_i}) - \widetilde{p}_{i-1}\exp\left[ -\frac{1}{2}(\phi_i - \phi_{i-1}) \left( \frac{1}{\theta(\widetilde{\rho}_{i-1}, \bar{T}_{i-1})} + \frac{1}{\theta(\rho, \bar{T_{i}})}\right) \right]
	\] \label{eqn:frho}
	The pressure can be computed using $\widetilde{p}_{i}$ = $\widetilde{\rho}_{i} \theta(\widetilde{\rho}_{i}, \bar{T}_i)$.
\begin{definition}[Discrete hydrostatic solution]
The set of grid point values $\phe_i$, $\rhohe_i$ obtained from the procedure described in the previous section will be refered to as discrete hydrostatic solution.
\end{definition}
\begin{remark}
It is  not necessary to compute the discrete hydrostatic solution in practical computations. We will use this only to demonstrate numerically that the scheme is well-balanced if we start the computations using the discrete hydrostatic solution as the initial condition.
\end{remark}
\begin{remark}
For the ideal gas model and an isothermal hydrostatic solution, theorem~(\ref{thm:iso}) shows that $\phe_i = \pe_i$ and $\rhohe_i = \rhoe_i$.
\end{remark}
\section{Two dimensional scheme}
\label{sec:fvm2d}
The two dimensional Euler equations with gravity are a system of four balance laws which can be written as
\[
\df{\con}{t} + \df{\fl}{x} + \df{\gl}{y} = \bs
\]
where $\con=[\rho, \ \rho u, \ \rho v, \ E]^\top$ is the set of conserved variables and $\fl = [\rho u, \ p + \rho u^2, \ \rho u v, \ (E+p)u]^\top$, $\gl = [\rho v, \ \rho u v, \ p + \rho v^2, \ (E+p)v]^\top$ are the corresponding flux vector in Cartesian coordinates.  Here $\rho$ is the density, $(u,v)$ is the velocity, $p$ is the pressure, $E$ is the energy per unit volume excluding the gravitational energy and $\phi$ is the gravitational potential. The  energy $E$ is given by
\[
E = \rho \intem + \half \rho (u^2 + v^2)
\]
where $\intem$ is the internal energy per unit mass. The source term is $\bs = [0, \ \alpha, \ \beta, \ u\alpha + v\beta]^\top$ where $\alpha = -\rho\df{\phi}{x}$ and $\beta=-\rho\df{\phi}{y}$. We also define the set of primitive variables by $\prim=[\rho, \ u, \ v, \ p]^\top$. The hydrostatic solution satisfies the equation
\[
\nabla\pe = - \rhoe \nabla\phi
\]
Similar to the one dimensional case, define
\[
\psi(x,y) = -\int^{(x,y)} \frac{\nabla \phi(x',y') \cdot \ud \bm{r}'}{\theta(x',y')}
\]
which is independent of the path if we are at hydrostatic solution. Moreover the hydrostatic solution satisfies
\[
\pe(x,y) \ee^{-\psi(x,y)} = \const
\]
In the finite volume scheme, we will choose a path along $x$ or $y$ axis depending on the flux that is being evaluated. For example, for the computation of the $x$ component of the flux, the reconstruction is made using a 1-D scheme along the $x$ direction and the path is parallel to the $x$-axis. The discretization of the balance law on Cartesian mesh is obtained in a straight-forward way by using the one-dimensional scheme along the two coordinate directions, and hence we give only a brief description of the 2-D scheme.
\subsection{Finite volume scheme}
\begin{figure}
\begin{center}
\begin{tabular}{cc}
\includegraphics[width=0.48\textwidth]{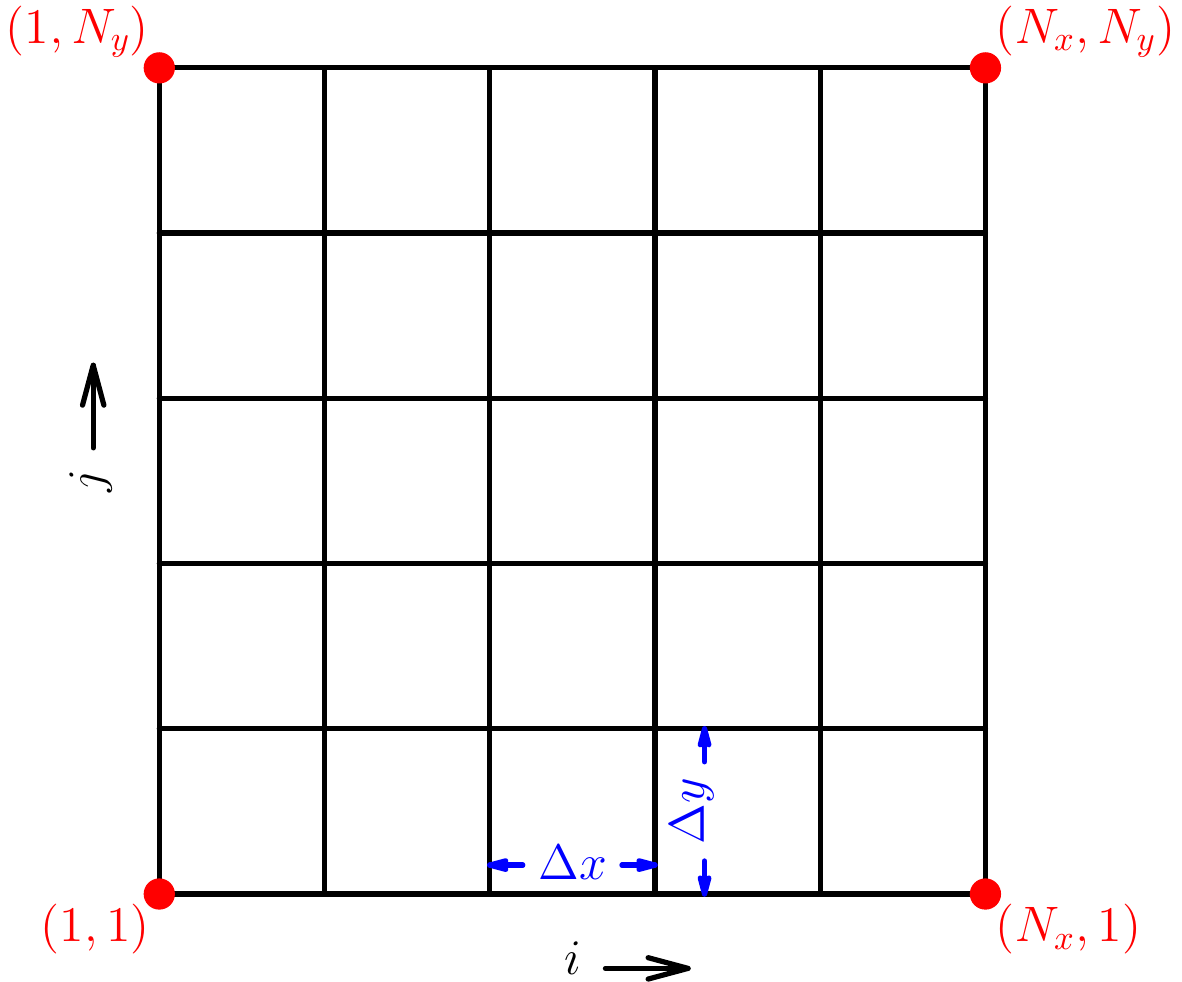} &
\includegraphics[width=0.45\textwidth]{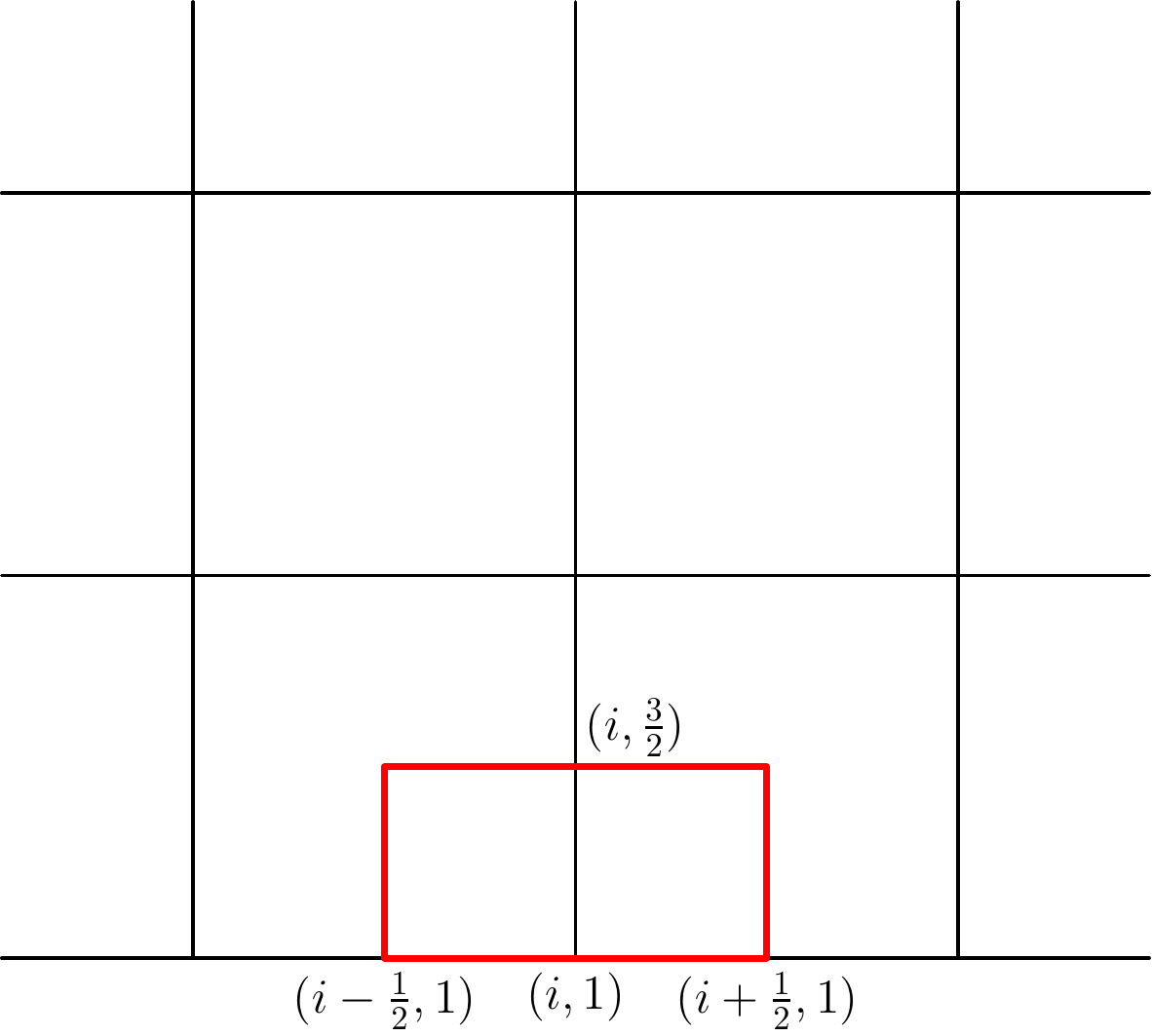} \\
(a) & (b)
\end{tabular}
\caption{Schematic of 2-D grid: (a) point indexing, (b) half cell on bottom boundary.}
\label{fig:grid}
\end{center}
\end{figure}
Consider a partition of the computational domain $\Omega = [x_{min},x_{max}] \times [y_{min}, y_{max}]$ by a Cartesian mesh consisting of $\nx$ and $\ny$ number of points along the coordinate directions, see figure~(\ref{fig:grid}a). The coordinates of the points are given by
\[
x_i = x_{min} + (i-1)\Delta x, \quad 1 \le i \le \nx, \qquad y_j = y_{min} + (j-1)\Delta y, \quad 1 \le j \le \ny
\]
where $\Delta x = (x_{max} - x_{min})/(\nx - 1)$ and $\Delta y = (y_{max} - y_{min})/(\ny - 1)$ are the cell sizes. Note that we have grid points located on the boundary of the computational domain which is commonly refered to as a {\em cell-vertex scheme}. Around each grid point $(i,j)$, we construct a cell $\Omega_{ij} = [x_\imh, x_\iph] \times [y_\jmh, y_\jph]$ where $x_{i \pm \half} = x_i \pm \half \Delta x$ and $y_{j \pm \half} = y_j \pm \half \Delta y$. For the points on a solid wall boundary, the cells are only half the size. For e.g., on the bottom boundary for which $j=1$, the cells are defined by $\Omega_{i,1} = [x_\imh, x_\iph] \times [y_{min},y_{min}+\half\Delta y]$ (see figure~(\ref{fig:grid}b))  whereas on the left boundary for which $i=1$, the cells are defined as $\Omega_{1,j} = [x_{min}, x_{min}+\half\Delta x] \times [y_\jmh, y_\jph]$. If we are using periodic or transmissive boundary condition in some direction, then the cells on the boundary of the corresponding direction are of full size.

The semi-discrete finite volume scheme is given by
\begin{equation}
\label{eq:fvint}
\dd{\con_{i,j}}{t} + \frac{\nfl_{\iph,j} - \nfl_{\imh,j}}{\Delta x} + \frac{\ngl_{i,\jph} - \ngl_{i,\jmh}}{\Delta y}= \bs_{i,j}
\end{equation}
for all interior cells. The source term is computed as $\bs_{i,j} = [0, \ \alpha_{i,j}, \ \beta_{i,j}, \ u_{i,j}\alpha_{i,j} + v_{i,j}\beta_{i,j}]^\top$ where
\[
\alpha_{i,j} = \frac{\pe_{\iph,j}^L - \pe_{\imh,j}^R}{\Delta x}, \qquad \beta_{i,j} = \frac{\pe_{i,\jph}^L - \pe_{i,\jmh}^R}{\Delta y}
\]
where the hydrostatic pressures $\pe_{\iph,j}^L$, etc., are computed by applying the one-dimensional scheme along each of the coordinate directions. If the bottom boundary is a solid wall, then the cells are only half the size as shown in figure~(\ref{fig:grid}b) and the semi-discrete finite volume scheme for cells on this boundary is of the form
\begin{equation}
\label{eq:fvwall}
\dd{\con_{i,1}}{t} + \frac{\nfl_{\iph,1} - \nfl_{\imh,1}}{\Delta x} + \frac{\ngl_{i,\ohalf} - \ngl_{i,1}}{\half\Delta y}= \bs_{i,1}
\end{equation}
The source term is $\bs_{i,1} = [0, \ \alpha_{i,1}, \ \beta_{i,1}, \ u_{i,1}\alpha_{i,1}]^\top$ where
\[
\alpha_{i,1} = \frac{\pe_{\iph,1}^L - \pe_{\imh,1}^R}{\Delta x}, \qquad \beta_{i,1} = \frac{\pe_{i,\ohalf}^L - p_{i,1}}{\half\Delta y}
\]
The computation of the wall pressure and fluxes are explained in a later section. The reconstruction is performed in terms of the $\bw$ variables which are defined as
\[
\bw = [\rho \ee^{-\psi_h}, \ u, \ v, \ p \ee^{-\psi_h} ]^\top
\]
where $\psi_h$ is defined as
\[
\psi_h(x,y) = - \int^{(x,y)} \frac{\nabla \phi_h(x',y') \cdot \ud \bm{r}'}{\theta_h(x',y')}
\]
In the above equation, $\phi_h$ is a piecewise linear interpolant and $\theta_h$ is a piecewise constant interpolant. Note that the path we choose is  along $x$ axis to compute the reconstruction at $(\iph,j)$ and along $y$ axis to compute the reconstruction at $(i,\jph)$. We will make use of one-dimensional piecewise linear interpolants $\phi_h$ and piecewise constant interpolant $\theta_h$ as in the 1-D case described before.
\subsection{Boundary conditions}
In a cell-vertex finite volume approach, in order to reconstruct the variables near the domain boundary using the scheme described in section~(4.2), we have to specify a maximum of two values outside the computational domain using the ghost cell approach. In the test cases considered, we use three types of boundary conditions; namely periodic, solid wall and transmissive boundary conditions. The treatment of each boundary condition in a vertex-centered finite-volume setting is discussed in the following subsections.

\subsubsection{Periodic boundary condition}
The values of conservative variables at the two ghost points are obtained by copying the values from the periodic locations.  For example, if the solution is periodic along the $x$ axis, the ghost cell values are given by
\begin{eqnarray*}
\con_{i,j} &=& \con_{\nx- i - 1,j}, \qquad i=-1,0, \qquad 1 \le j \le \ny \\
\con_{i,j} &=& \con_{i+1-\nx,j}, \qquad i = \nx+1,\nx+2, \qquad 1 \le j \le \ny
\end{eqnarray*}
Then the finite volume scheme~(\ref{eq:fvint}) is applied to points located on the periodic boundaries also.


\subsubsection{Solid wall boundary condition}
Boundary treatment for a solid wall involves the enforcement of a no-penetration condition i.e. normal component of velocity should be zero at the boundary.  Cells located on a solid wall have only half the size of interior cells. For example, if the bottom boundary is a solid wall, then the finite volume scheme is given by~(\ref{eq:fvwall}) and the flux on the wall is given by
\[
\ngl_{i,1} = [0, \ 0, \ p_{i,1}, \ 0]^\top
\]
The pressure on the wall is computed as
\[
p_{i,1} = (\gamma-1)\left( E_{i,1} - \half \rho_{i,1} u_{i,1}^2 \right)
\]
To compute the flux $\ngl_{i,\ohalf}$ we need to define the two states at this face. The $\bw$ variables at the wall point are given by
\[
\bw_{i,1} = \left[ \rho _{i,1} \ee^{-\psi_{i,1}},  \ u_{i,1} ,  \  0, \ p_{i,1} \ee^{-\psi_{i,1}} \right]^\top \qquad 1 \le i \le \nx
\]
where $\psi_{i,1}$ is defined with respect to the face $(i,\ohalf)$. The two states at this face are defined as
\[
\bw_{i,\ohalf}^L = \bw_{i,1}, \qquad \bw_{i,\ohalf}^R = \bw_{i,2} - \half \minmod( \theta(\bw_{i,2} - \bw_{i,1}), (\bw_{i,3}-\bw_{i,1})/2, \theta(\bw_{i,3} - \bw_{i,2}) )
\]
The reconstructed $\bw$ variables are converted to primitive variables which are used to compute the numerical flux.

\subsubsection{Transmissive boundary conditions}
The treatment for transmissive condition is done by zero-order extrapolation of $\bw$ variables from the domain boundary point to the centre of the two ghost cells.  For example, if the left boundary is a transmissive boundary, the values of $\bw$ variables are obtained as
\[
\bw_{i,j} = \bw_{1,j}, \qquad i=-1,0, \qquad 1 \le j \le \ny 
\]
If the bottom boundary is a transmissive boundary, then $\bw$ variables in the ghost cells are given by
\[
\bw_{i,j} = \bw_{i,1}, \qquad j=-1,0, \qquad 1 \le i \le \nx 
\]
Once the ghost values are obtained, the cells on the transmissive boundary are updated like interior cells.

\begin{theorem}
Consider the 2-D finite volume scheme together with the boundary conditions. Assume that the gravity is acting along the $y$ axis and the initial condition is a function of $y$ coordinate only. If the initial condition satisfies
\[
u_{i,j} = \const, \qquad v_{i,j} = 0, \qquad p_{i,j} \exp(-\psi_{i,j}) = \const
\]
then it is exactly preserved by the numerical scheme.
\end{theorem}
\begin{proof}
Since the solution is a function of $y$ coordinate only, the well-balanced property for an interior cell follows similarly to the one-dimensional case. We next consider the case of a cell on a solid wall as shown in figure~(\ref{fig:grid}b). Under the assumed conditions in the theorem, the $\bw$ variables are constant on the grid and hence any reconstruction scheme will be exact and yields the interface values
\[
(w_2)_{i,\ohalf}^L = (w_2)_{i,\ohalf}^R = \const, \qquad (w_3)_{i,\ohalf}^L = (w_3)_{i,\ohalf}^R = 0, \qquad (w_4)_{i,\ohalf}^L = (w_4)_{i,\ohalf}^R = p_{i,1} \exp(-\psi_{i,1})
\]
and hence
\[
u_{i,\ohalf}^L = u_{i,\ohalf}^R = \const, \qquad v_{i,\ohalf}^L = v_{i,\ohalf}^R = 0, \qquad p_{i,\ohalf}^L = p_{i,\ohalf}^R = p_{i,1} \exp(\psi_{i,\ohalf} - \psi_{i,1}) =: p_{i,\ohalf}
\]
The fluxes in the $y$ direction are $\ngl_{i,1}=[0, \ 0, \ p_{i,1}, \ 0]^\top$ and $\ngl_{i,\ohalf} = [0, \ 0, \ p_{i,\ohalf}, \ 0]^\top$. By definition the local hydrostatic pressure $\pe_{i,\ohalf}^L = p_{i,\ohalf}$ and hence the source term $\beta_{i,1}$ is exactly balanced by the flux derivative in the $y$ direction.
\end{proof}
\section{1-D numerical results}
\label{sec:res1d}
All the test cases in this section use a universal gas constant, $R$ = 1 and specific heat ratio, $\gamma$ = 1.4 unless stated otherwise. HLLC scheme~\cite{Toro1994} is used for computing the numerical fluxes. Time integration is performed using 3-stage strong stability preserving Runge-Kutta scheme~\cite{Shu1988}. In some of the tests, we compare the results with a non well-balanced (NWB) scheme where the source term is approximated as
\[
\frac{\partial \phi}{\partial x} (x_i) = \frac{\phi _{i+1} - \phi _ {i-1}}{2 \Delta x}
\]
Whenever we use the NWB scheme, we will reconstruct conserved variables to achieve high order accuracy.

\subsection{Isothermal hydrostatic solution}
To study the well-balanced property of the scheme for a system with an initial isothermal equilibrium condition, we use three types of gravitational potential functions defined as $\phi(x)$ = $x$, $\frac{1}{2}x^2$ and $\sin(2\pi x)$ . The initial density and pressure are given by
\[
\rho_{e} (x) = p_{e} (x) = \exp (- \phi (x))
\]
We interpolate the exact hydrostatic solution onto the grid and theorem~(\ref{thm:iso}) shows that this will be exactly preserved by the scheme. For grids with 100 and 1000 cells, the simulation is performed upto a final time of 2.0 and the errors in $L_1$ norm for the primitive variables are shown in table~(\ref{t:isothermalwb1}). We observe that, as the errors for all the primitive variables are of the order of machine precision for each of the gravitational potentials, the scheme maintains well-balanced property for isothermal hydrostatic solutions.

\begin{table}
\begin{center}
\begin{tabular}{|c|c|c|c|c|}
\hline
Potential & Cells & Density & Velocity & Pressure\\\hline
        \multirow{2}{*}{$x$} & 100 & 8.779E-15 & 7.031E-16 & 1.127E-14\\
                           & 1000& 9.126E-14 & 2.701E-15 & 1.193E-13 \\\hline
        \multirow{2}{*}{$\frac{1}{2}x^{2}$} & 100 & 1.160E-14 & 5.288E-16 &1.202E-14\\
                           & 1000& 1.143E-13&1.332E-15 &1.174E-13\\\hline
        \multirow{2}{*}{$\sin(2\pi x)$} & 100 & 1.213E-14 & 3.907E-16 &2.080E-14\\
                           & 1000& 1.162E-13&6.533E-15 &2.072E-13\\\hline
\end{tabular}
\caption{Errors in density, velocity and pressure for isothermal examples using different potentials}
\label{t:isothermalwb1}
\end{center}
\end{table}


We next compare the accuracy of the well-balanced scheme for simulating the evolution of small perturbations added to the initial isothermal hydrostatic solution, against a non well-balanced scheme. The potential is taken as $\phi(x) = x$ and the initial condition is given by
\[
\rho (x) = \exp  (-\phi (x)), \qquad p (x) = \exp (-\phi (x)) + \varepsilon \exp (-100(x - 1/2)^2)
\]
The simulation is performed on a domain of size [0,1] upto a final time of $t = 0.25$ units. We perform the tests for two values of the amplitude of pressure perturbation, $\varepsilon = 10^{-3}$ and $\varepsilon = 10^{-5}$. The results are shown in figures~(\ref{fig: isothermalpert}). It can be observed (figure  (\ref{fig: isothermalpert}a)) that for the larger perturbation case, the non well-balanced scheme provides results comparable to the well-balanced scheme even with coarse meshes. However, for the smaller perturbations, the non well-balanced scheme  is very inaccurate on the coarse mesh and requires a much finer mesh to give reasonably accurate solution, while the well-balanced scheme is able to resolve the features of the solution accurately even with the coarse mesh. Using the finer mesh ($\Delta x = 0.0005$), we also perform a simulation for the smaller perturbation with a compact fifth order WENO reconstruction where the convective fluxes are computed using the Roe scheme (henceforth called CRWENO5 scheme~\cite{Ghosh2016}). The resulting solution is compared with that obtained using the current scheme and illustrated in figure~(\ref{fig: isothermalpert}d). We observe that the solutions from the current scheme converge to the reference solution as the mesh is refined.
\begin{figure}
\begin{center}
\begin{tabular}{cc}
\includegraphics[width=0.48\textwidth]{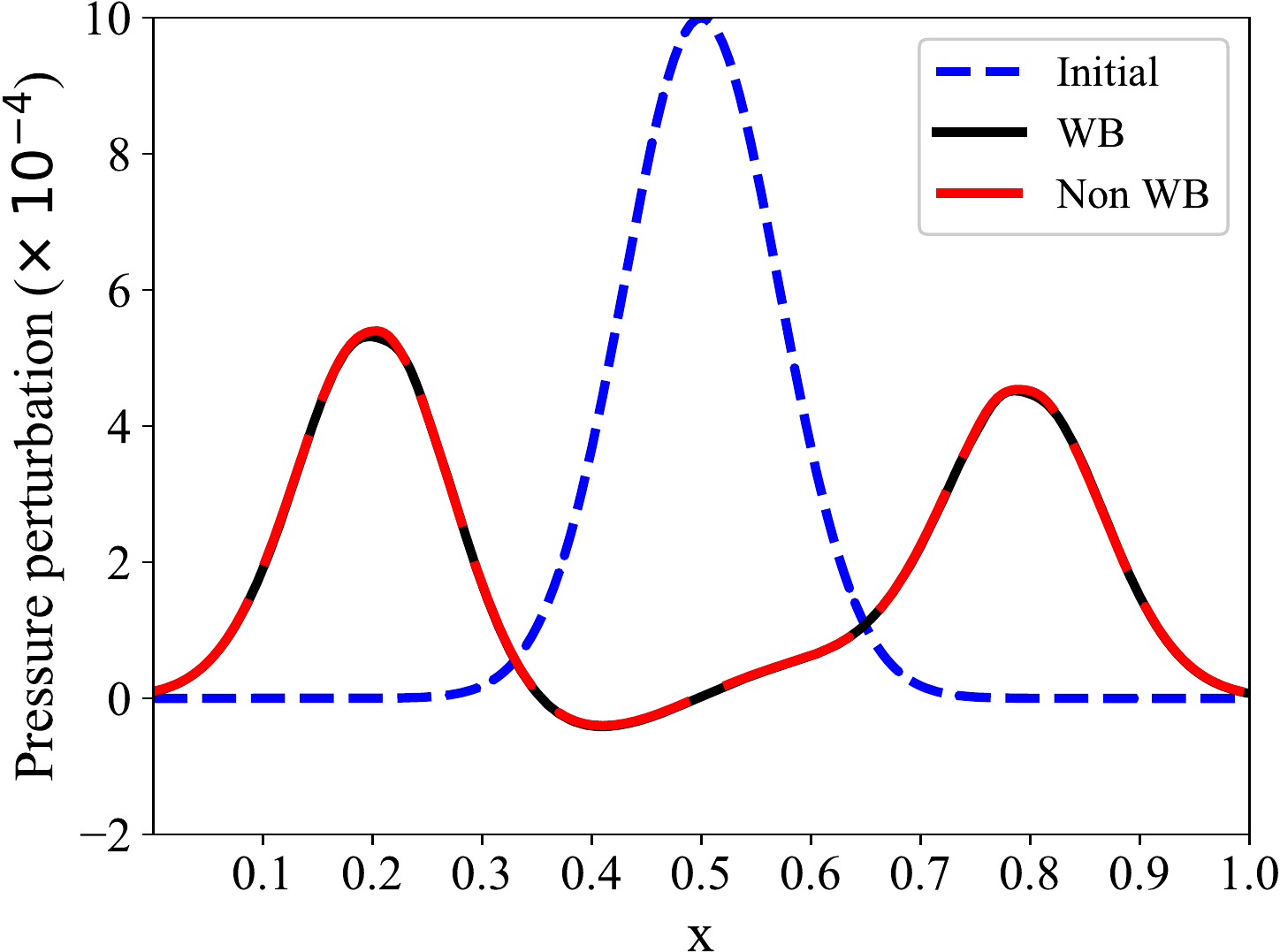} &
\includegraphics[width=0.48\textwidth]{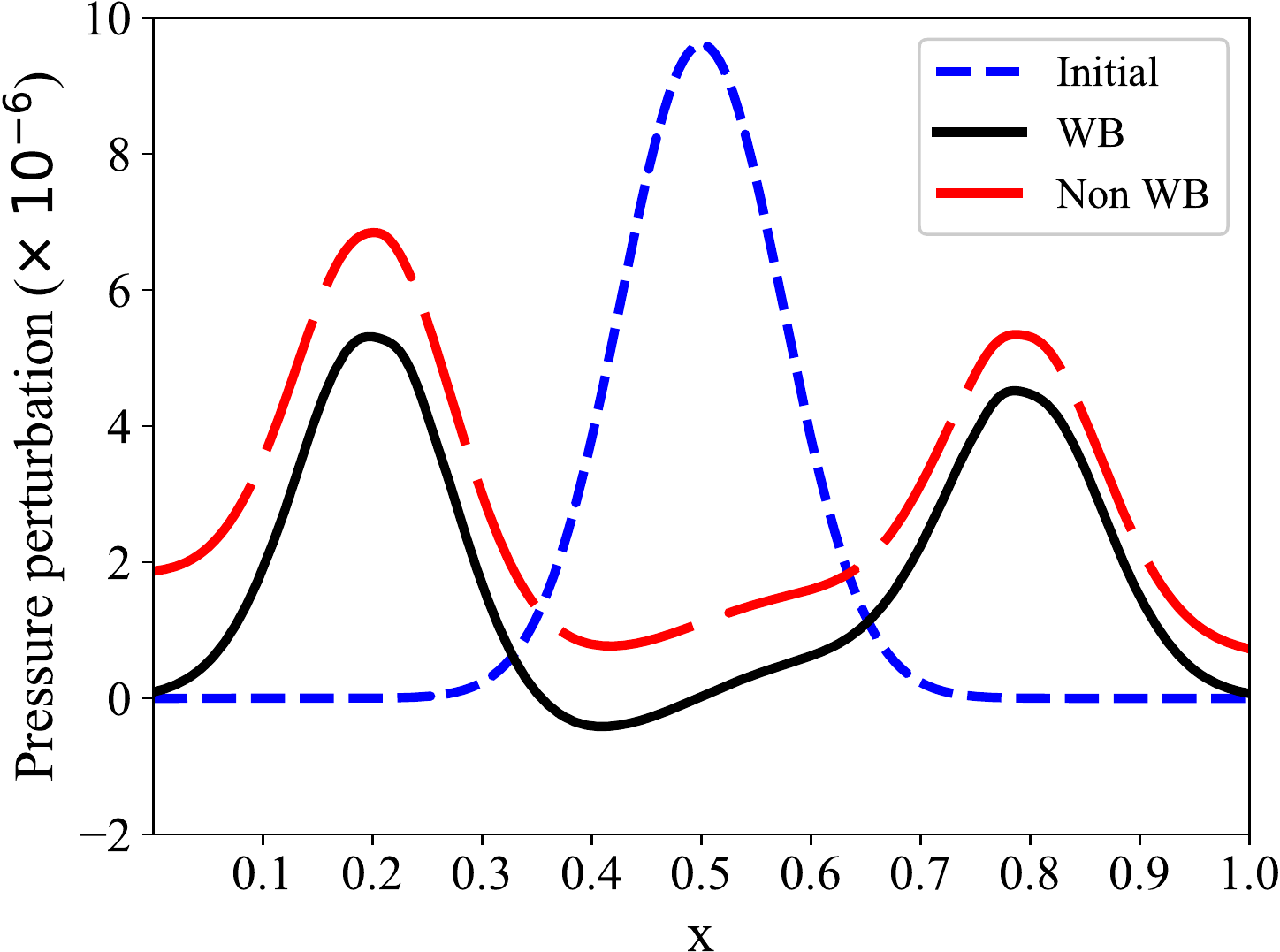} \\
(a)  & (b)  \\
\includegraphics[width=0.48\textwidth]{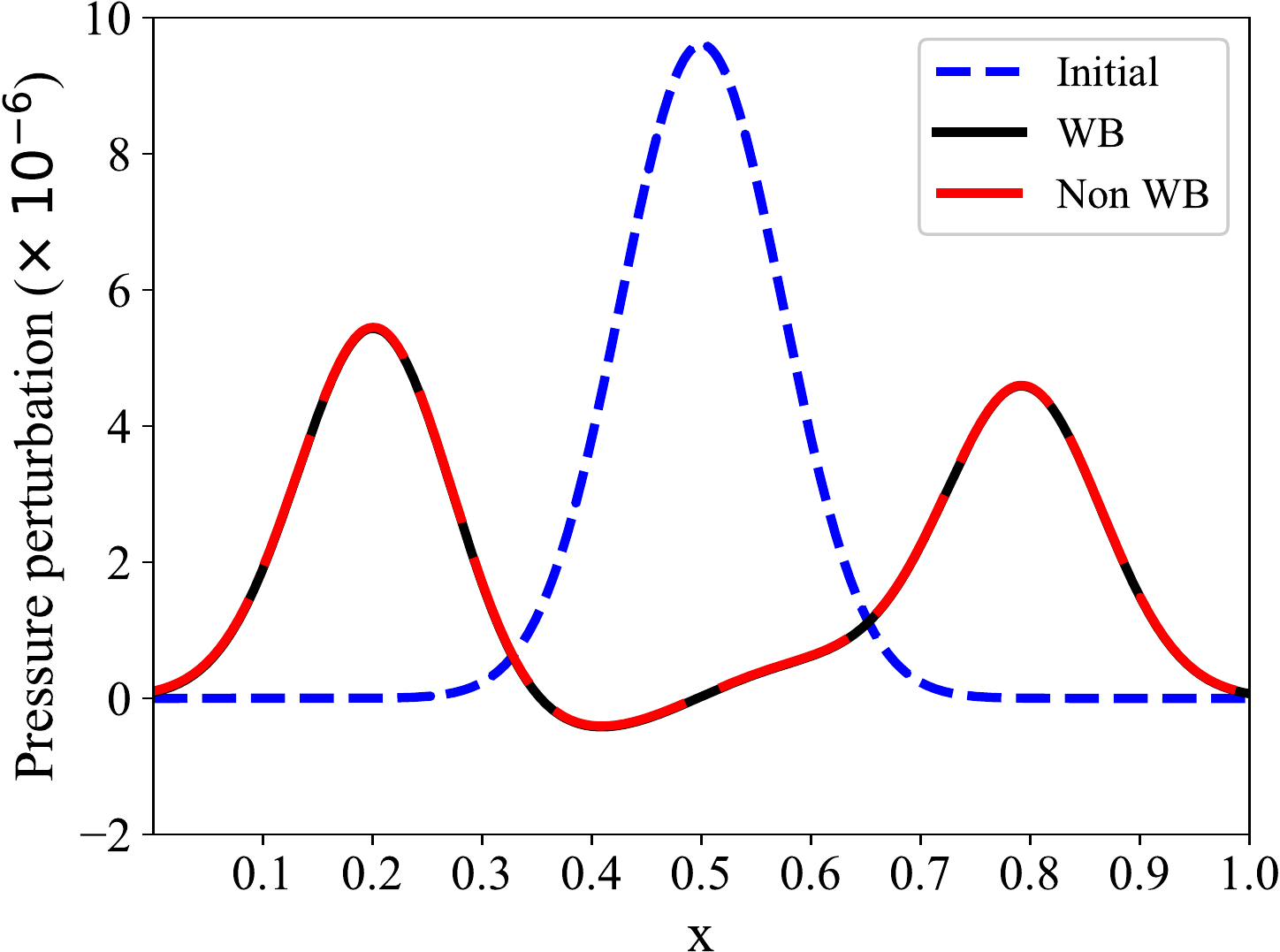} &
\includegraphics[width=0.48\textwidth]{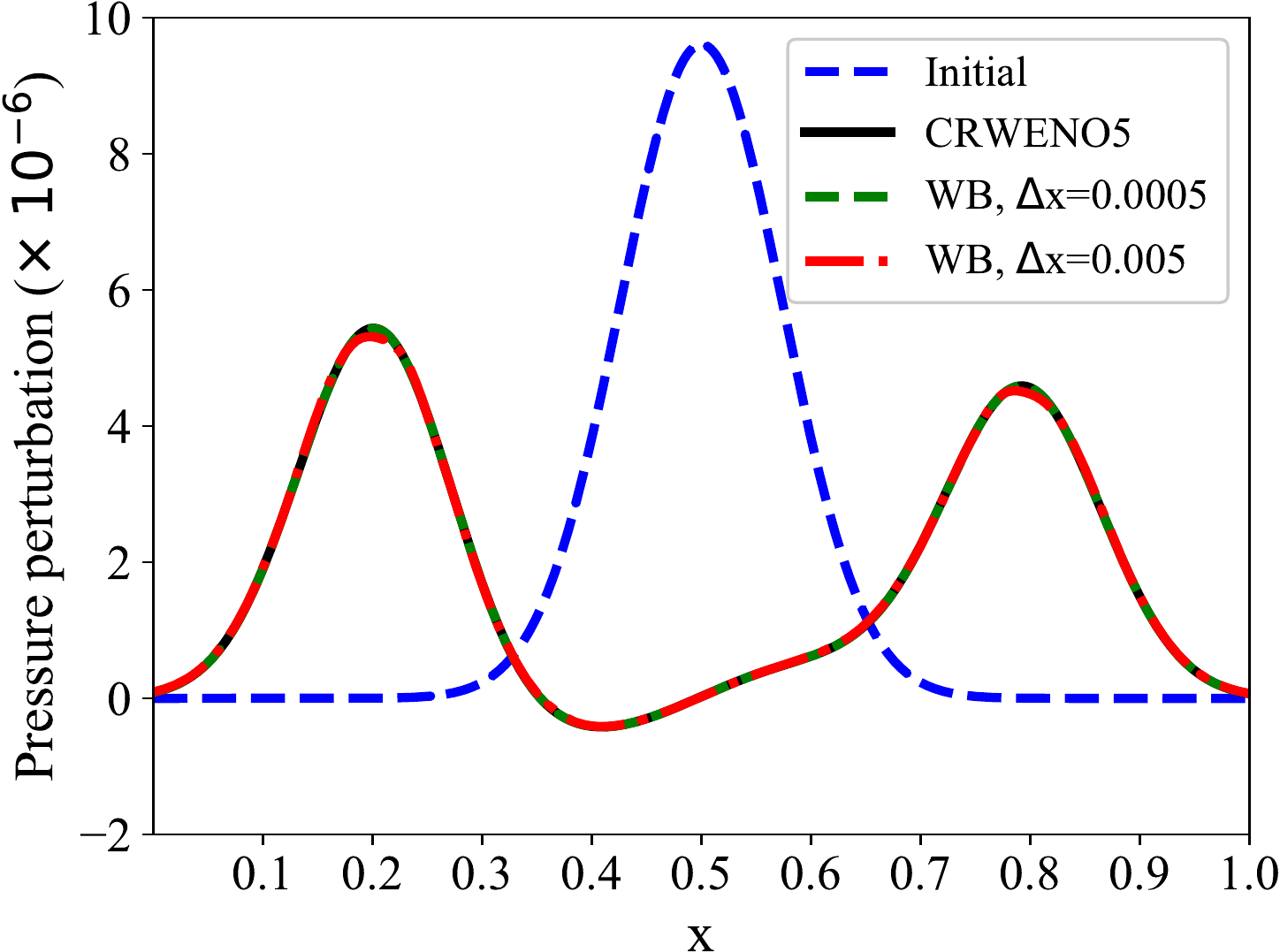} \\
(c)  & (d) \\
\end{tabular}
\caption{Evolution of pressure perturbations for isothermal case: (a) $\varepsilon$ = $10^{-3}$, $\Delta x = 0.005$, (b) $\varepsilon$ = $10^{-5}$, $\Delta x$ = 0.005, (c) $\varepsilon$ = $10^{-5}$, $\Delta x$ = 0.0005, (d) $\varepsilon$ = $10^{-5}$ }
\label{fig: isothermalpert}
\end{center}
\end{figure}
\subsection{Convergence to steady state}

In this test case which is taken from~\cite{Xing2013}, we model a gas falling into an environment of fixed external potential. Simulations are performed after adding a small pressure perturbation to an isothermal hydrostatic condition. The solution is expected to converge to the steady state after a long time due to numerical dissipation. The gravitational potential is in the form of a sine wave,
\[
\phi (x) = -\phi_{0} \frac{L}{2\pi} \sin\left(\frac{2\pi x}{L}\right)
\]
where the domain length, $L = 64$ units, and the amplitude of gravitational potential, $\phi_{0} = 0.02$. The initial condition is given by
\[
\rho(x) = \exp\left(\frac{-\phi (x)}{RT_0}\right), \quad u = 0, \quad p(x) = \rho(x) RT_0 + 10^{-3} \exp(-100(x - 32)^2)
\]
where the temperature, $T_0 = 0.6866$ and the specific heat ratio $\gamma = 5/3$. For a mesh of 64 cells with periodic boundary conditions, results are plotted and compared with non well-balanced scheme after 100 million time steps. From figure~(\ref{fig:steadyconverge}), it can be seen that the velocity distribution for the well-balanced scheme converges to the order of $10^{-13}$ which is close to machine precision, while the solution for the non well-balanced scheme seems to diverge away from the steady state. The other quantities also converge to their hydrostatic values in case of the well-balanced scheme while we do not see this convergence happening for the non well-balanced scheme.
\begin{figure}
\begin{center}
\begin{tabular}{cc}
\includegraphics[width=0.48\textwidth]{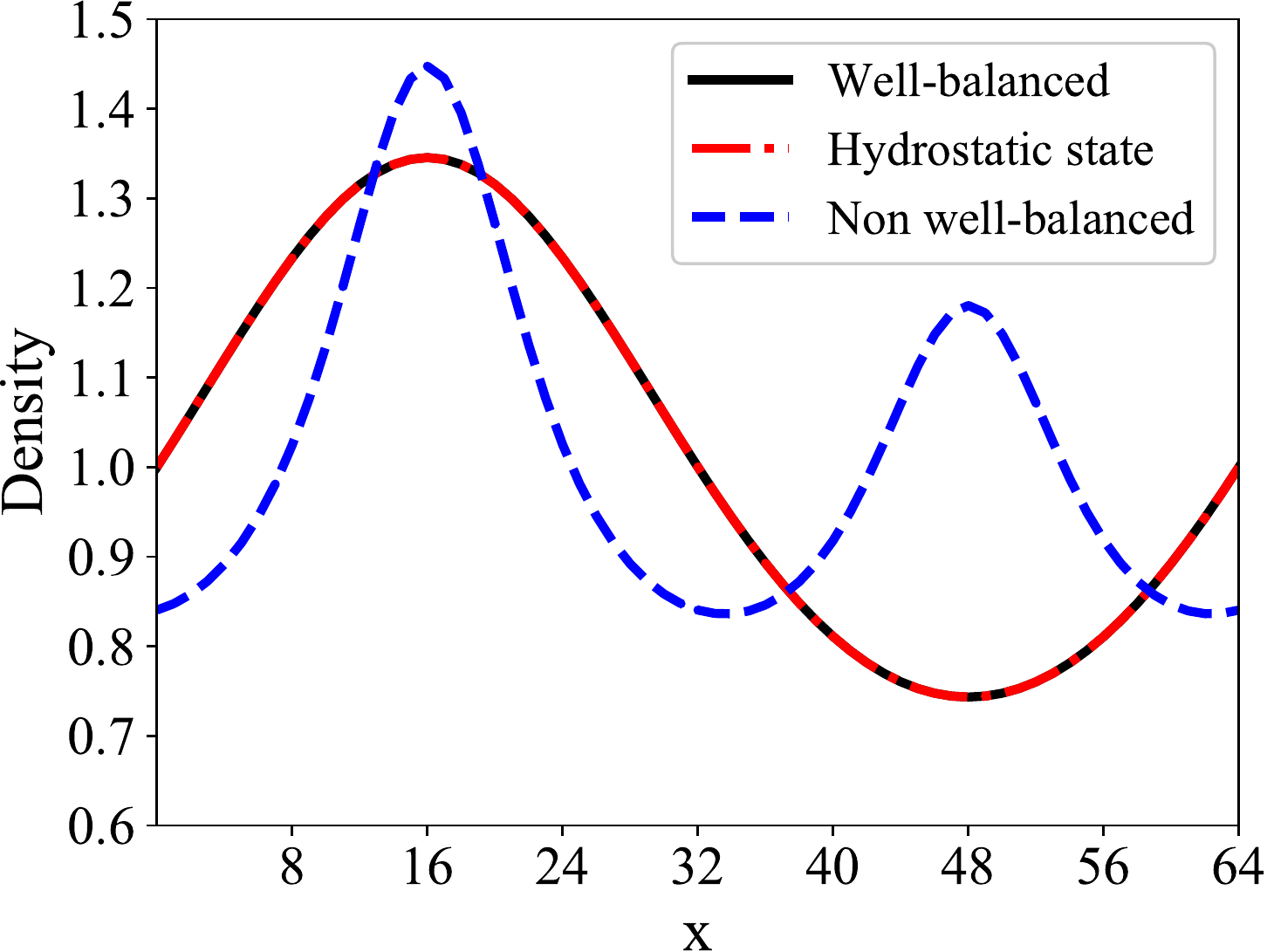} &
\includegraphics[width=0.48\textwidth]{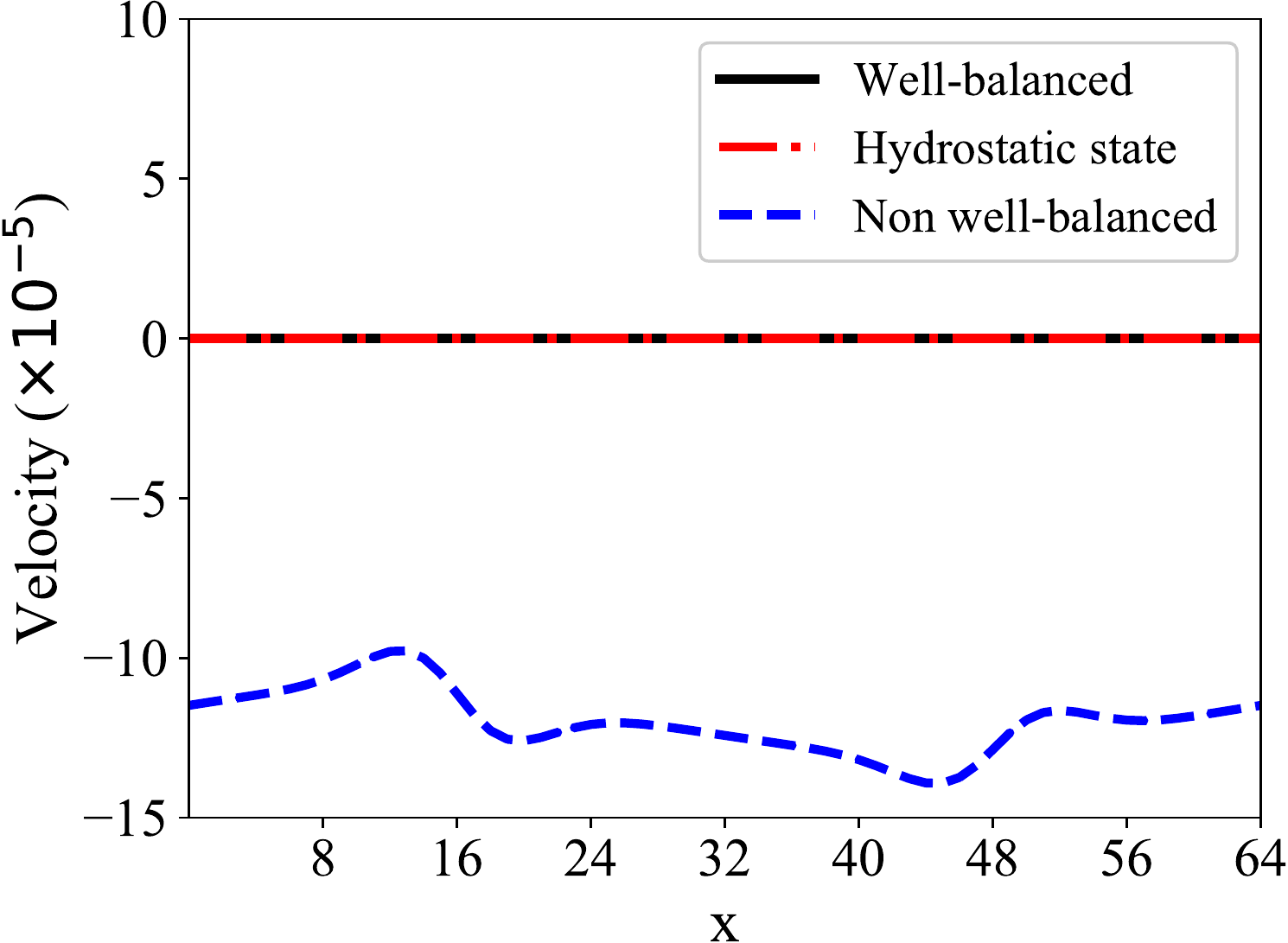} \\
(a)  & (b)  \\
\includegraphics[width=0.48\textwidth]{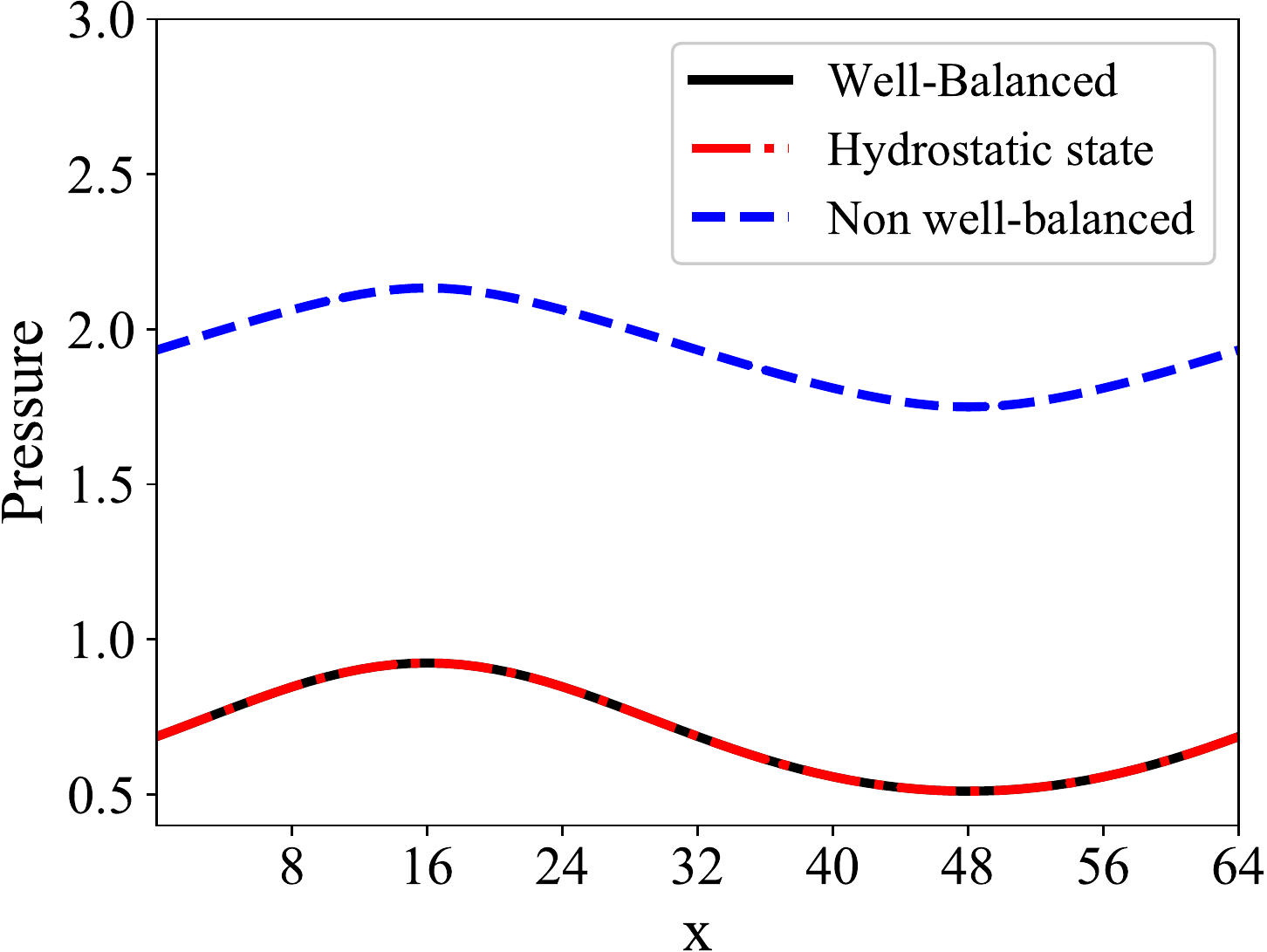} &
\includegraphics[width=0.48\textwidth]{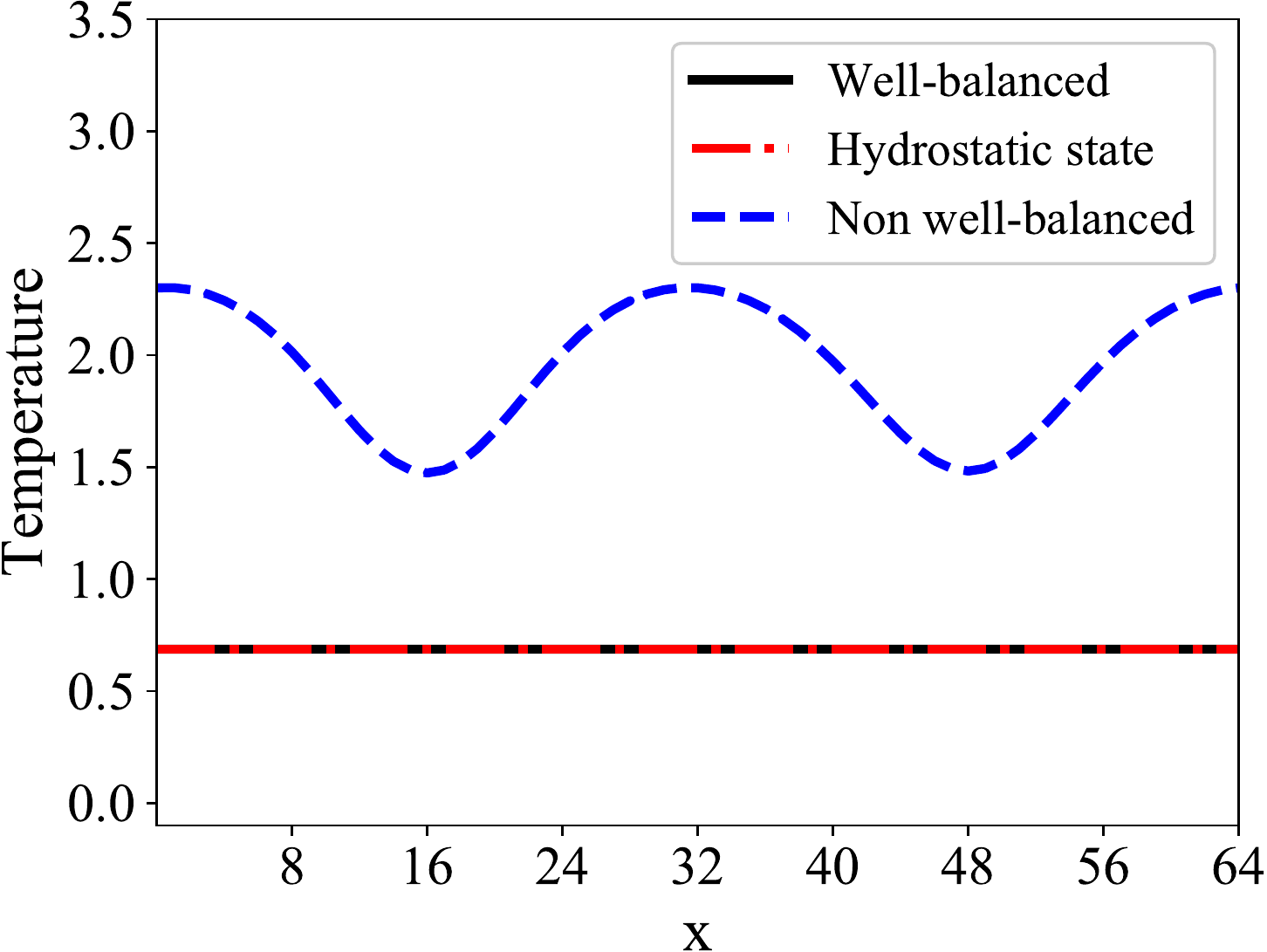} \\
(c)  & (d) \\
\end{tabular}
\caption{Comparison of convergence for well-balanced and non well-balanced scheme after 100 million time steps for : (a) density, (b) velocity, (c) pressure, (d) temperature}
\label{fig:steadyconverge}
\end{center}
\end{figure}
\subsection{Polytropic hydrostatic solution}
\label{sec:polyhydro}
In this test, we consider a polytropic hydrostatic solution given by
\begin{equation}
\Te (x) = 1 - \frac{\gamma - 1}{\gamma} \phi (x), \qquad \rhoe (x) = [\Te(x)]^{\frac{1}{\gamma - 1}}, \qquad \pe(x) = [\rhoe(x)]^\gamma
\label{e:polytropic}
\end{equation}
where the potential is $\phi(x) = x$ and this is actually an isentropic solution. The scheme that we have proposed is not well-balanced for the exact hydrostatic solution but only for the discrete hydrostatic solution. We first study the accuracy of the discrete hydrostatic solution by computing it on a sequence of meshes using the technique in section~(\ref{sec:dischydro}), and measuring the $L^2$ error norm relative to the exact hydrostatic solution. The results  shown in table~(\ref{t:polytropicconvergence}) indicate that the error in the discrete hydrostatic solution converges to zero at a rate of 2, showing the second order accuracy.
\begin{table}
 \begin{center}
 \begin{tabular}{|c|c|c|c|c|c|}
 \hline
Cells & $\rho$ error & $\rho$ rate & Velocity & $p$ error & $p$ rate\\\hline
        100 & 1.272E-6 & - & 0.0 &1.105E-6 & -  \\ \hline
        200 & 3.180E-7 & 1.9996 & 0.0 & 2.763E-7  & 1.9992 \\ \hline
         400 & 7.952E-8 & 1.9998 & 0.0 & 6.909E-8  & 1.9996\\ \hline
         800 & 1.988E-8 & 1.9999& 0.0 & 1.727E-8 & 1.9998\\ \hline
      1600 & 4.970E-9 & 1.9999 & 0.0 & 4.319E-9 & 1.9999\\ \hline
 \end{tabular}
\caption{Convergence of error in the discrete hydrostatic solution for a polytropic case}
 \label{t:polytropicconvergence}
 \end{center}
\end{table}
Next we consider three test cases to study the well-balanced property.
\begin{enumerate}
\item NWB-Exact: We use the non well-balanced scheme with exact hydrostatic solution as initial condition.
\item WB-Exact: We use the well-balanced scheme with exact hydrostatic solution as initial condition.
\item WB-Discrete: We use the well-balanced scheme with discrete hydrostatic solution as initial condition.
\end{enumerate}
The simulations are performed on the domain $[0,1]$ until a final time of 2.0 units using a grid of 100 cells and solid wall boundary conditions. At the final time, we measure the error norm relative to the initial condition which are shown in table~(\ref{t:polytropiclinear}). As expected, the NWB-Exact scheme is not able to keep the solution near the hydrostatic solution while the WB-Discrete is able to do so. \begin{table}                                                                                  
\begin{center}
\begin{tabular}{|c|c|c|c|c|}                                                                    
\hline
Scheme & Cells & Density & Velocity & Pressure\\\hline
        NWB-Exact  & \multirow{3}{*}{100}  & 9.372E-5 & 7.225E-5 & 1.009E-4\\
        WB-Exact             &                                  & 5.241E-9 & 5.338E-8 & 5.814E-9\\ 
        WB-Discrete                   &                                  & 6.743E-15 & 1.328E-16 & 7.874E-15\\ 
        \hline           
    NWB-Exact &\multirow{3}{*}{1000}  & 1.046E-6 & 8.579E-7  & 1.124E-6 \\
    WB-Exact           &                                    & 4.876E-11 & 5.407E-10  & 5.407E-11\\
    WB-Discrete                 &                                    & 6.579E-14 & 8.446E-16  & 7.738E-14\\        
        \hline          
\end{tabular}
\caption{Errors in density, velocity and pressure for polytropic examples using linear potential}     
\label{t:polytropiclinear}
\end{center}                                                                                    
\end{table}
In case of WB-Exact, we see that the errors are not as large as in the case of NWB-Exact scheme. We run this case for a long time of 750 units and plot the error norm as a function of time as shown in figure~(\ref{fig:polyconverge}). We can observe that while the velocity error tends towards machine epsilon ($10^{-16}$), the errors for density and pressure reaches steady state with a value $\approx$ $10^{-7}$, but importantly, they do not grow with time. The non well-balanced scheme is not able to keep the solution stationary as seen from the figure for velocity which takes large values and does not seem to reach stationary state even at large times.
\begin{figure}
\begin{center}
\begin{tabular}{ccc}
\includegraphics[width=0.32\linewidth]{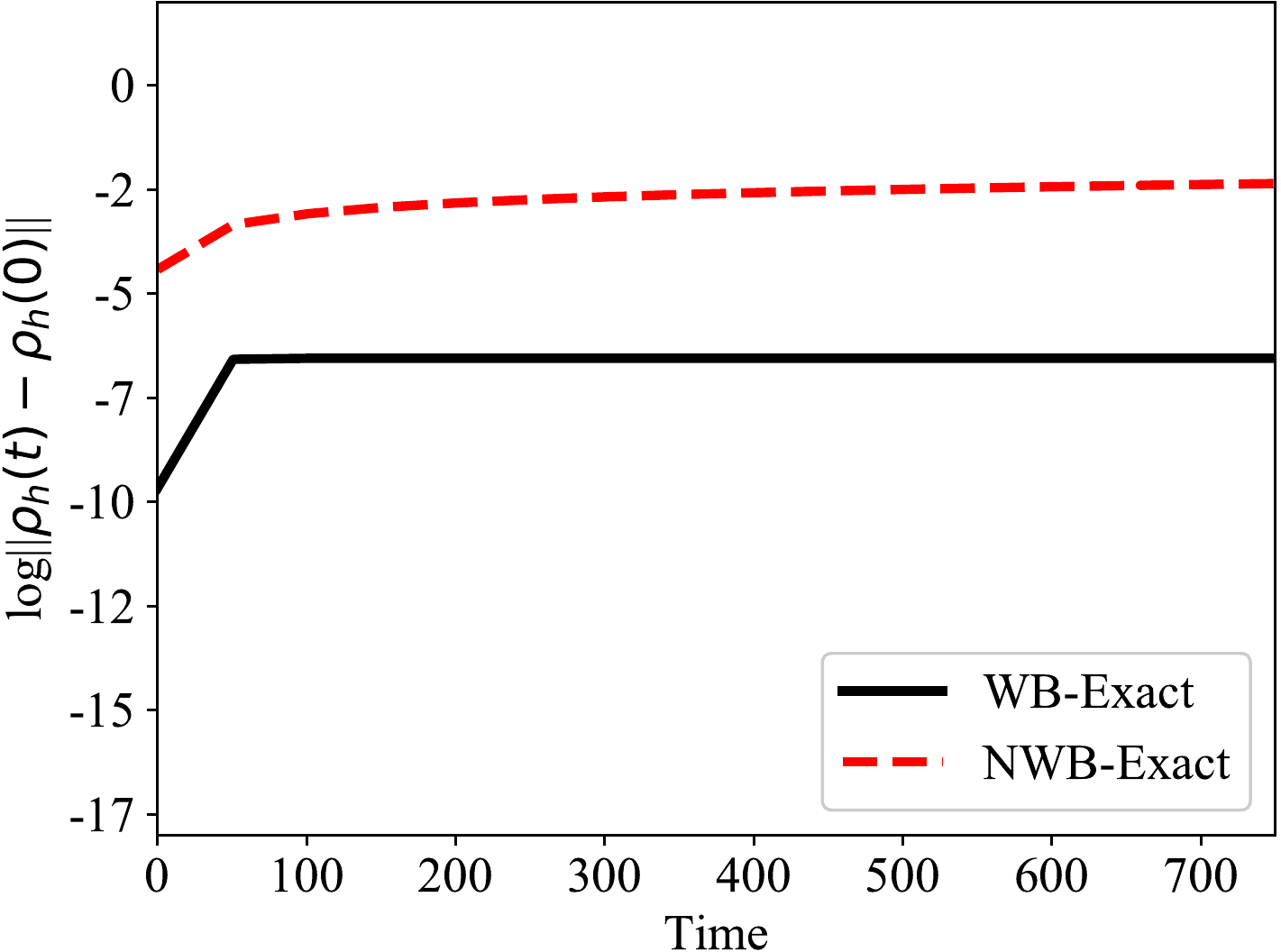} &
\includegraphics[width= 0.32\linewidth]{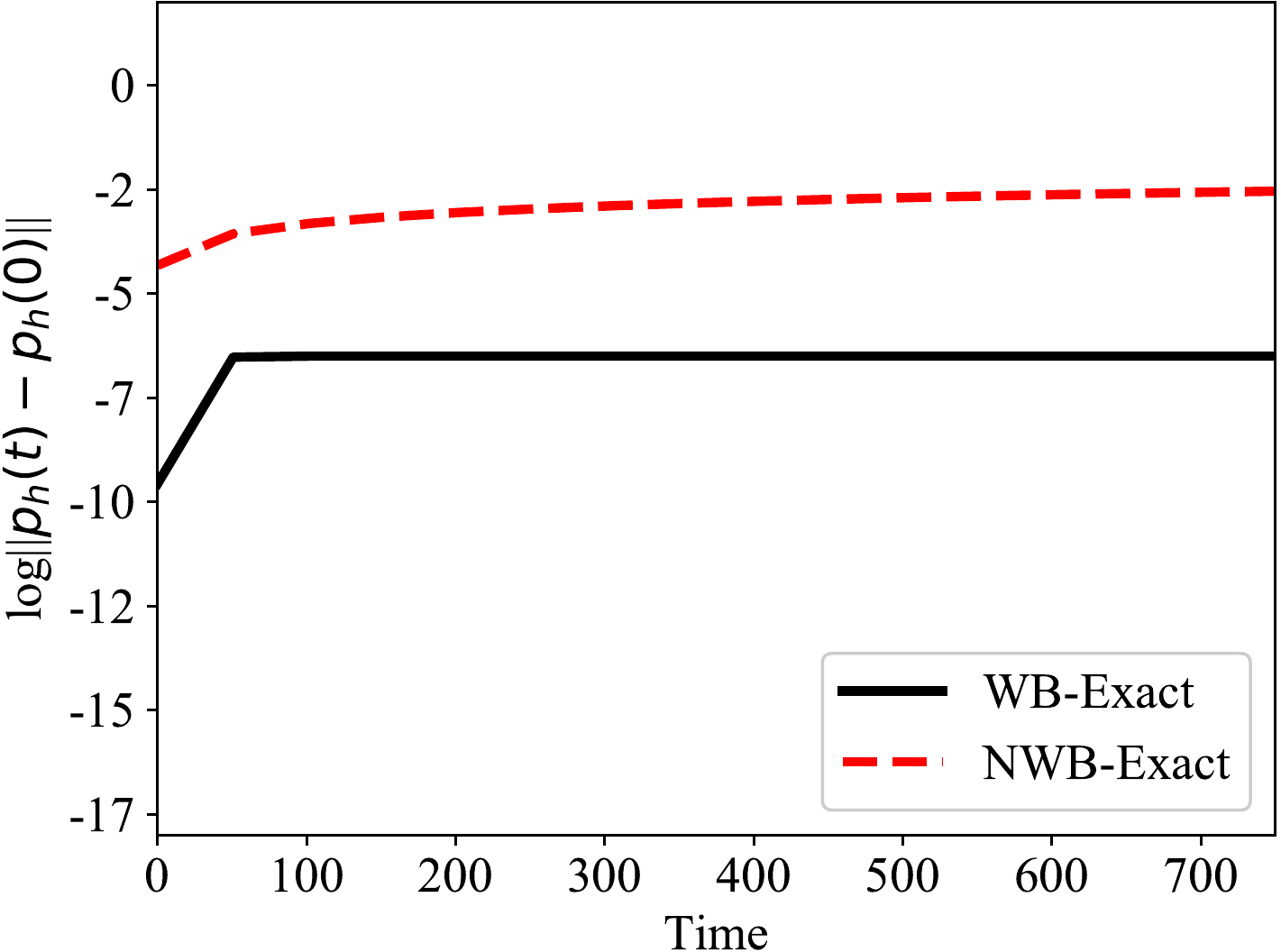} &
\includegraphics[width=0.32\linewidth]{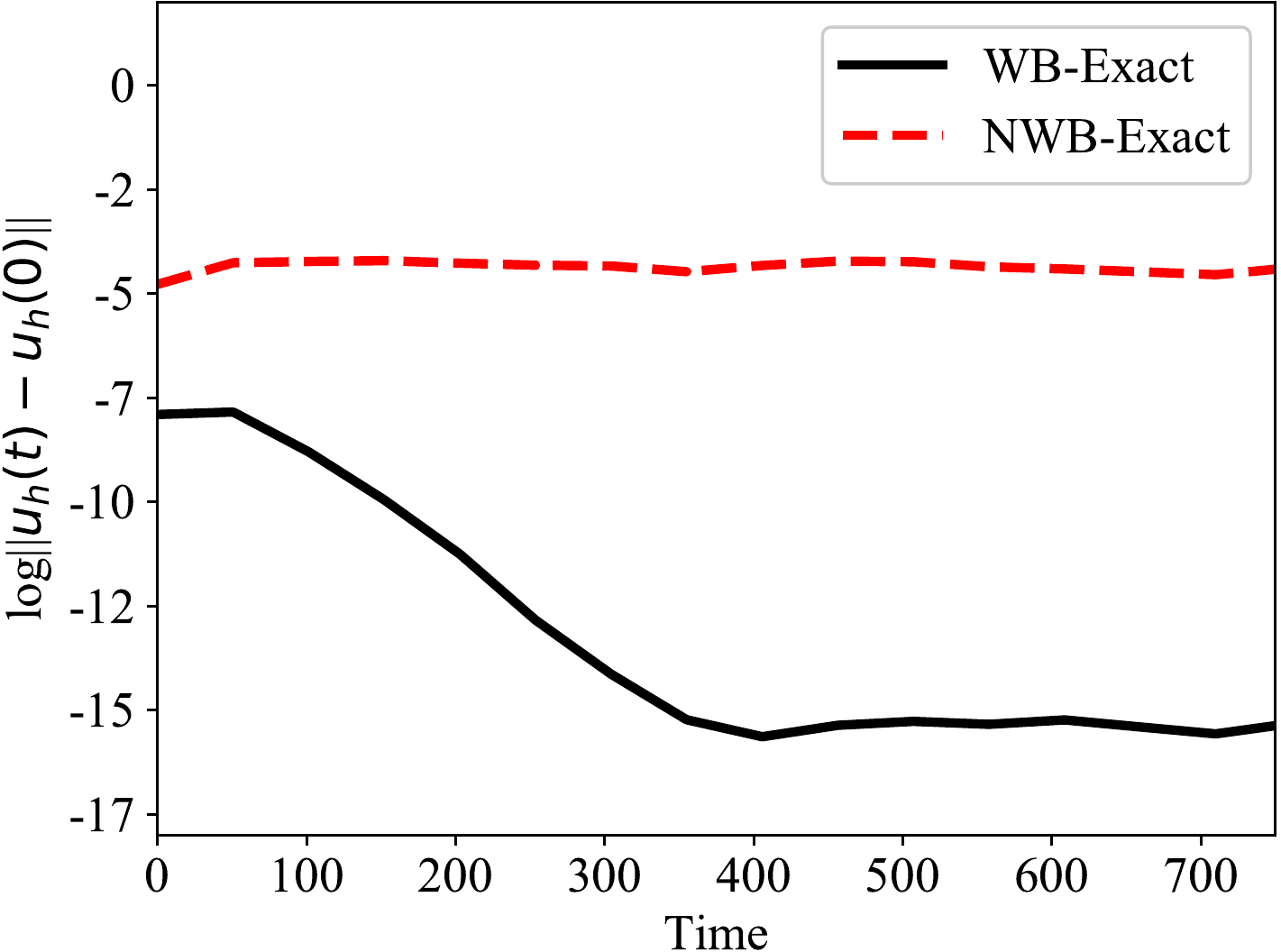}\\
(a)  & (b)  & (c) 
\end{tabular}
\caption{Stability of polytropic solution: (a) density (b) pressure (c) velocity.}
\label{fig:polyconverge}
\end{center}
\end{figure}
We perform the simulations corresponding to WB-Discrete for other types of potentials and the $L_1$ error norms are shown in table~(\ref{t:polytropictotal}). We see that the discrete hydrostatic solution is exactly preserved by our well-balanced scheme even for general potentials.
\begin{table}                                                                              
\begin{center}
\begin{tabular}{|c|c|c|c|c|}                                                                    
\hline
Potential & Cells & Density & Velocity & Pressure\\\hline
        \multirow{2}{*}{$\frac{1}{2}x^{2}$} & 100  & 1.063E-14 & 2.115E-16 &1.033E-14\\
                                                              & 1000& 1.056E-13&1.281E-15 &1.031E-13\\\hline
        \multirow{2}{*}{$\sin(2\pi x)$} & 100 & 1.282E-14 & 3.643E-16 &1.781E-14\\
                           & 1000& 1.224E-13&2.190E-15 &1.722E-13\\\hline
\end{tabular}
\caption{Errors in density, velocity and pressure for polytropic examples using different potentials}     
\label{t:polytropictotal}
\end{center}                                                                                    
\end{table}

We next study the performance of the scheme in computing small perturbations around the polytropic solution. To do this, we start the computations with the discrete hydrostatic solution but add a small perturbation to the discrete hydrostatic pressure, so that initial condition is given by
\[
p_i = \phe_i + 10^{-5} \exp (-100(x_i - 1/2)^{2})
\]
Since we do not have exact solution, we compare with the solutions obtained from another scheme~\cite{Chandrashekar2015} which is well-balanced for the exact polytropic hydrostatic solution. Simulations are performed upto a final time of $t = 0.25$ with a grid size of $\Delta x = 0.01$ and solid wall boundary conditions. The pressure perturbations at final time are compared for both schemes and is shown in figure~(\ref{fig:pressureperturbpoly}). We can see that the present scheme which is not well-balanced for exact polytropic solution is still able to give solutions which are as accurate as the exactly well-balanced scheme.  Keeping the grid size and other parameters constant, we also perform simulations by perturbing an exact hydrostatic solution. The result is shown in figure~(\ref{fig:pressureperturbpoly}b). It can be observed that even with an exact polytropic hydrostatic solution, the evolution of small perturbations is captured quite accurately by the scheme.
\begin{figure}
	\begin{center}
		\begin{tabular}{cc}
			\includegraphics[width=0.48\linewidth]{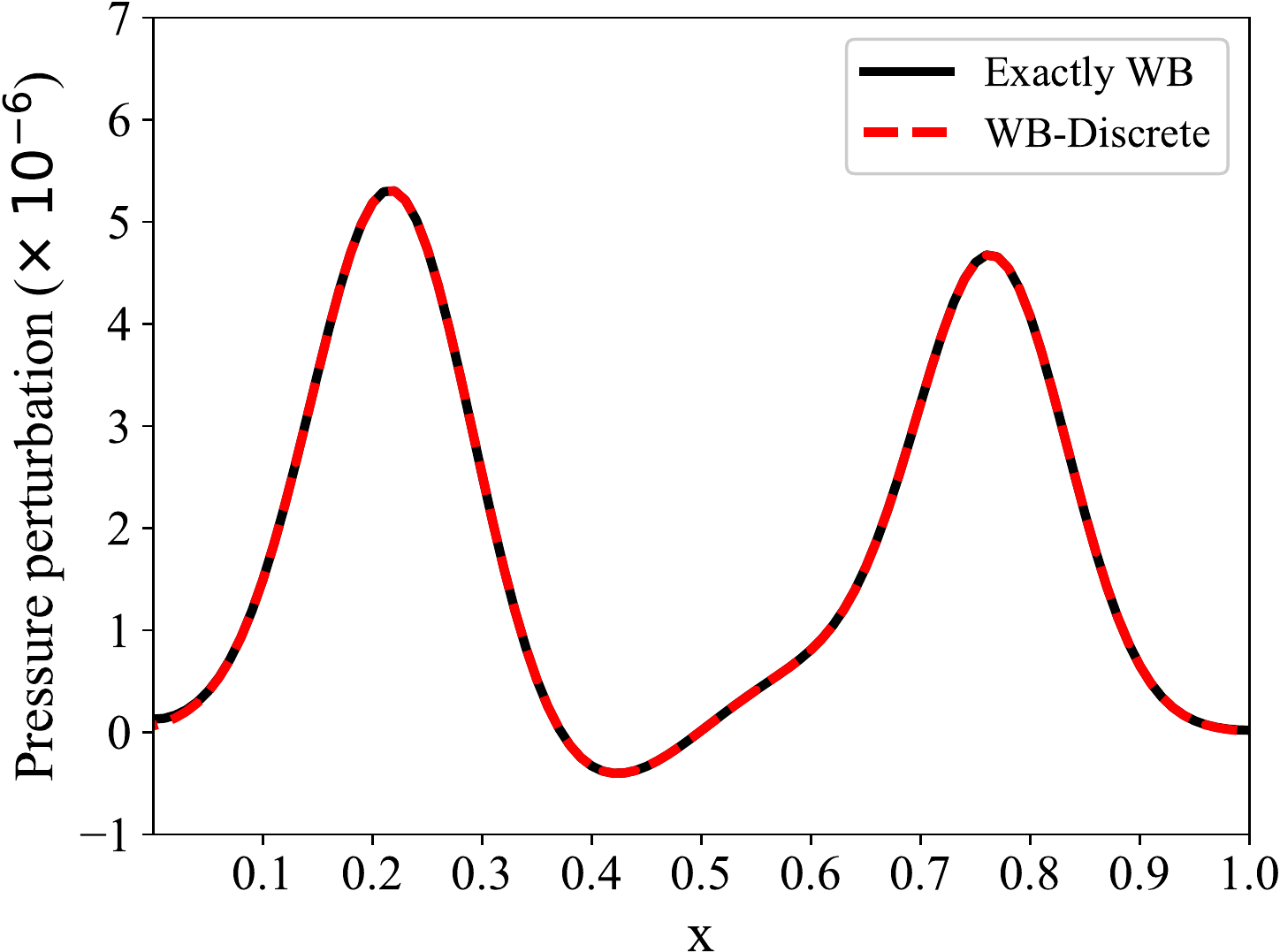} &
			\includegraphics[width= 0.48\linewidth]{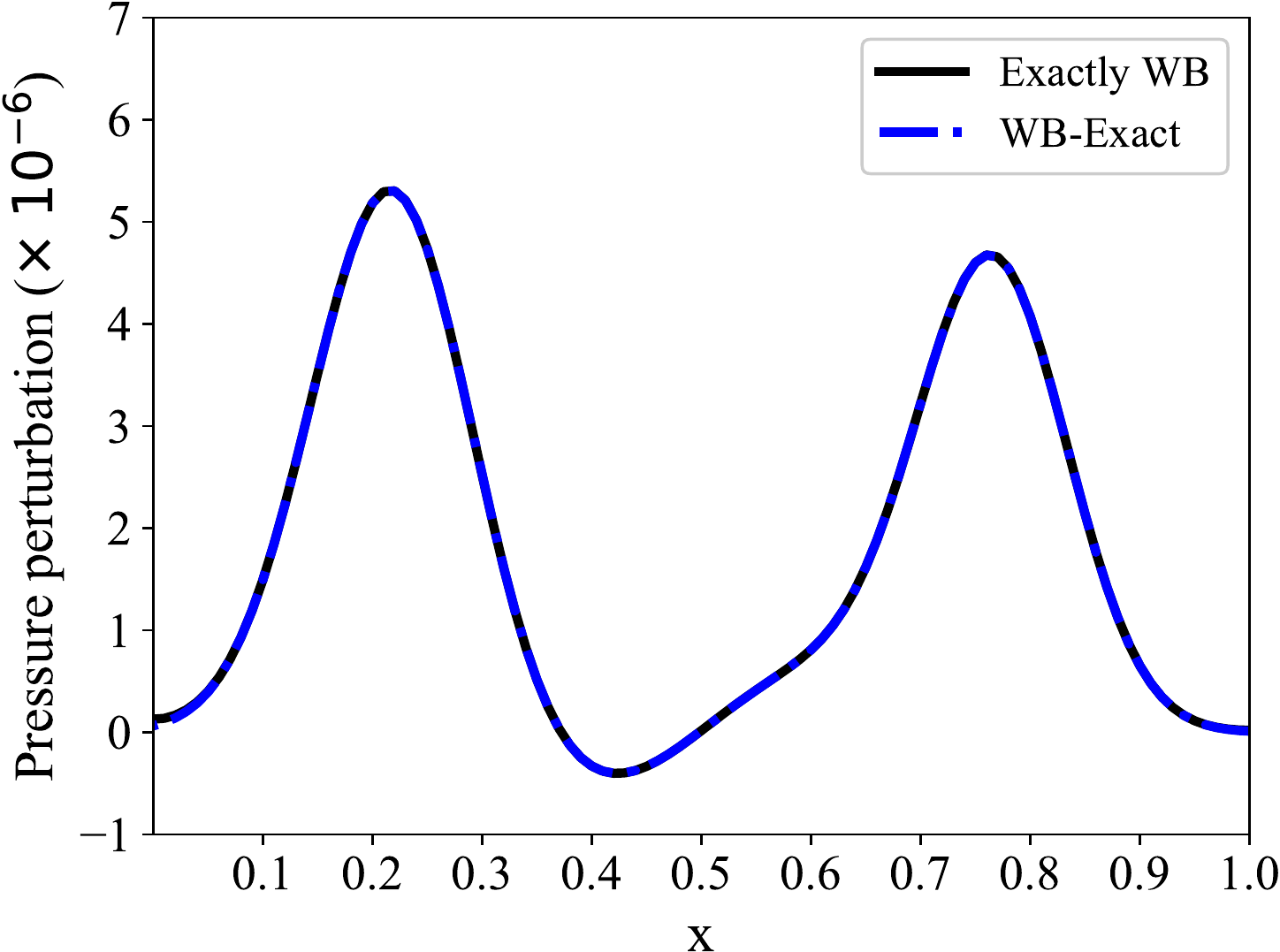} \\
			(a)  & (b) 
		\end{tabular}
		\caption{Comparison of pressure perturbation evolution with exactly well-balanced scheme using (a) discrete hydrostatic initial condition (b) exact hydrostatic initial condition.}
		\label{fig:pressureperturbpoly}
	\end{center}
\end{figure}

\subsection{Hydrostatic solution for van der Waals equation of state}
\label{sec:vdWhydro}
In this test case, we use a polytropic van der Waals gas as defined in section~(\ref{sec:vdWeos}). As an exact hydrostatic solution cannot be obtained analytically for the van der Waals equation of state, we obtain the solution, which we will refer to as the ``exact hydrostatic solution", by solving the following form of the hydrostatic equation using Chebfun.
\begin{equation}
\frac{d \rho}{d x} + \frac{1}{f(\rho)}\frac{d \phi}{d x} = 0; \qquad \text{where} \qquad f(\rho) = \frac{MR_uT}{\rho(M - \rho b)^2} - \frac{2a}{M^2}
\label{eqn:hydrovanode}
\end{equation}
In the above initial value problem, the domain is taken to be [0,1] and the value of $\rho$ at the left boundary is taken as 1.0. The universal gas constant and molar mass of the gas is also taken as 1.0. To introduce the effects of real gases, the van der Waals constants must be greater than 0; we take $a$ = 0.4 and $b$ = 0.001. Linear gravitational potential (i.e. $\phi(x) = x$) is used. The hydrostatic solution obtained from this system is plotted in figure (\ref{fig:idealvsvdW}) and compared with a solution for the same system under ideal gas conditions i.e. $a$ = $b$ = 0.
\begin{figure}
\begin{center}
\includegraphics[width=0.6\textwidth]{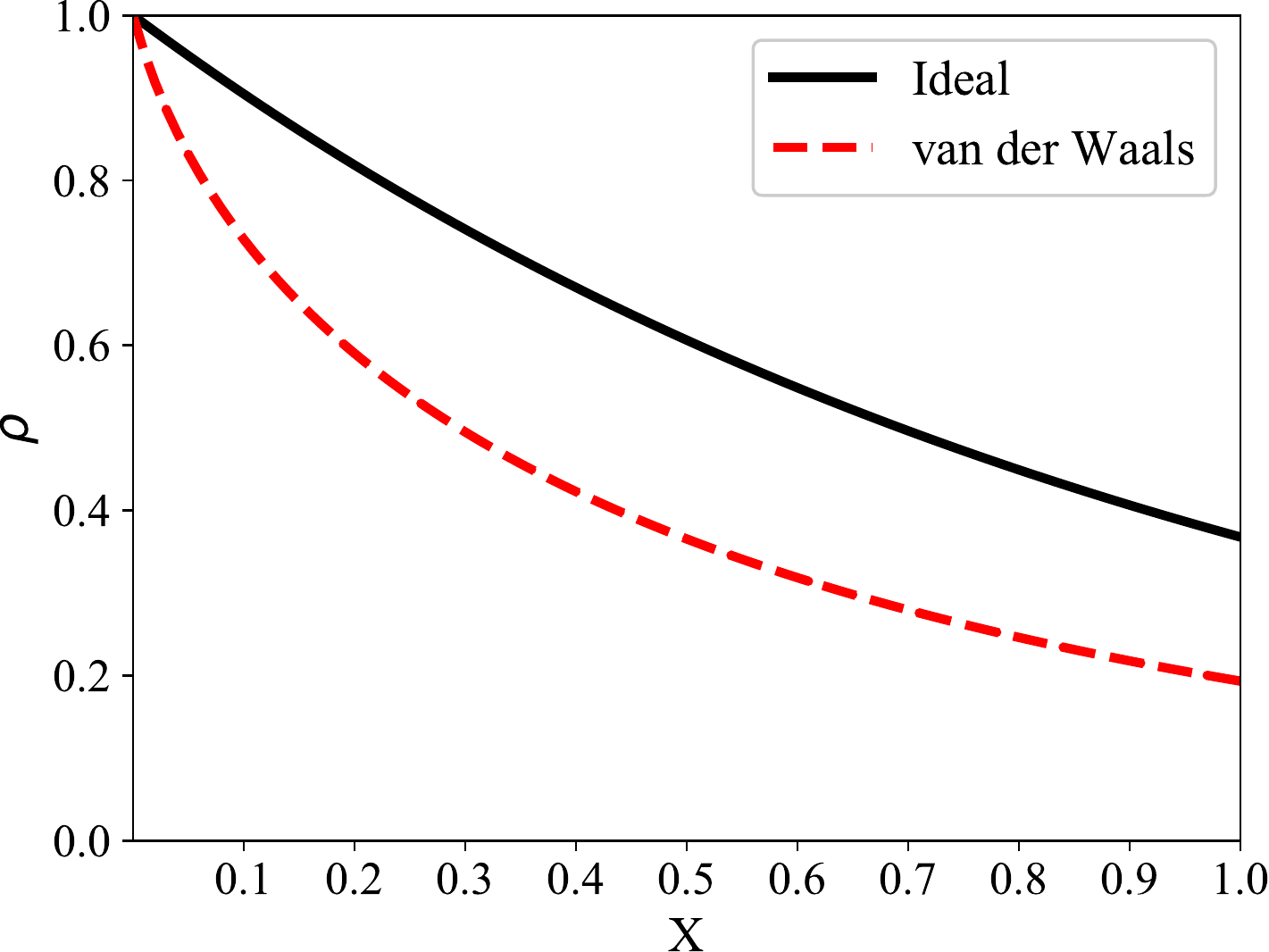}
\end{center}
\caption{Hydrostatic solutions obtained using equation (\ref{eqn:hydrovanode}) under van der Waals and ideal gas assumptions}
\label{fig:idealvsvdW}
\end{figure}
From theorem~(\ref{thm:iso}), we can infer that this exact hydrostatic solution would not be exactly preserved by the scheme. Therefore, we compute the discrete  hydrostatic solution using the technique shown in section~(\ref{sec:dischydro}). To calculate the order of accuracy for the discrete hydrostatic solution, we compute the $L_2$ error norms for density and pressure with respect to the exact solution on a sequence of meshes. From the results shown in table (\ref{tab:vdWordacc}), we can establish that the order of convergence for the discrete hydrostatic solution is two.\\
	\begin{table}
		\begin{center}
			\begin{tabular}{|c|c|c|c|c|}
				\hline
				Nodes & \multicolumn{2}{|c|}{$\rho$} & \multicolumn{2}{|c|}{$p$} \\ \hline
				& Error & Order & Error & Order \\ \hline
				101 &  2.08E-5 &  - &  1.10E-5 &  - \\ \hline
				201 &  5.22E-6 &  1.99 &  2.77E-6 &  1.99 \\ \hline
				401 &  1.31E-6 &  2.00 &  6.93E-7 &  2.00 \\ \hline
				801 &  3.27E-7 &  2.00 &  1.73E-7 &  2.00 \\ \hline
				1601 &  8.18E-8 &  2.00 &  4.34E-8 &  2.00 \\ \hline
			\end{tabular}
			\caption{Convergence of error  for discrete hydrostatic solution for van der Waals equation of state}
			\label{tab:vdWordacc}
		\end{center}
	\end{table}
	For testing the well-balanced property of the scheme, we take the three cases: 1) NWB-Exact, 2) WB-Exact and 3) WB-Discrete (definitions are same as in section~(\ref{sec:polyhydro})). For all the cases, simulations are performed by assuming transmissive boundary conditions up to a final time of 2.0. $L_1$ error norms are measured with respect to the hydrostatic initial condition at the final time and is shown in table (\ref{tab:vdWerrornorm}). We observe that, for the WB-Discrete case, the initial condition is exactly preserved by the well-balanced scheme as the errors are of the order of machine precision for all the primitive variables. As predicted, the scheme does not maintain well-balanced property for the exact hydrostatic solution. However, when a longer simulation with final time $t$ = 250 units was performed, we observe that velocity error norm tends towards machine precision while pressure and density error norms achieve steady state ($\approx 10^{-5}$) without further increase in values as shown in figure (\ref{fig:vdWexactlong}). We can also see that the solutions for the non well-balanced scheme becomes unstable at time $t \approx 19.1$ units.  Therefore, the exact hydrostatic state is preserved to a good accuracy by the well-balanced scheme and especially the static condition of zero velocity is well maintained.
	\begin{table}
		\begin{center}
			\begin{tabular}{|c|c|c|c|c|}
				\hline
				Scheme & Cells &$\rho$ & u & p \\ \hline
				NWB-Exact & \multirow{3}{*}{100} &  1.268E-1  & 4.964E-1  & 6.002E-2 \\ 
				WB-Exact & &  2.037E-5  & 1.454E-5  & 1.022E-5 \\ 
				WB-Discrete & & 1.299E-13  & 3.595E-13  & 2.193E-13 \\ \hline
				NWB-Exact & \multirow{3}{*}{1000} &  1.2173E-1  & 4.881E-1  & 5.601E-2 \\ 
				WB-Exact & &  1.921E-7 &   1.413E-7 & 9.453E-8 \\ 
				WB-Discrete & &  4.219E-13 &  4.949E-13  & 2.997E-13 \\ \hline
			\end{tabular}
			\caption{$L_1$ error norms for density, velocity and pressure for hydrostatic solution for van der Waals equation of state}
			\label{tab:vdWerrornorm}
		\end{center}
	\end{table}
	\begin{figure}
		\begin{tabular}{ccc}
			\includegraphics[width=0.33\textwidth]{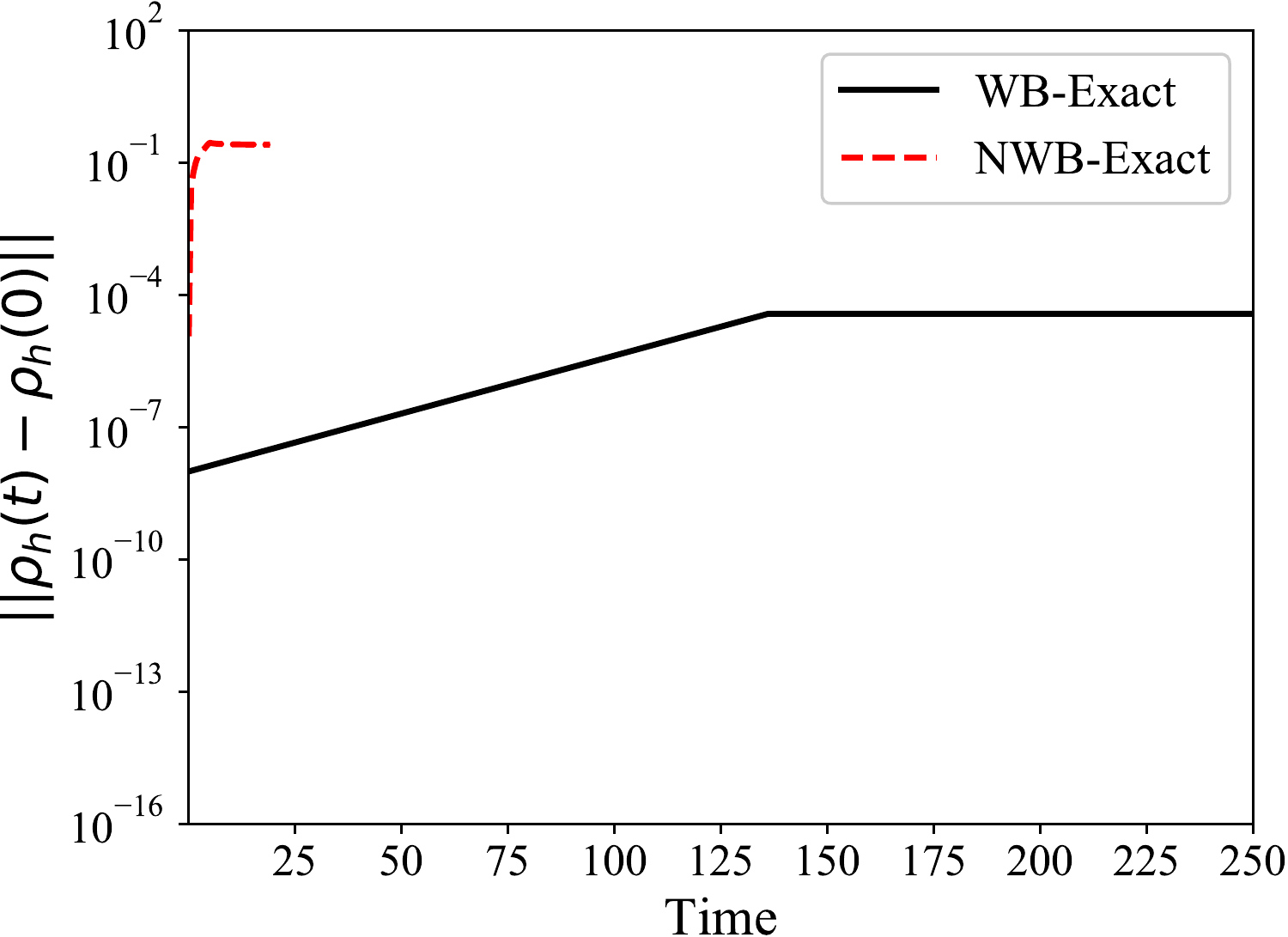} &
			\includegraphics[width=0.33\textwidth]{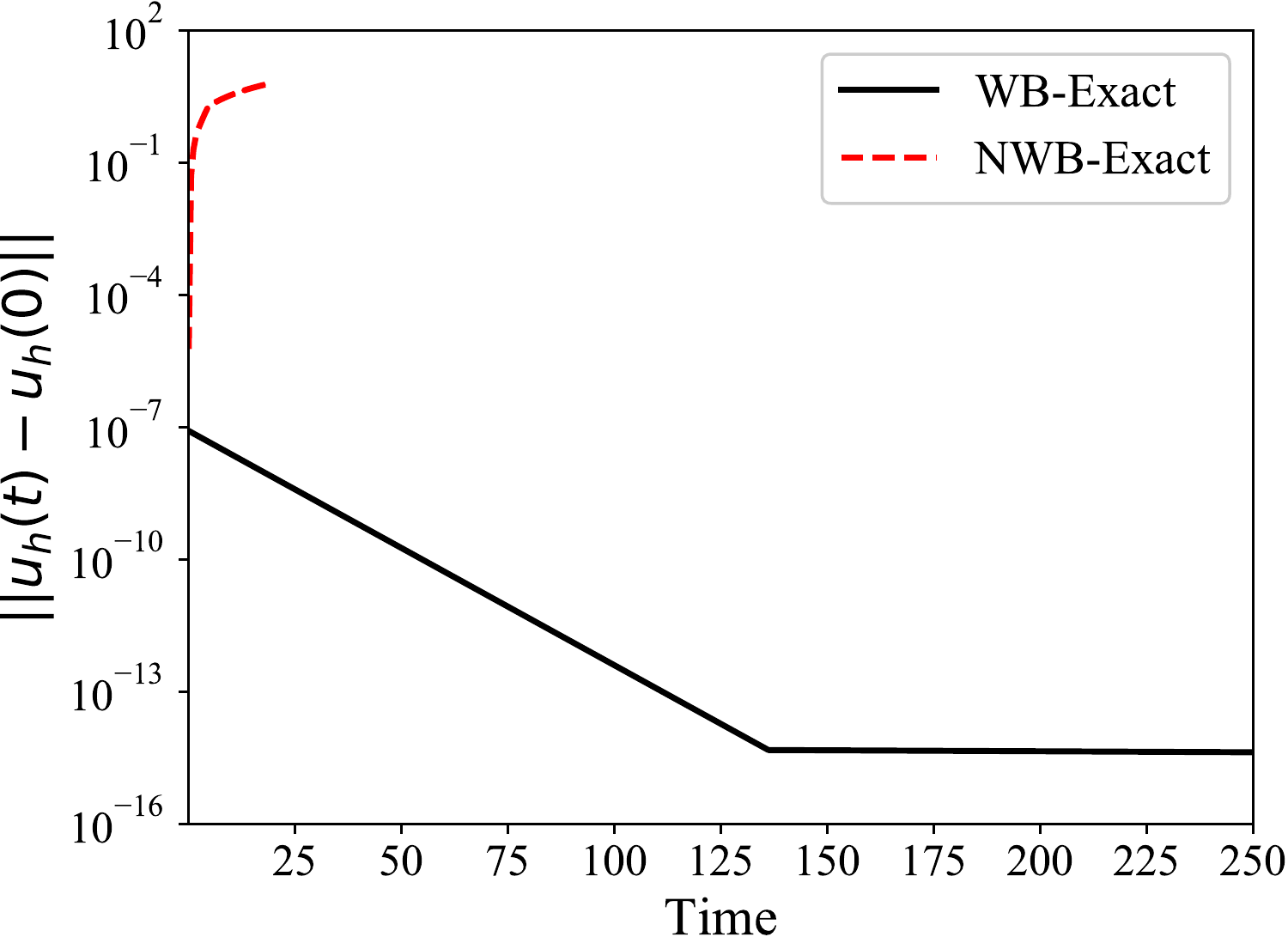} &
			\includegraphics[width=0.33\textwidth]{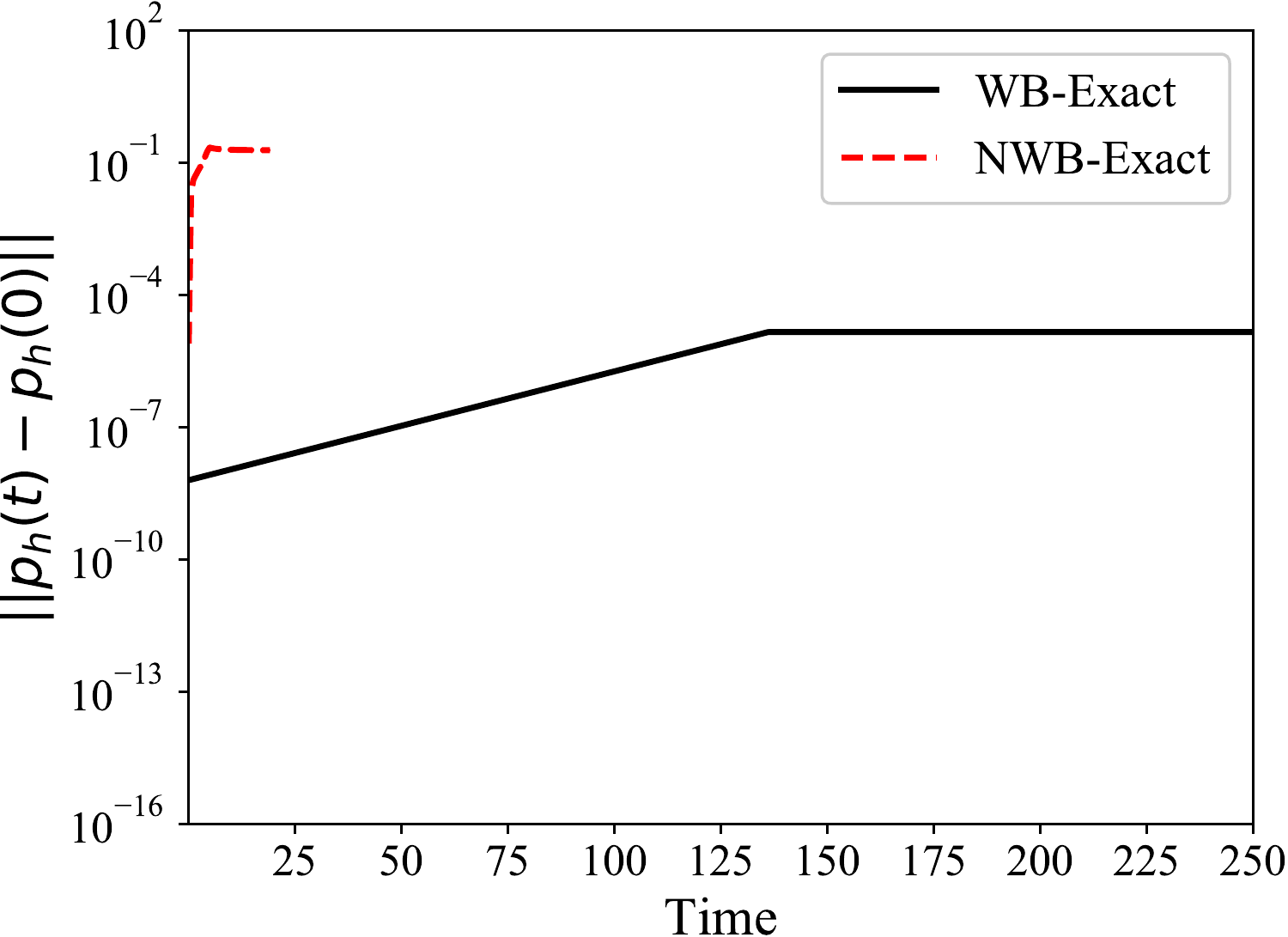} 
		\end{tabular}
		\caption{Stability of  exact hydrostatic solution for van der Waals equation : (a) density  (b) velocity (c) pressure}
		\label{fig:vdWexactlong}
	\end{figure}
\subsection{Perturbation from hydrostatic solution for van der Waals equation of state}
In this section, we test the ability of the well-balanced scheme to simulate the evolution of small perturbations from the discrete hydrostatic solution for van der Waals equation computed in section (\ref{sec:vdWhydro}). The perturbations are applied to the pressure and is given as 
\[
p_i = \widetilde{p_i} + \Delta p_i \qquad \text{where} \qquad \Delta p_i = \eta \exp(-100(x_i - 1/2)^2)
\] 
	Here, $\eta$ is equal to 0.1 or 0.001. The convergence of the solution is tested by performing simulations for two grid sizes of 100 and 1000 cells upto a final time of $t$ = 0.2 for the domain [0,1] using the well-balanced scheme. Transmissive boundary conditions are used for all the cases. Figure~(\ref{fig:perturbedvdW}a) shows that the solution has converged for the coarse mesh of 100 cells for a large perturbation of $\eta$ = 0.1. Now, we compare the solution with that obtained by using a non well-balanced scheme for the fine mesh as shown in figure ~(\ref{fig:perturbedvdW}b). It can be observed that even for such a large amplitude of perturbation and fine mesh, the non well-balanced scheme shows significant divergence from the converged solution. This divergence is also observed for the coarse mesh (not shown here). However, figure~(\ref{fig:perturbedvdW}c) shows that the well-balanced scheme provides a converged solution even for a small perturbation of $\eta$ = 0.001 and even on a coarse mesh. Therefore, it can be inferred that well-balancing is crucial for a real gas obeying the van der Waals equation of state and the proposed scheme is able to achieve this.\\
	\begin{figure}
		\begin{tabular}{ccc}
			\includegraphics[width=0.33\textwidth]{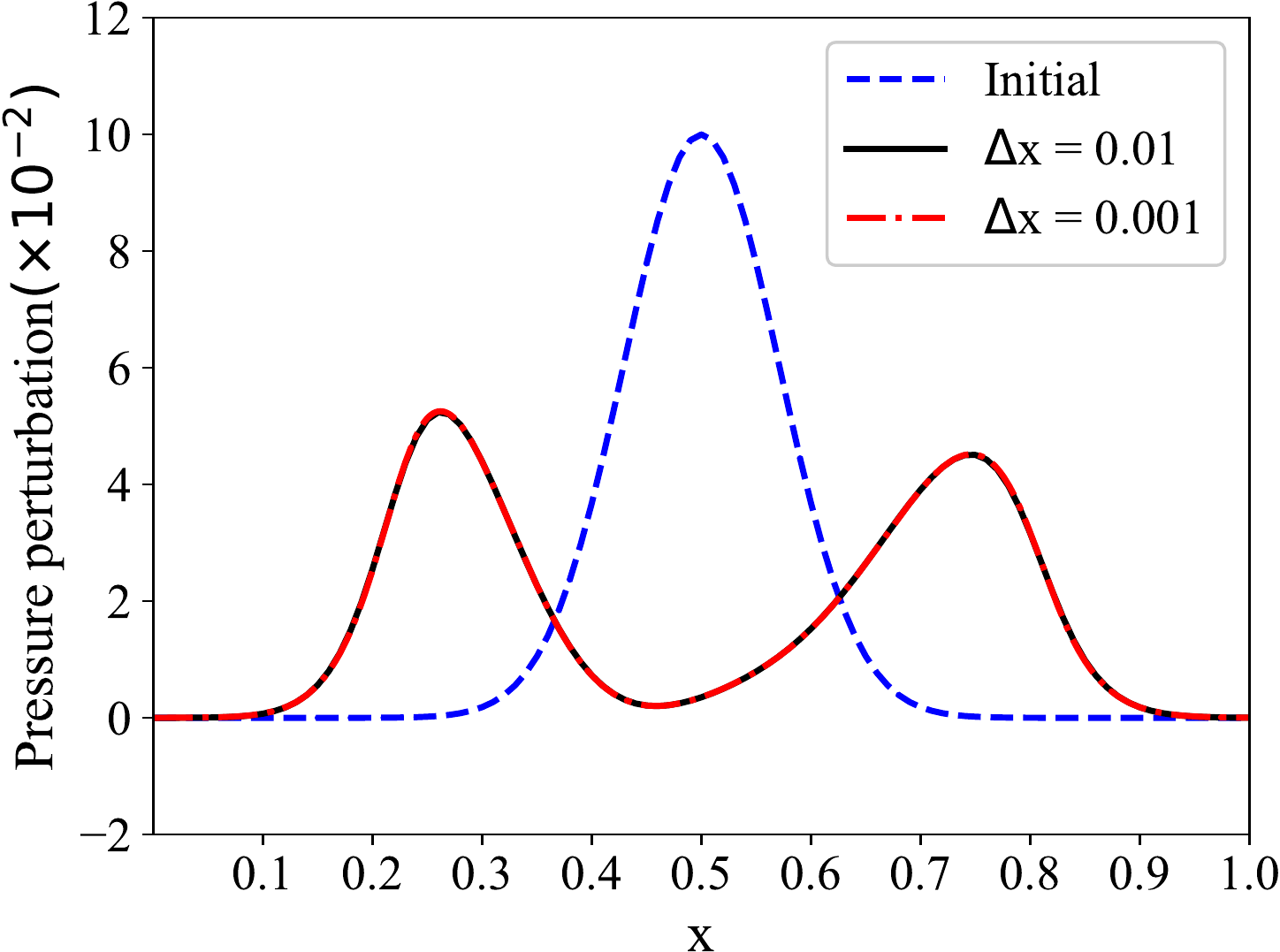} &
			\includegraphics[width=0.33\textwidth]{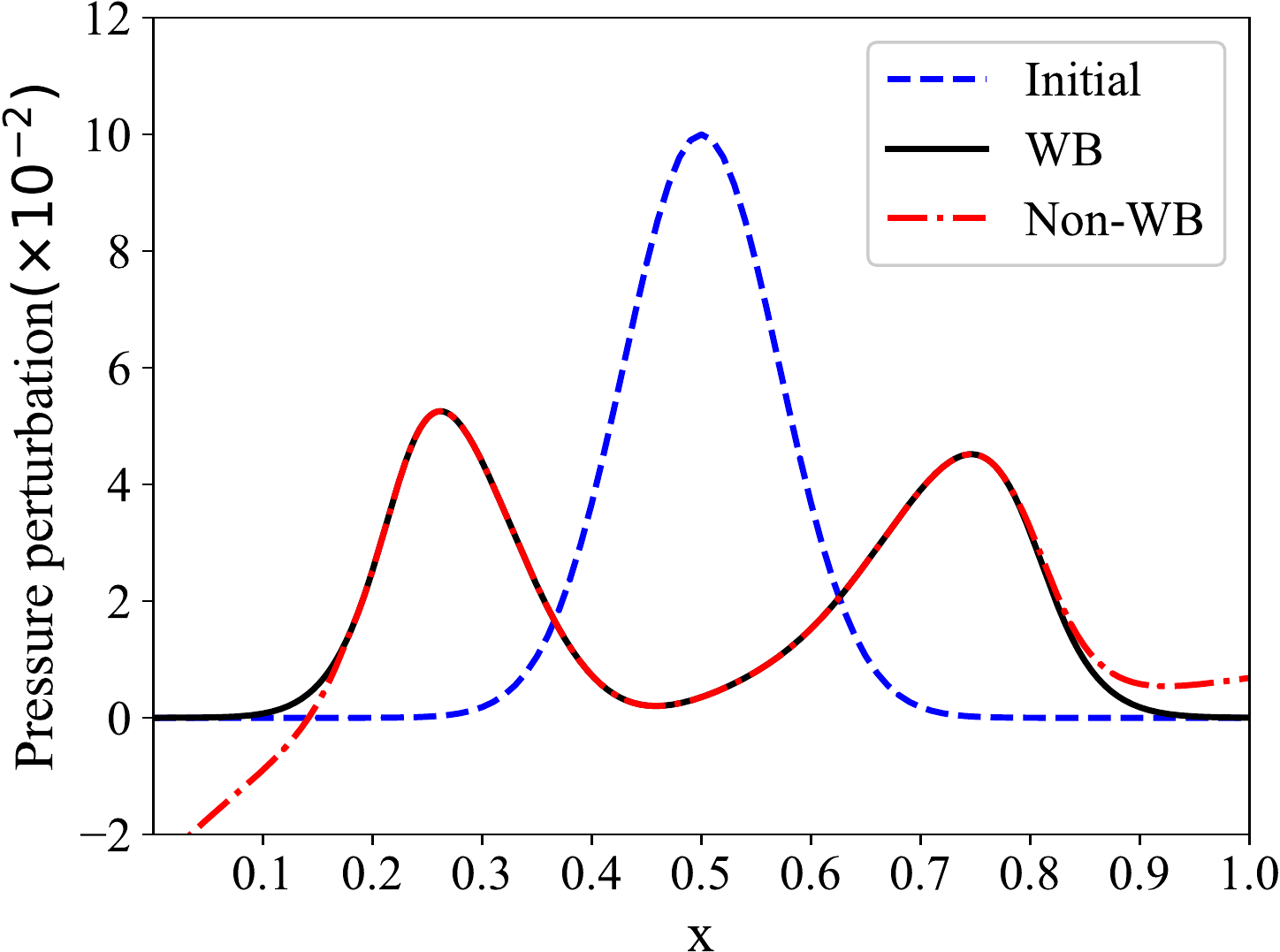} &
			\includegraphics[width=0.33\textwidth]{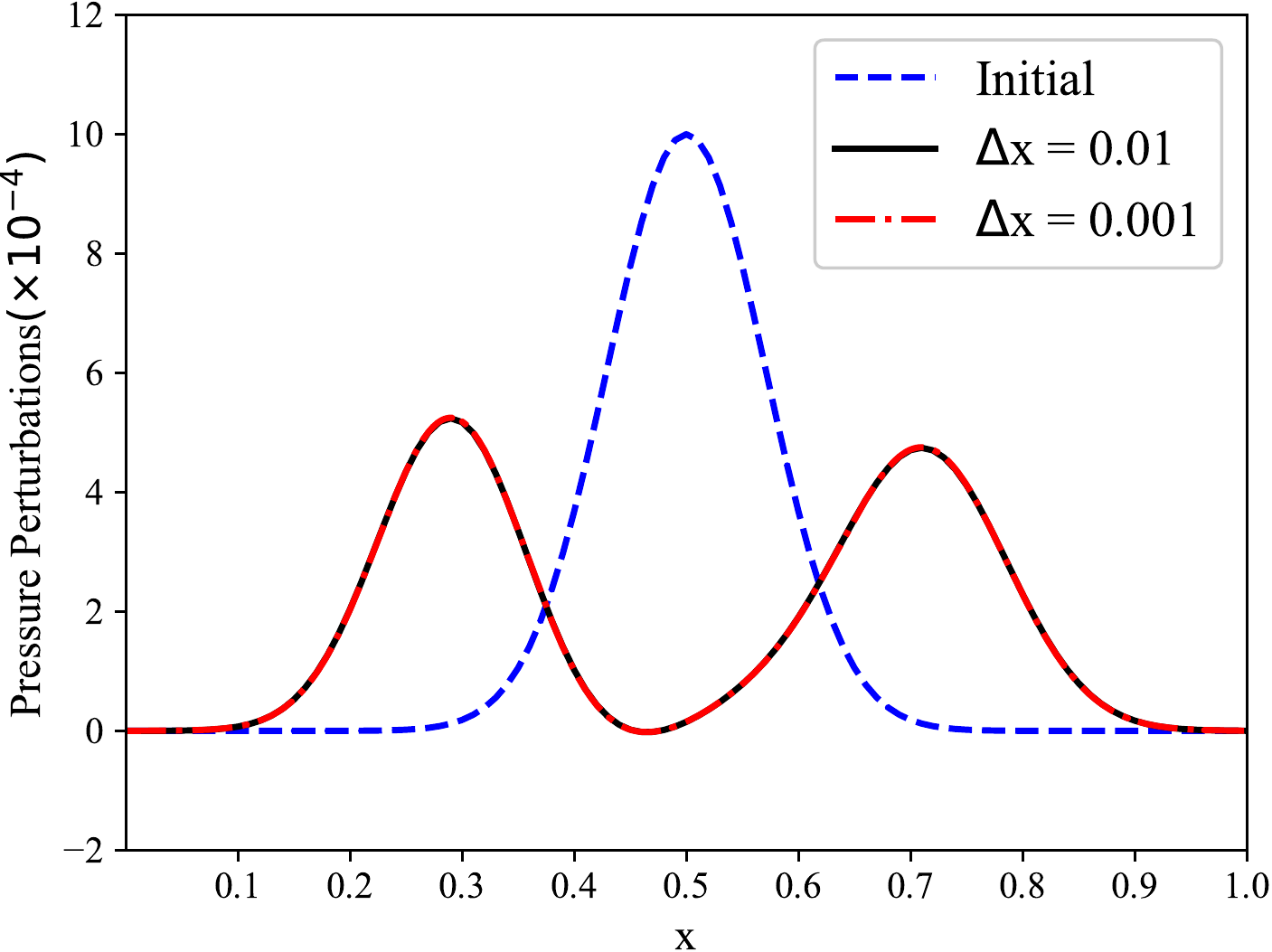}
		\end{tabular}
		\caption{Pressure perturbations from van der Waals discrete hydrostatic state : (a) $\eta$ = 0.1 (b) $\eta$ = 0.1, $\Delta x$ = 0.001 (c) $\eta$ = 0.001}
		\label{fig:perturbedvdW}
	\end{figure}
In order to obtain a reference solution to compare our results, we solve a linearised version of the Euler equations derived for the perturbations of the primitive variables $\left[\rho, u, p\right]$. Here, we apply the same initial pressure perturbation as before, with the amplitude of perturbation $\eta$ decreased to $10^{-5}$. This system of equations is solved over a domain [-1,2] using a central difference scheme. fourth order Runge-Kutta scheme and a grid size of 3000 cells. Figure~(\ref{fig:vdWref}a) shows a part of the domain i.e [0,1] in which solutions obtained from our well-balanced scheme compare favorably with the solutions from the linear Euler equations. This provides further evidence that the well-balanced scheme for the full Euler equations is able to compute small perturbations in a reliable way. Now, we compare this well-balanced solution with that obtained using a non well-balanced scheme for a very fine mesh of 300000 cells over the domain [-1,2]. From figure~(\ref{fig:vdWref}b), we observe that the solution obtained using the non well-balanced scheme shows a significant divergence from the correct solution even for such a fine mesh, again underscoring the importance of using a well balanced scheme.
	\begin{figure}
		\begin{tabular}{cc}
			\includegraphics[width=0.48\textwidth]{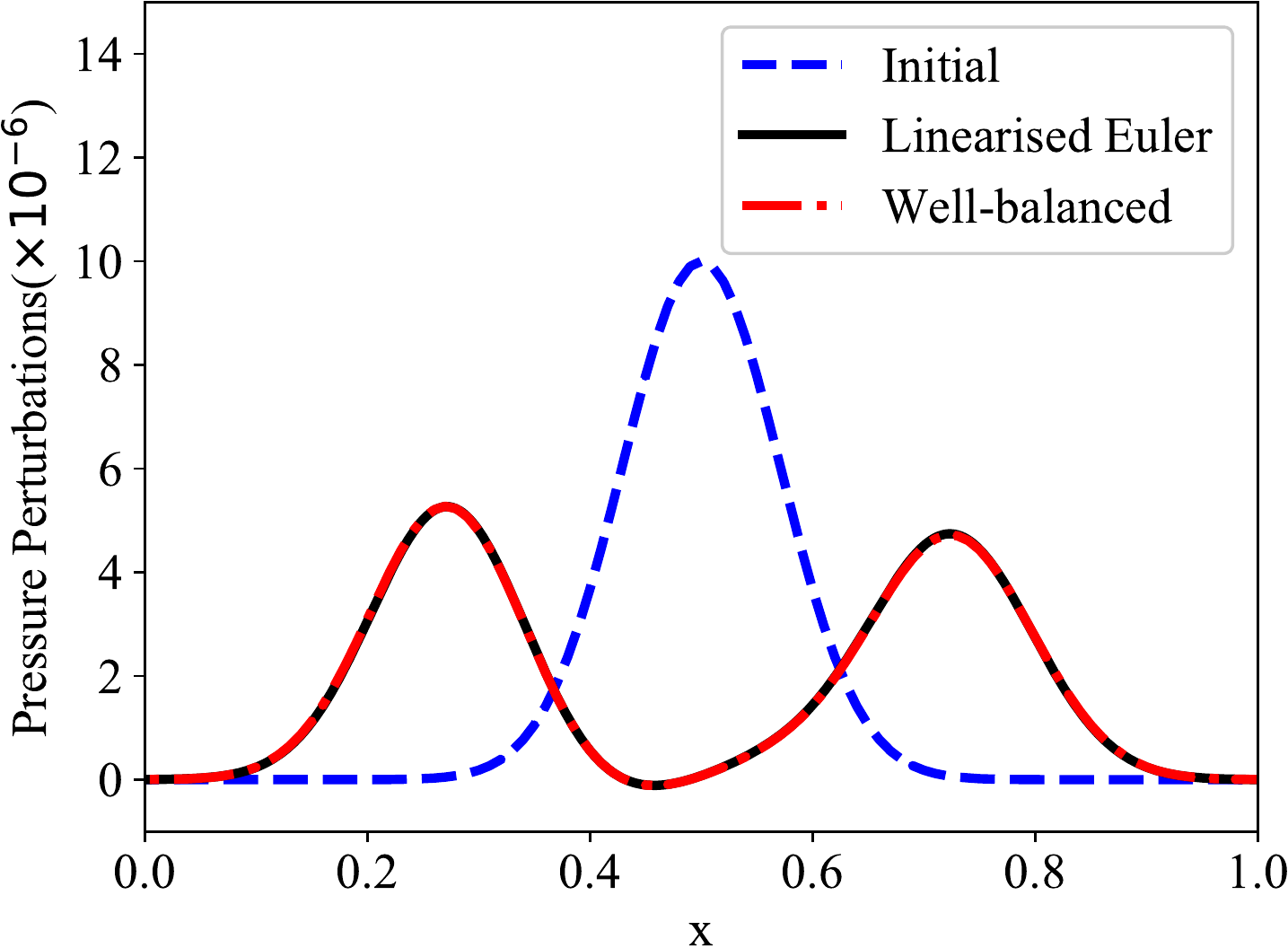} &
			\includegraphics[width=0.48\textwidth]{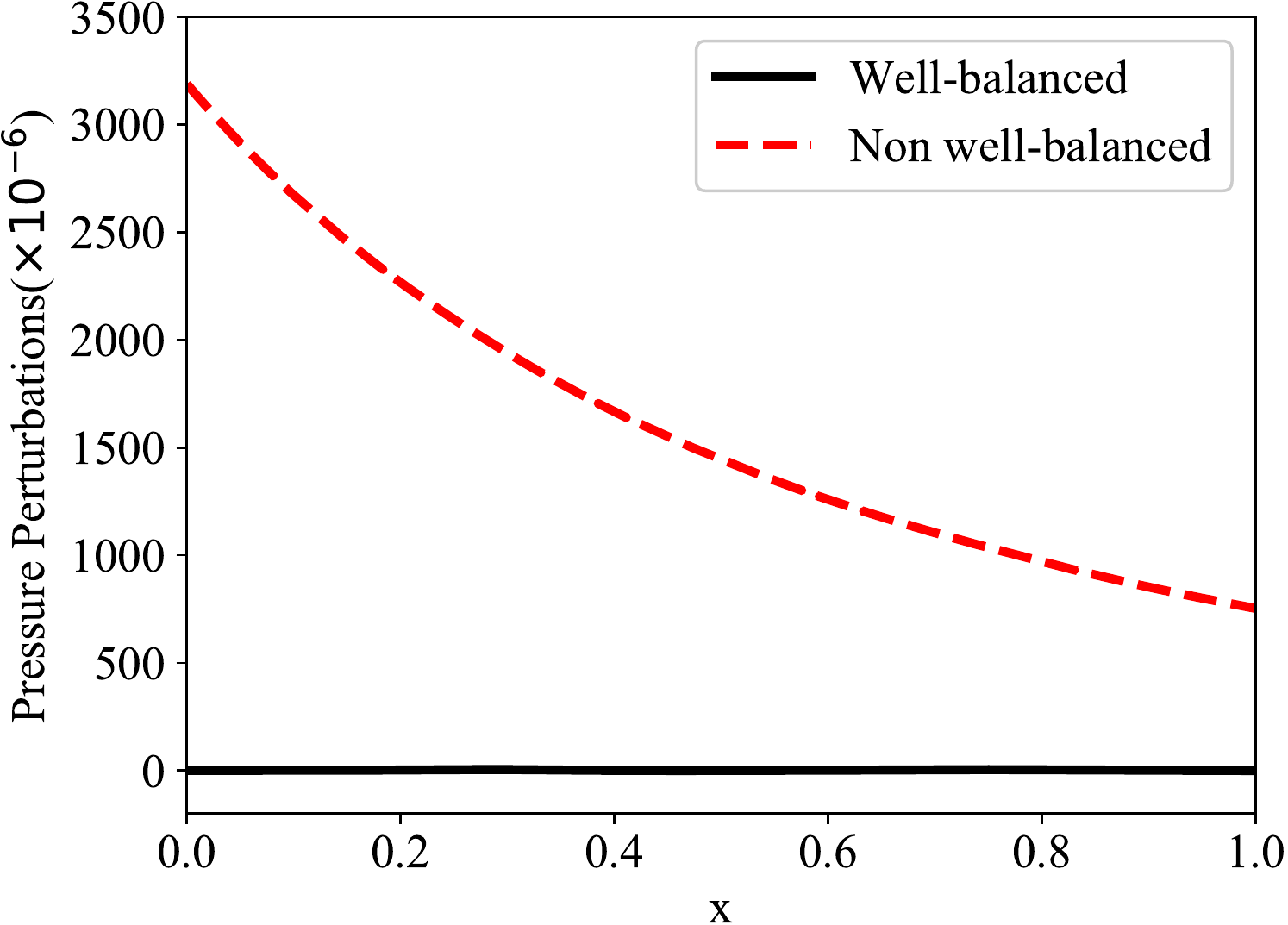}
		\end{tabular}
		\caption{Comparison of evolution of pressure perturbations of amplitude $\eta=10^{-5}$ from van der Waals hydrostatic state with (a) reference solution (b) solution from non well-balanced scheme.}
		\label{fig:vdWref}
	\end{figure}
\subsection{Sod test under gravitational field}
To study the performance of the scheme for discontinuous problems, we solve the Sod problem~\cite{Sod1978} under a gravitational field with potential $\phi (x) $ = $x$ in a domain of [0,1] with solid wall boundary conditions. The initial conditions are given by
\[
(\rho,u,p) = \begin{cases}
(1,0,1) & \textrm{if } x < \frac{1}{2}\\
(0.125,0,0.1) & \textrm{if } x>  \frac{1}{2}
\end{cases}
\]
The simulations are performed upto a final time $t$ = 0.2 with two grids of sizes 200 and 2000 cells using minmod based reconstruction scheme. To test the accuracy of the solution, we compare the results with a reference solution obtained using the CRWENO5 scheme on a fine mesh of 2000 cells. For the coarse mesh, the solutions for density and pressure near the shock region obtained using our well-balanced scheme is compared with that obtained for CRWENO5 scheme with the same grid size in figure~ (\ref{fig:sodzoomed}). It can be observed that the current second order well-balanced scheme achieves reasonable accuracy compared to a 5'th order WENO scheme. The results for various parameters are illustrated in figures~(\ref{fig:sod}). We observe that the solutions compare favourably with that obtained using the higher order reference scheme and the scheme is able to capture the features of the solution even for a coarse mesh with grid spacing, $\Delta x$ = 0.005. Since the gravity force is acting towards the left boundary, the density and velocity distributions are being pulled in that direction. We can see that the solutions do not have spurious oscillations even in the presence of shocks and contact discontinuities.
\begin{figure}
\begin{center}
\begin{tabular}{cc}
\includegraphics[width=0.4\textwidth]{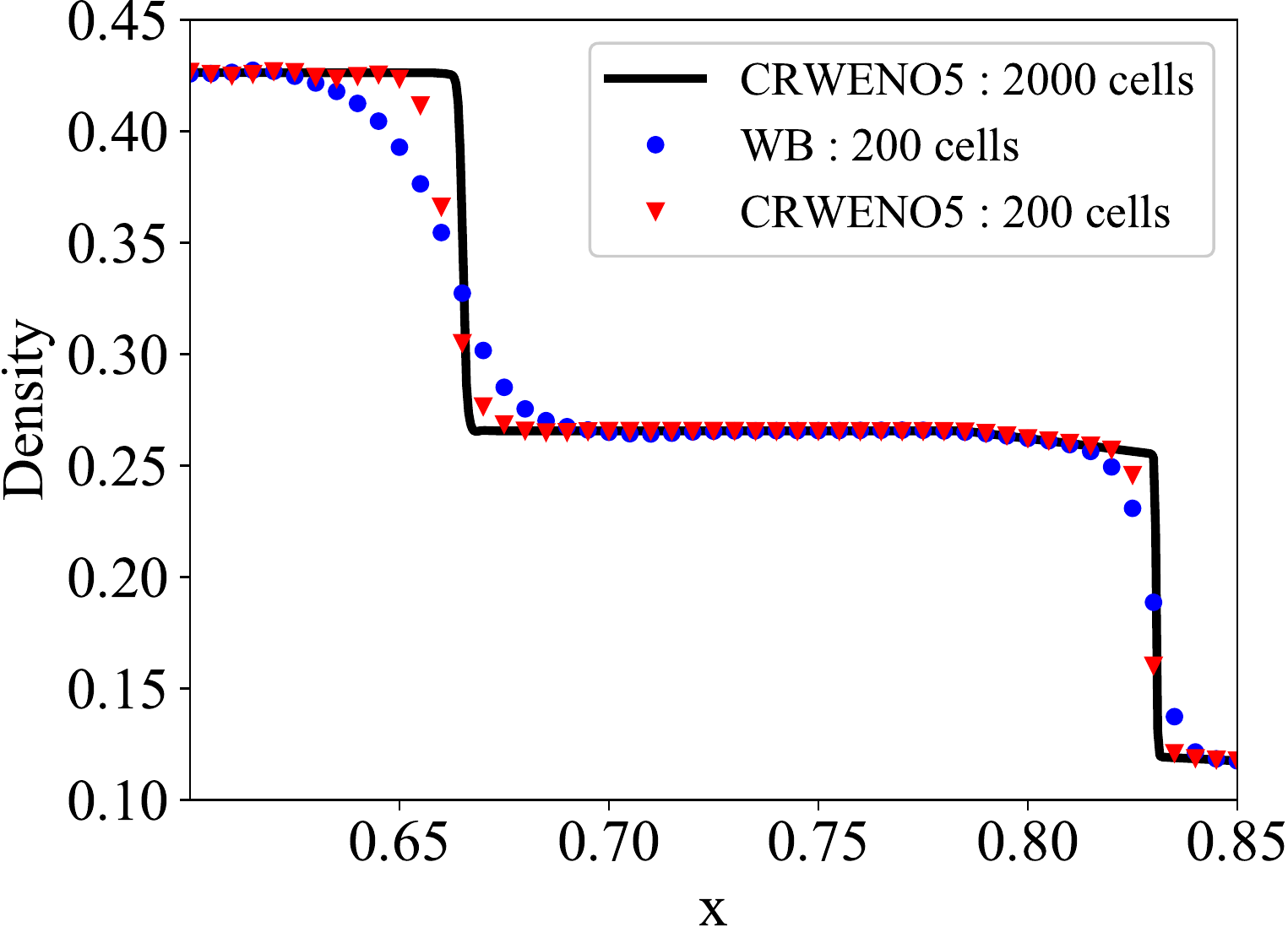} &
\includegraphics[width=0.4\textwidth]{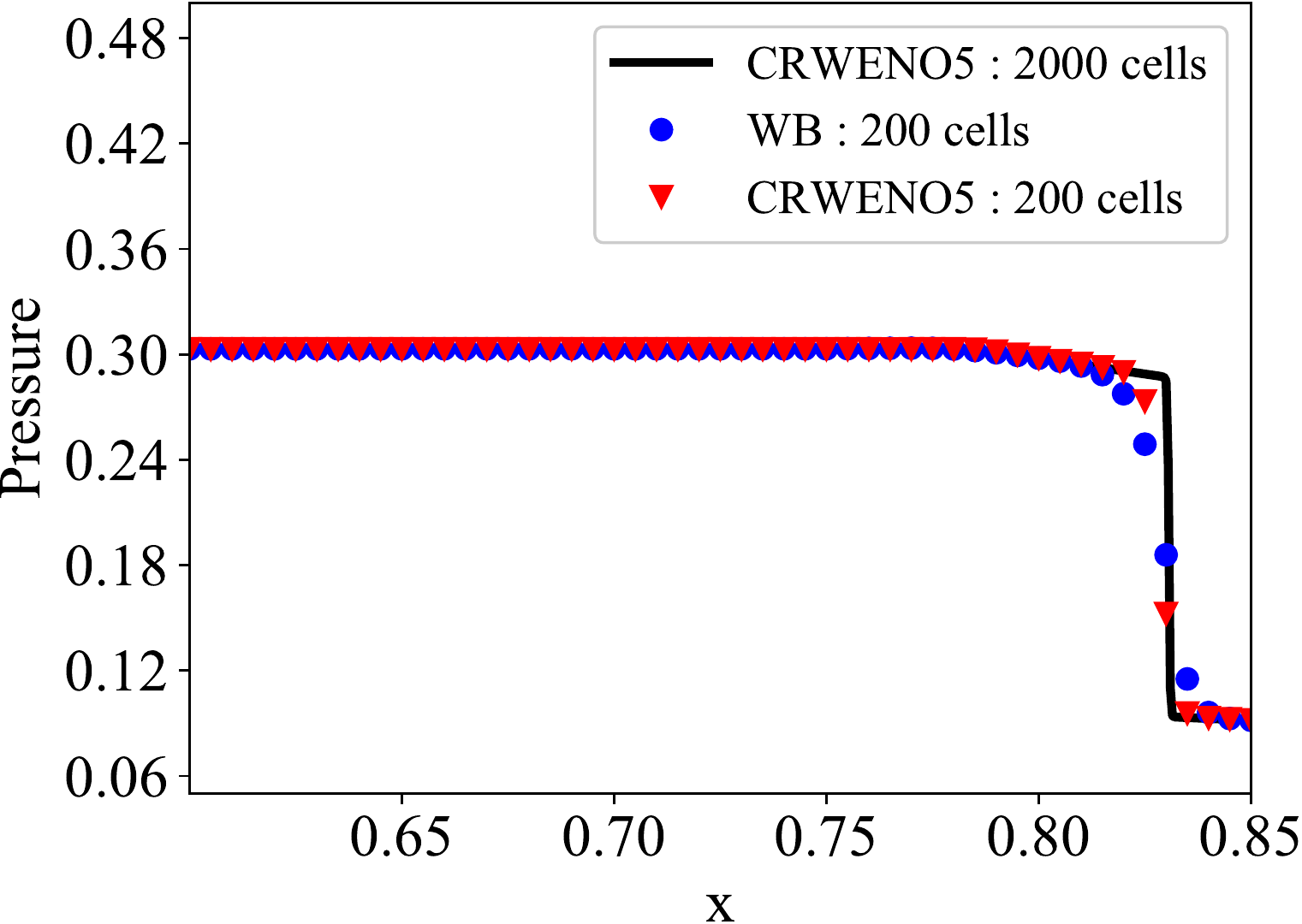} \\
(a)  & (b)  
\end{tabular}
\caption{Comparison of zoomed view near the shock region for Sod test under a linear gravitational field with CRWENO5 scheme for (a) density (b) pressure}
\label{fig:sodzoomed}
\end{center}
\end{figure}

\begin{figure}
\begin{center}
\begin{tabular}{cc}
\includegraphics[width=0.4\textwidth]{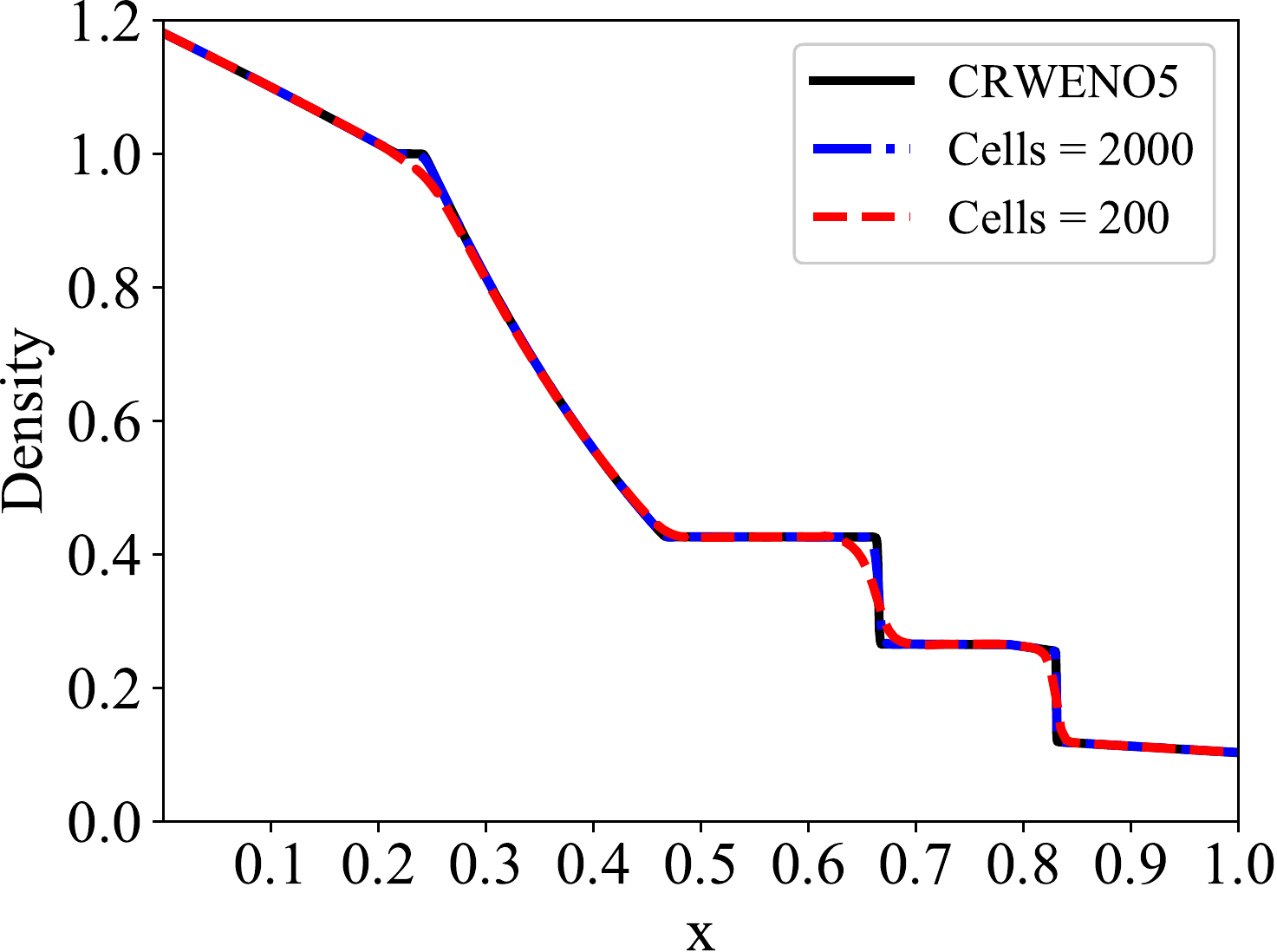} &
\includegraphics[width=0.4\textwidth]{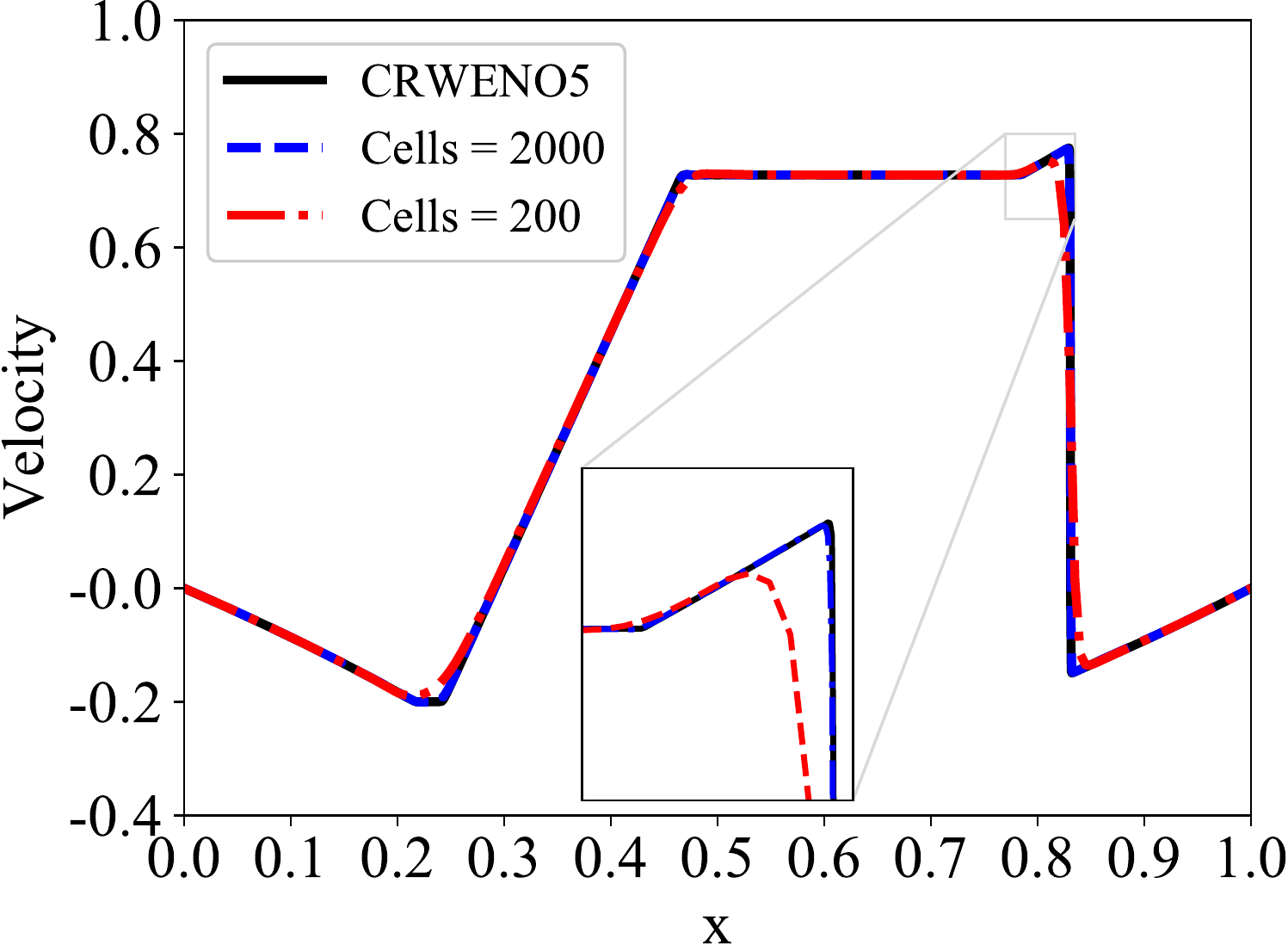} \\
(a)  & (b)  \\
\includegraphics[width=0.4\textwidth]{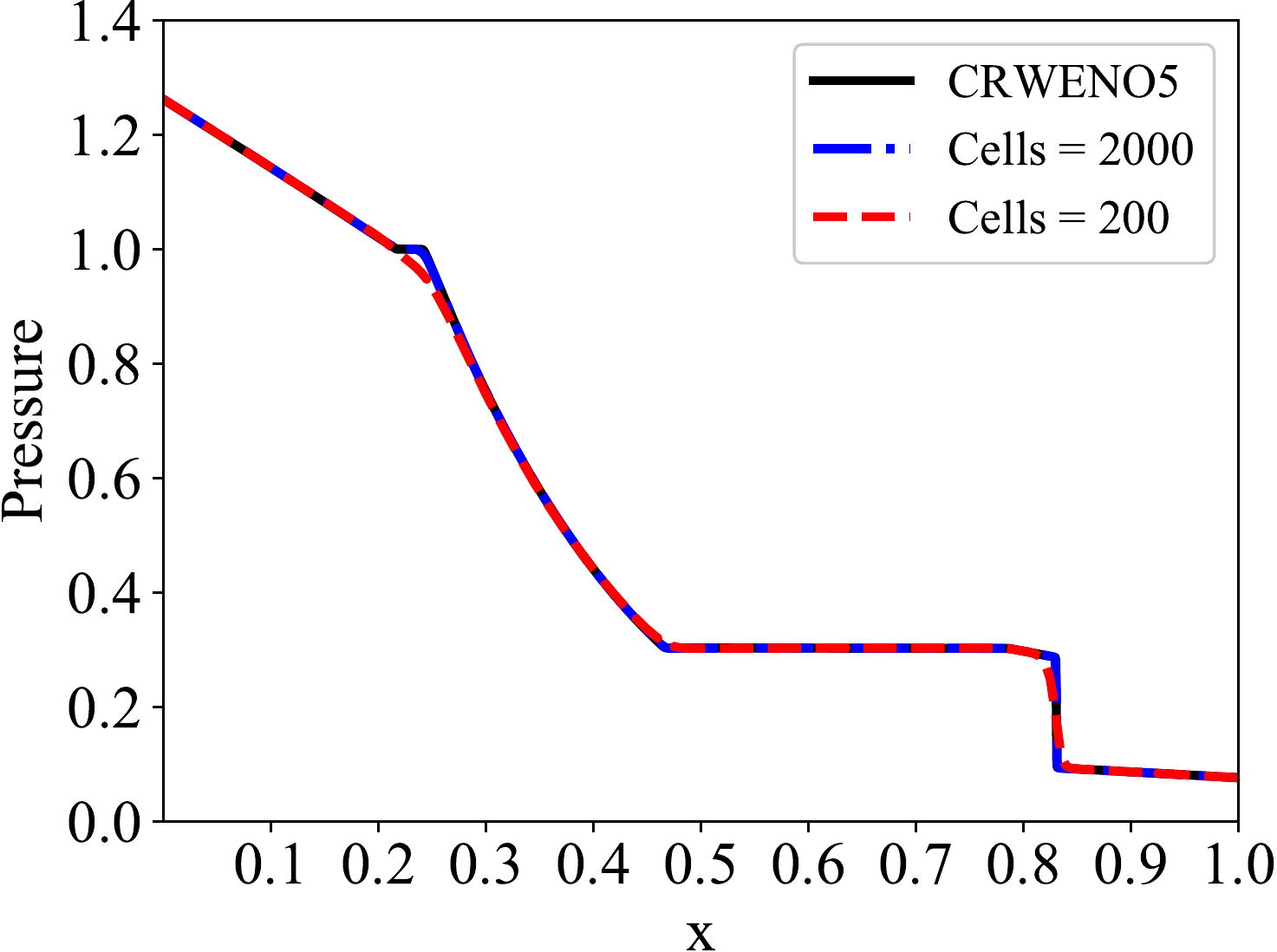} &
\includegraphics[width=0.4\textwidth]{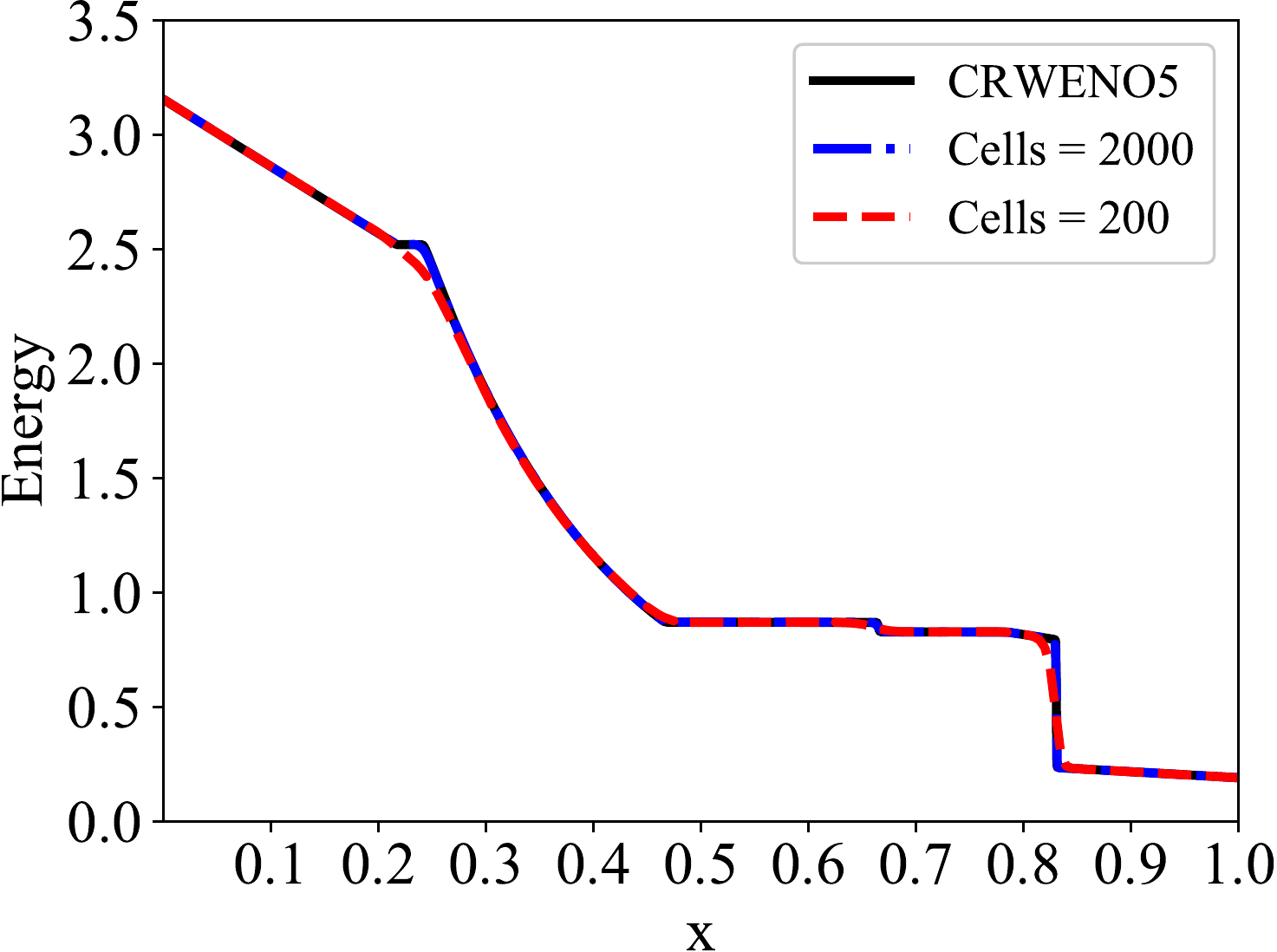} \\
(c)  & (d) \\
\end{tabular}
\caption{Sod test under a linear gravitational field. (a) density, (b) velocity, (c) pressure, (d) energy }
\label{fig:sod}
\end{center}
\end{figure}
\subsection{Contact discontinuity under gravitational field}
In this test case, the initial condition has a contact discontinuity in a  domain of size [0,1] under a gravitational field of $\phi(x)$ = $x$. Without gravity this would be a stationary solution. Solid wall boundary conditions are used at both ends of the domain and initial conditions are  given by
\[
(\rho,u,p) = \begin{cases}
(1,0,1) & \textrm{if } x < \frac{1}{2}\\
(10,0,1) & \textrm{if } x>  \frac{1}{2}
\end{cases}
\]
The simulations are performed upto a final time $t$ = 0.6 with two grids of sizes 200 and 2000 cells. The results are compared with those obtained from the non well-balanced scheme. We observe that both the schemes produce similar results even on the coarse mesh as shown in figure~(\ref{fig:contact}). This demonstrates that the present scheme performs similarly to standard schemes for discontinuous problems.
\begin{figure}
\begin{center}
\begin{tabular}{cc}
\includegraphics[width=0.4\textwidth]{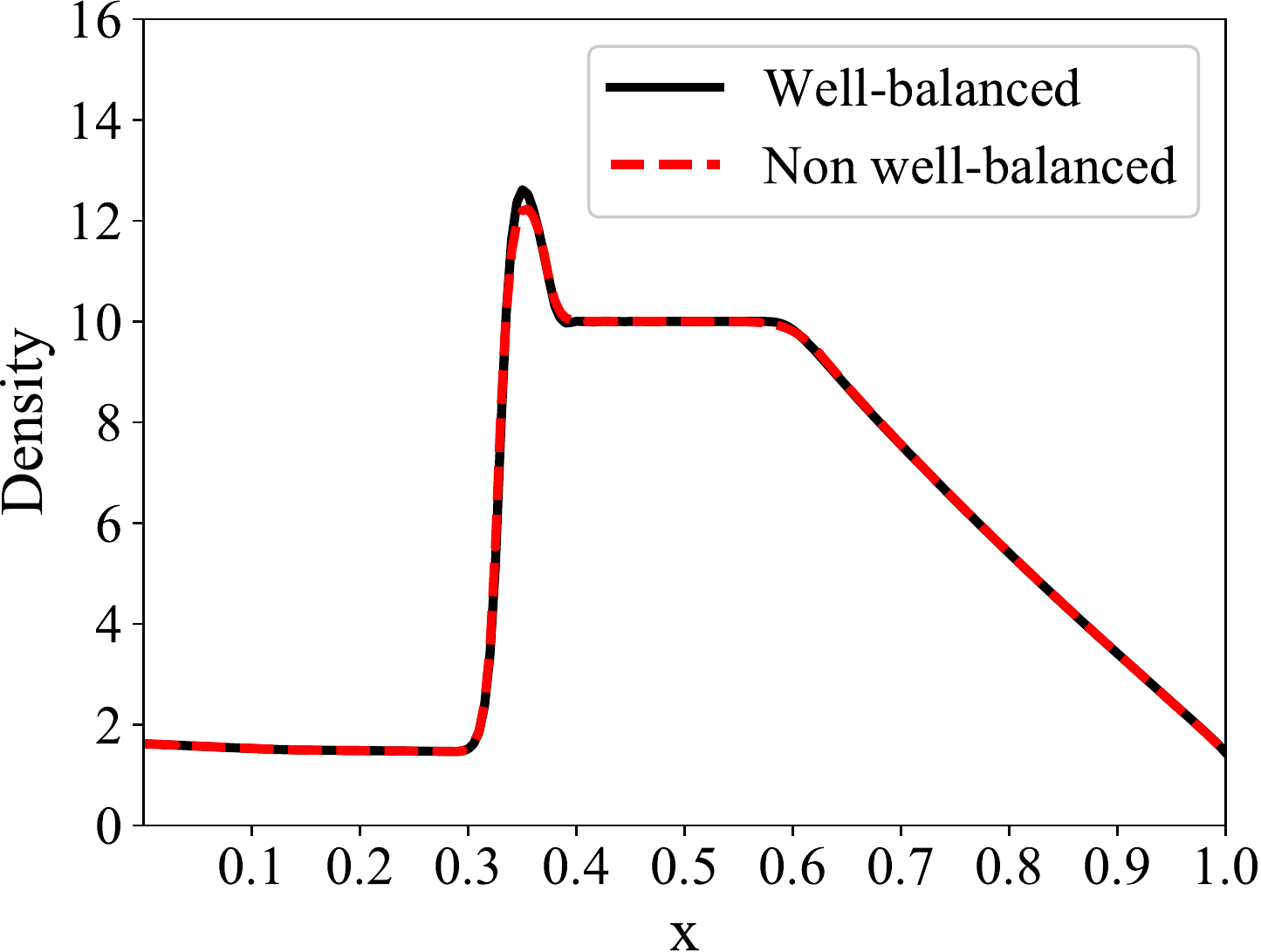} &
\includegraphics[width=0.4\textwidth]{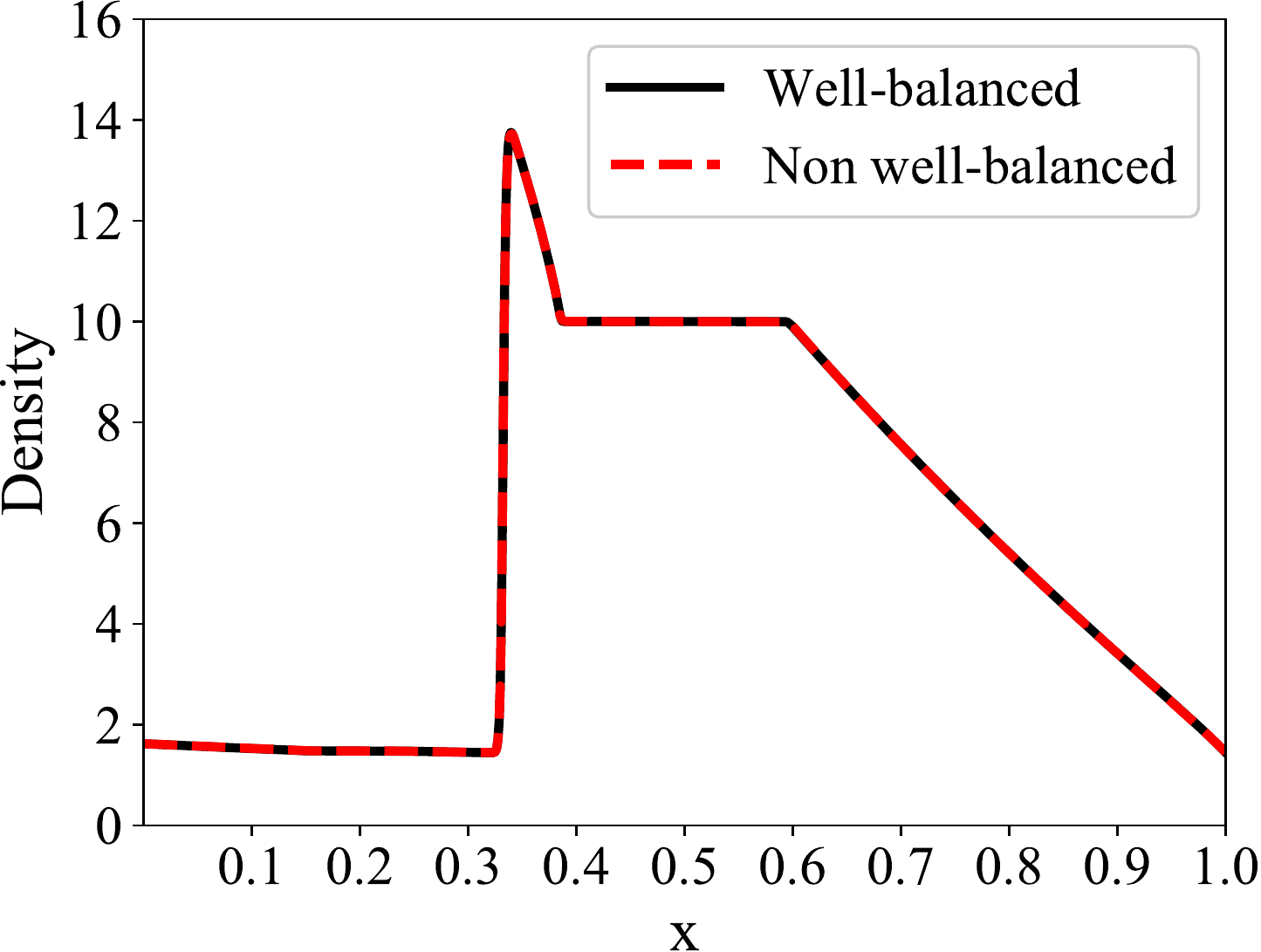}\\
(a)  & (d)  \\
\includegraphics[width=0.4\textwidth]{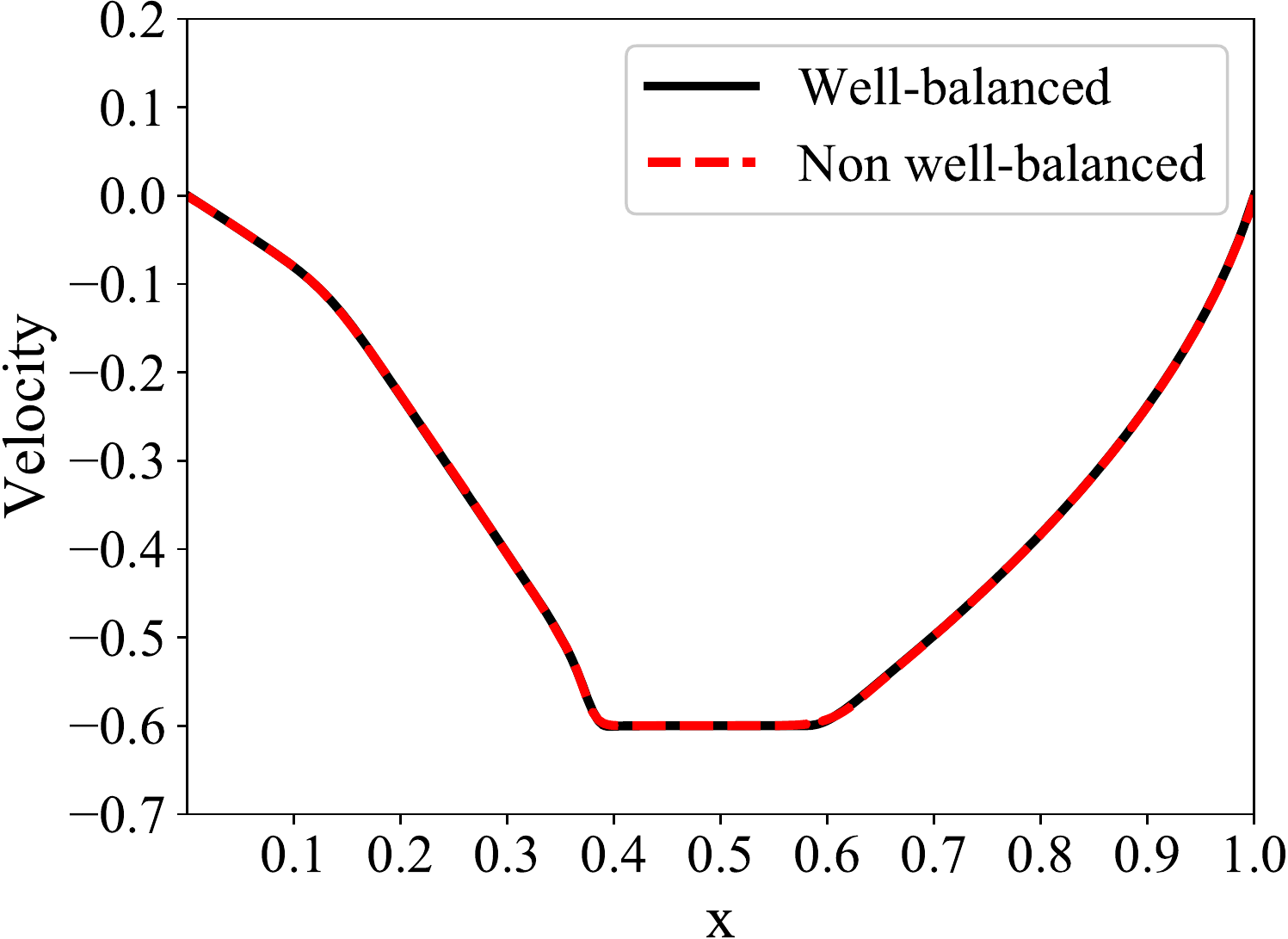} &
\includegraphics[width=0.4\textwidth]{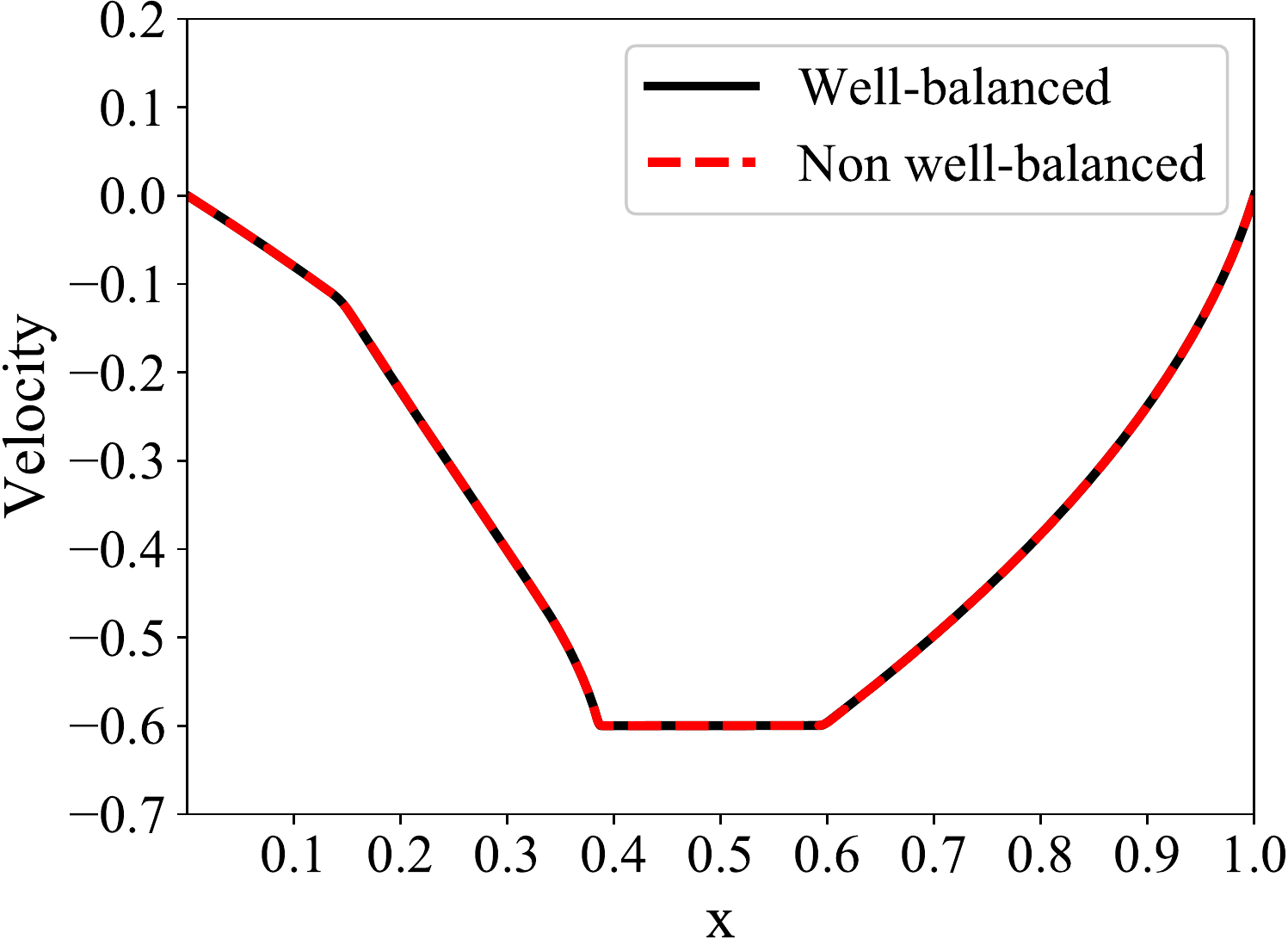}\\
(b)  & (e)  \\
\includegraphics[width=0.4\textwidth]{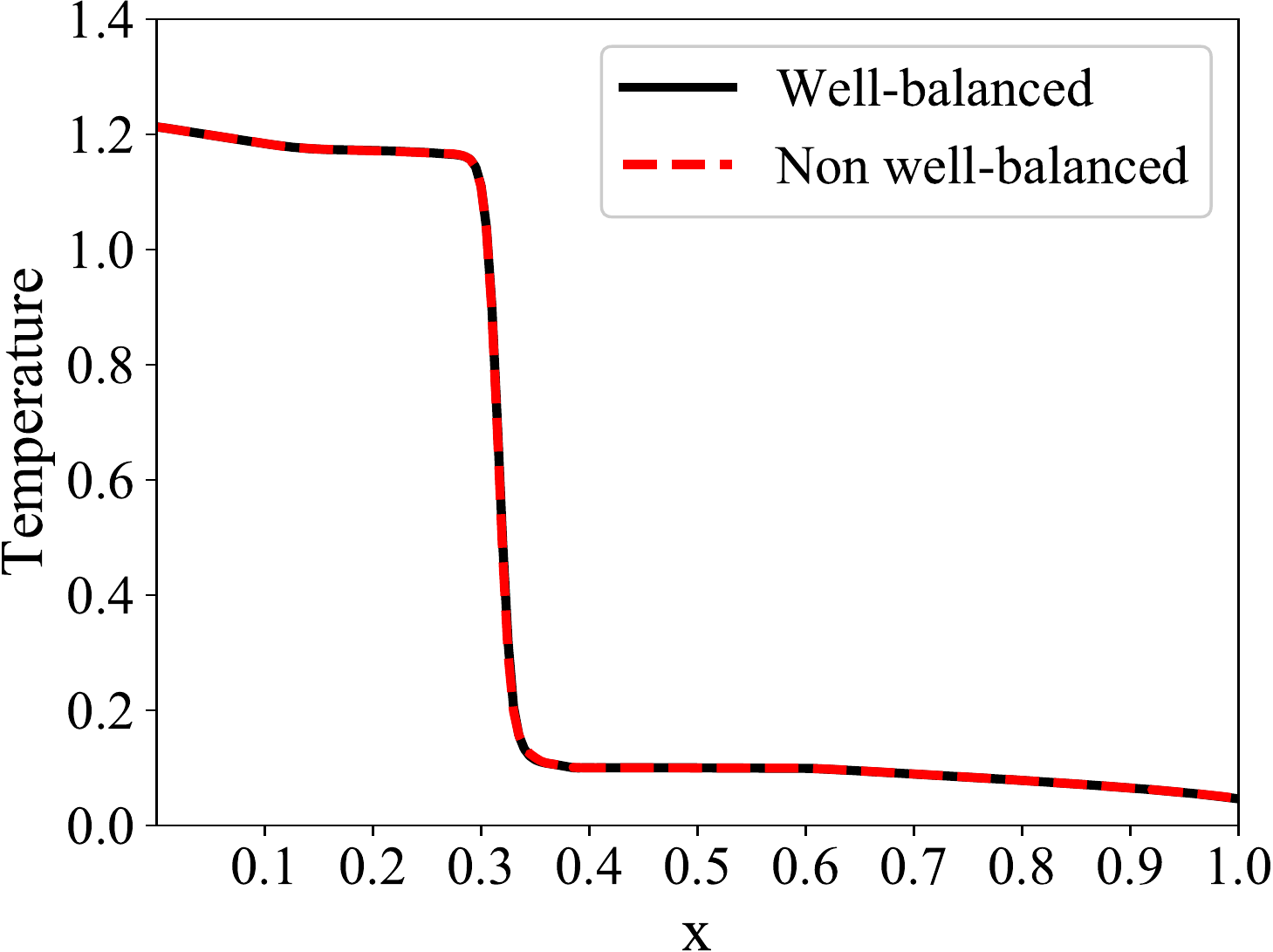} &
\includegraphics[width=0.4\textwidth]{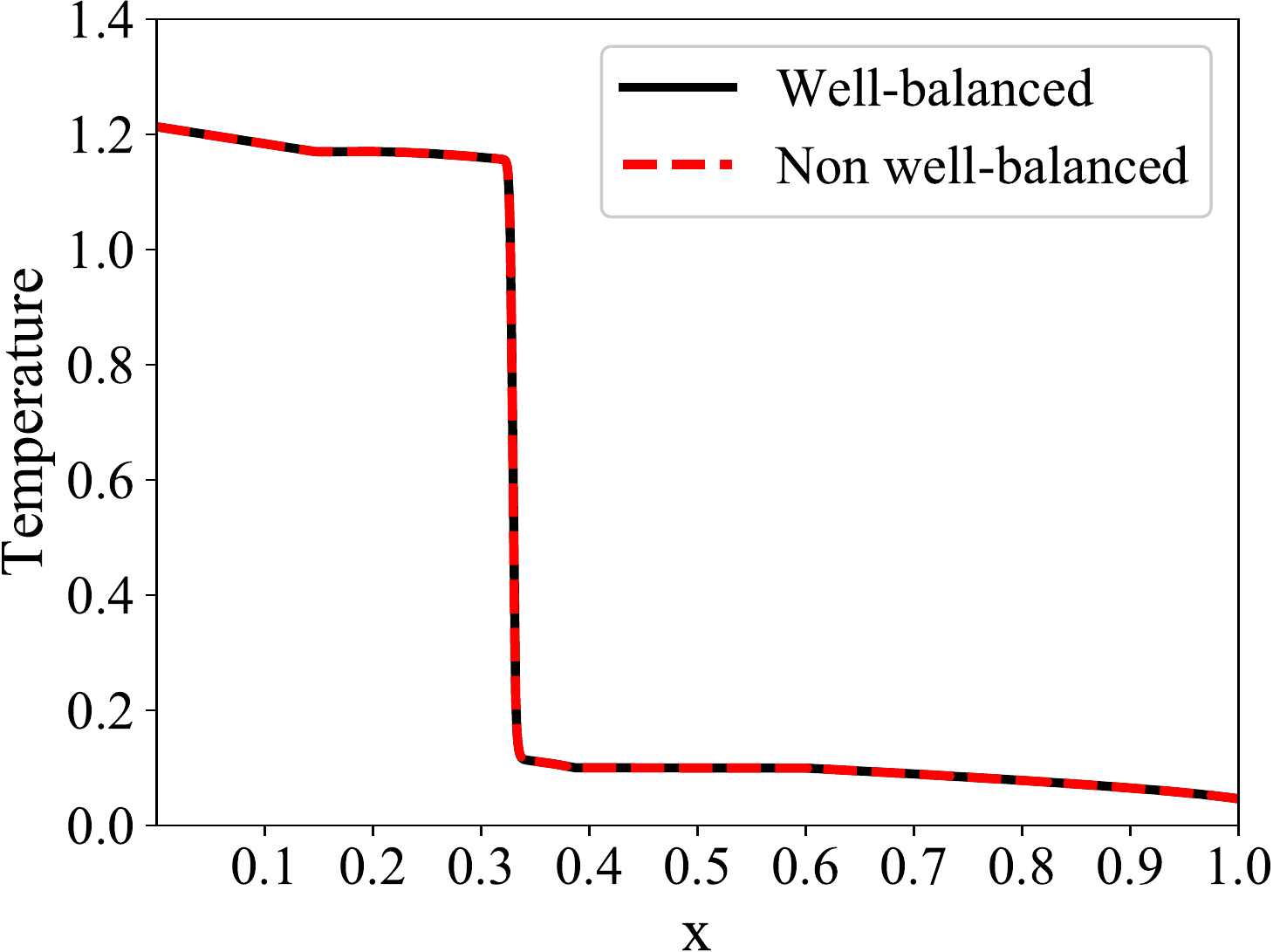}\\
(c)  & (f) \\
\end{tabular}
\caption{Contact Discontinuity under a linear gravitational field. 1st Column - 200 cells; 2nd Column - 2000 cells }
\label{fig:contact}
\end{center}
\end{figure}
\section{2-D numerical results}
\label{sec:res2d}
All the test cases make use of HLLC flux and MUSCL reconstruction. The specific heat ratio is $\gamma = 1.4$ and gas constant is equal to unity.
\subsection{Perturbation from isothermal hydrostatic solution}
\label{sec:2diso}
We consider an initial condition~\cite{Li2016} in which an isothermal hydrostatic condition is imposed on a unit square domain with transmissive boundary conditions and a gravitational potential defined as $\phi(x,y)$  = $x+y$. The hydrostatic state is defined as follows:
\[
\rhoe(x,y) = \rho_0 \exp(-\rho_0  \phi(x,y)/p_0), \qquad \pe(x, y) = p_0 \exp( -\rho_0 \phi(x,y)/p_0)
\]
where $\rho_0 = 1.21$ and $p_0 = 1$. To study the performance of the scheme in resolving small perturbations, we add an initial perturbation to the hydrostatic pressure distribution as follows
\[
p(x,y,0) = \pe(x,y,0) + \eta \exp(-100 \rho_0 ((x-0.3)^2 + (y -0.3)^2)/p_0)
\]
The simulation is performed on a grid of  $51 \times 51$ upto a final time of $t = 0.15$ units and the results for two different values of amplitude of pressure perturbations ($\eta$ = 0.1, 0.001) are compared with a non well-balanced scheme in which the source terms are computed using central differences. Also, to test the accuracy, a reference solution using CRWENO5 is computed using the same parameters. From figures~(\ref{fig:2disothermaleta1}) and (\ref{fig:2disothermaleta001}), we can see that the non well-balanced schemes distorts the small perturbations and this become more severe when the perturbation is small. The well-balanced scheme is able to resolve the pressure perturbations  for both amplitudes and does not suffer from loss of accuracy for very small perturbations. It can also be observed that the solutions obtained from the current second-order well-balanced scheme is quite comparable with the CRWENO5 solutions.
\begin{figure}
\begin{center}
	\begin{tabular}{ccc}
		\includegraphics[width=0.32\textwidth]{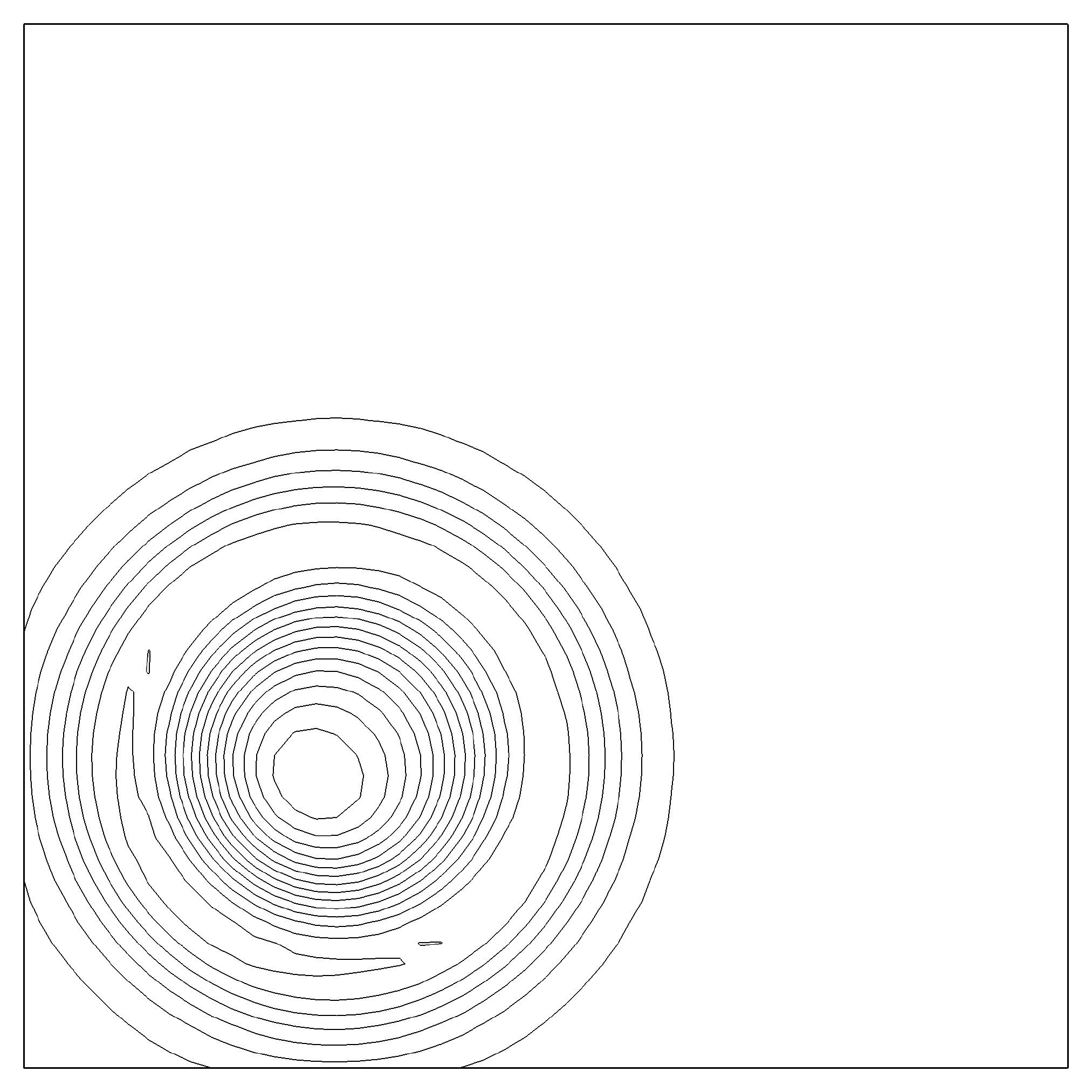} &
		\includegraphics[width=0.32\textwidth]{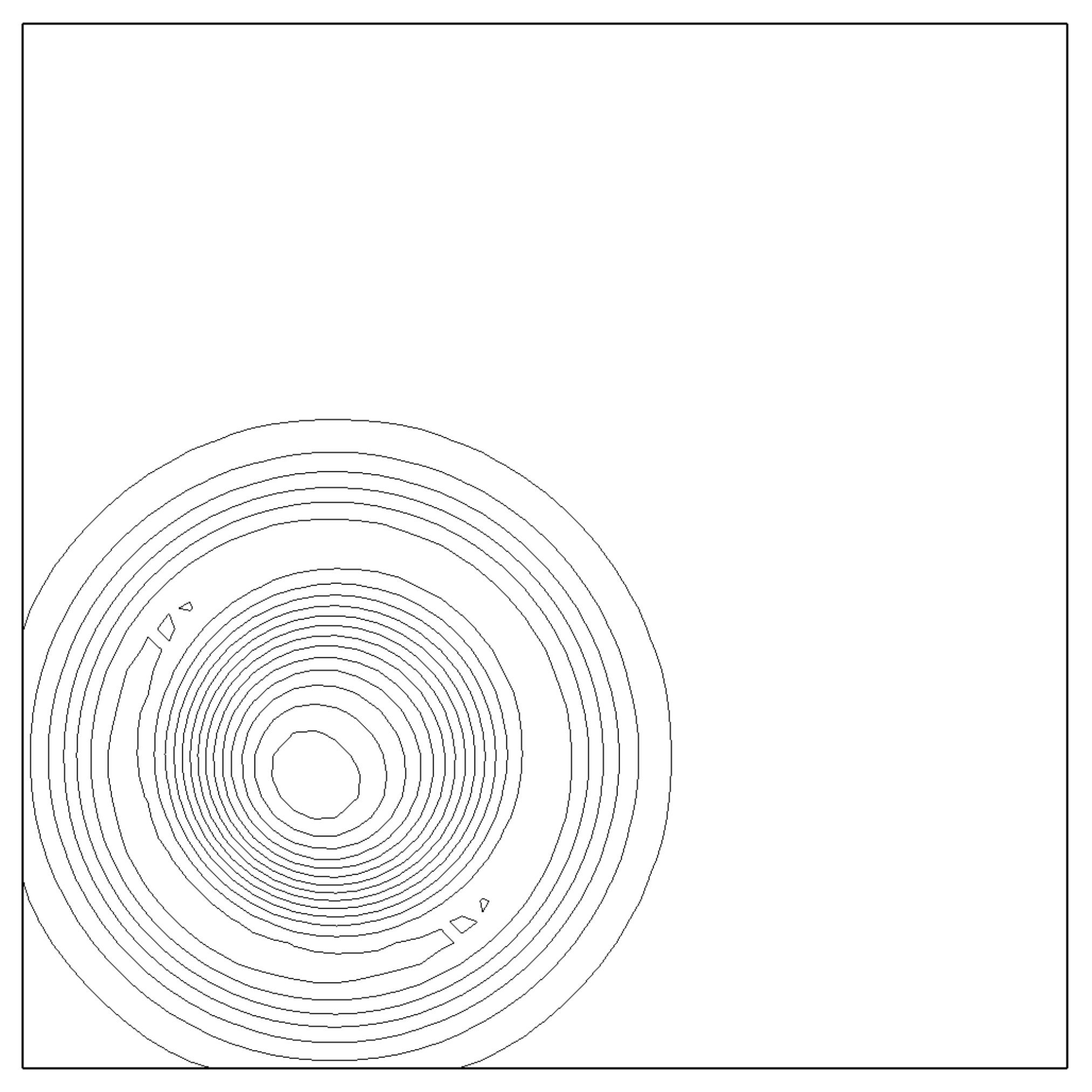} &
		\includegraphics[width=0.32\textwidth]{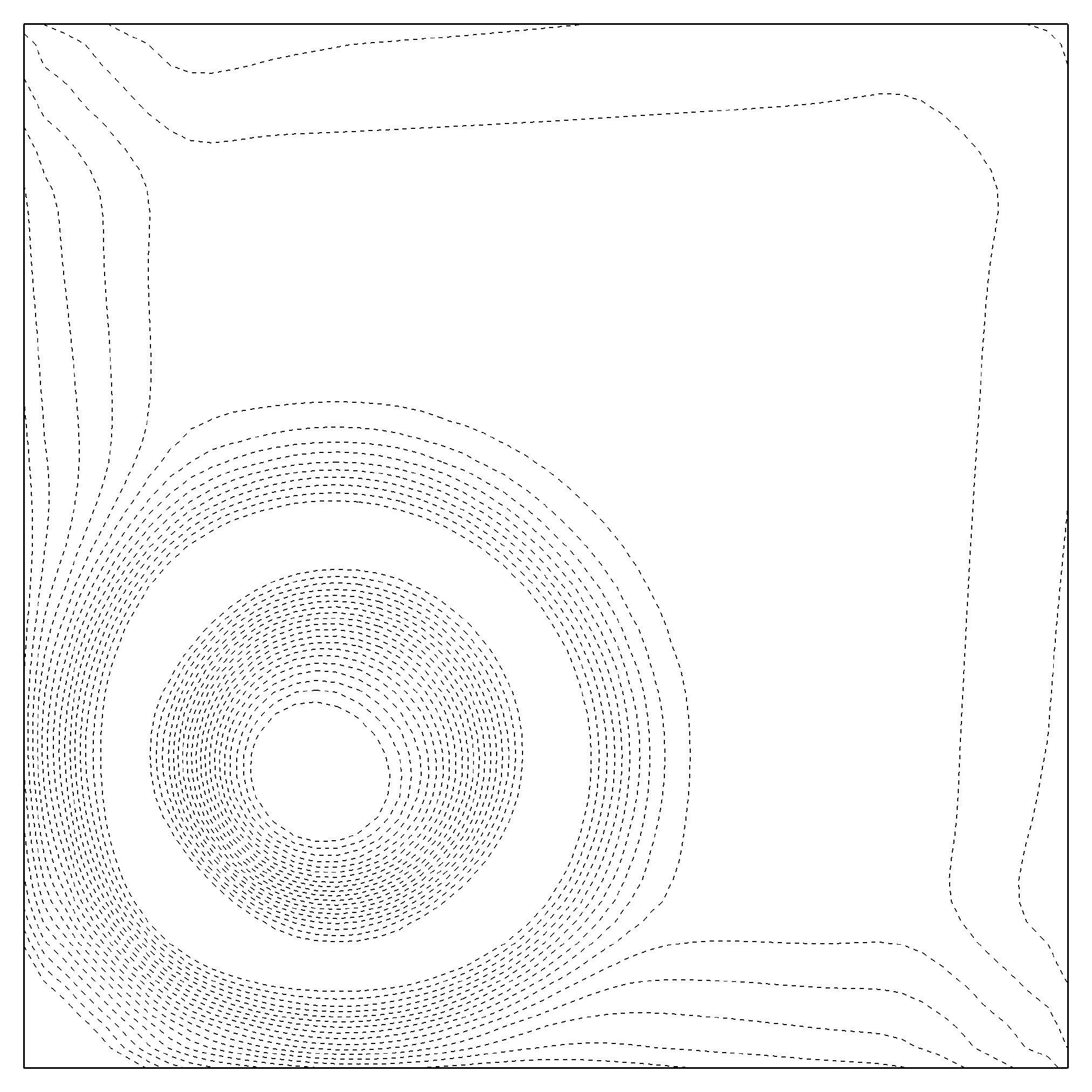} \\
		(a) Well-balanced scheme & (b) CRWENO5 & (c) Non well-balanced scheme
	\end{tabular}
	\caption{Pressure perturbations for test case defined in section (\ref{sec:2diso})  for $\eta$ = 0.1. 20 equally spaced contour levels between -0.03 and +0.03}
	\label{fig:2disothermaleta1}
\end{center}
\end{figure}
\begin{figure}
\begin{center}
	\begin{tabular}{ccc}
		\includegraphics[width=0.3\textwidth]{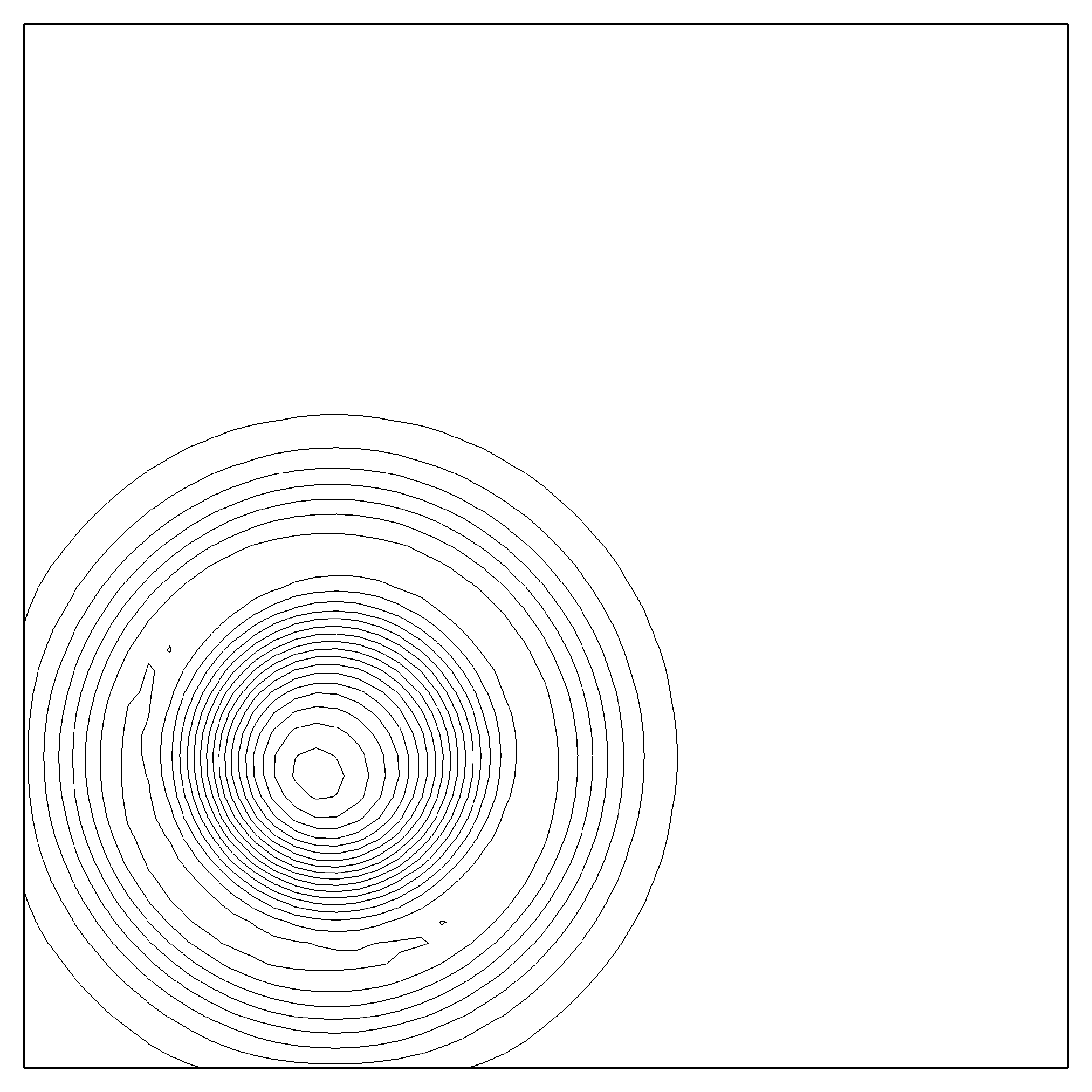} &
		\includegraphics[width=0.3\textwidth]{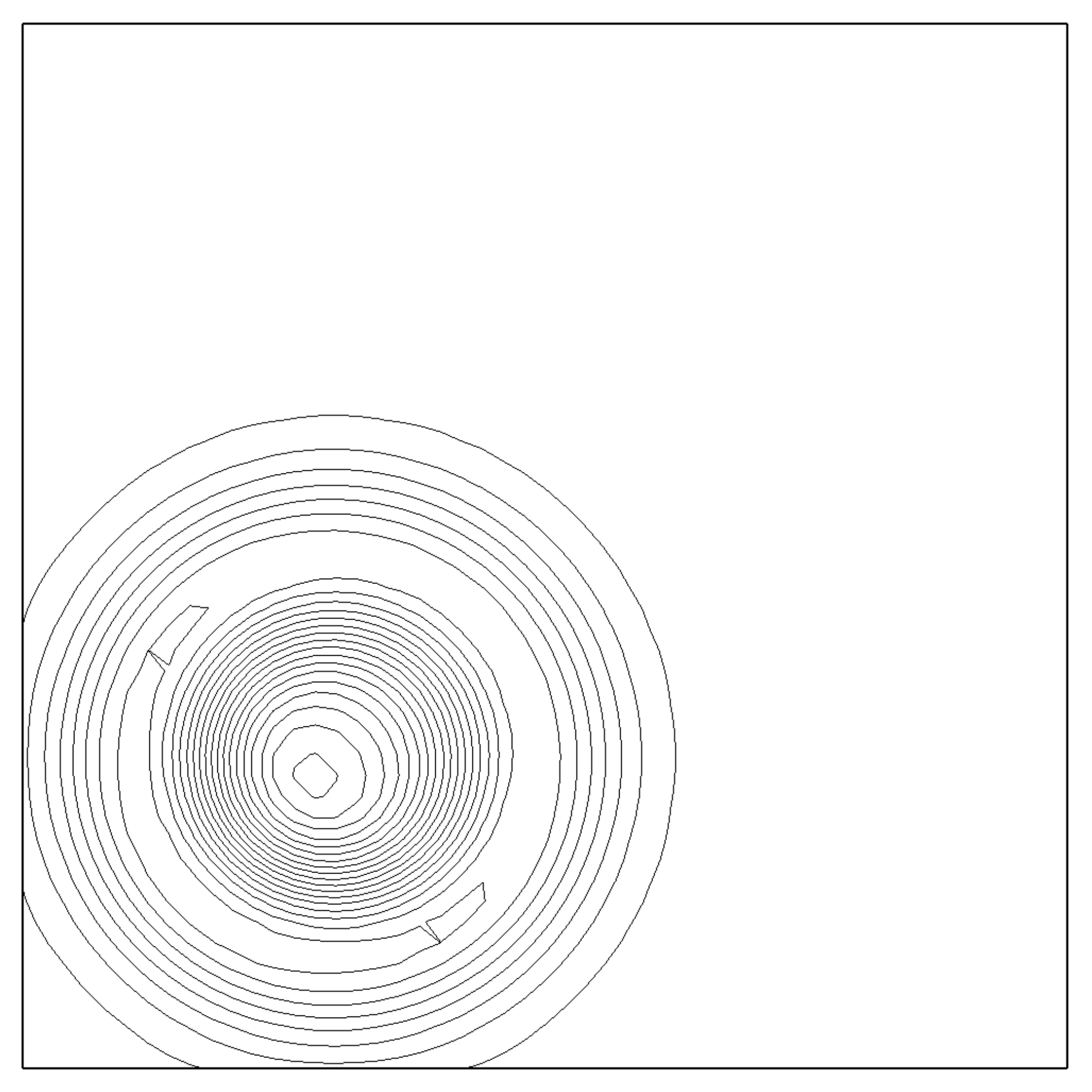} &
		\includegraphics[width=0.3\textwidth]{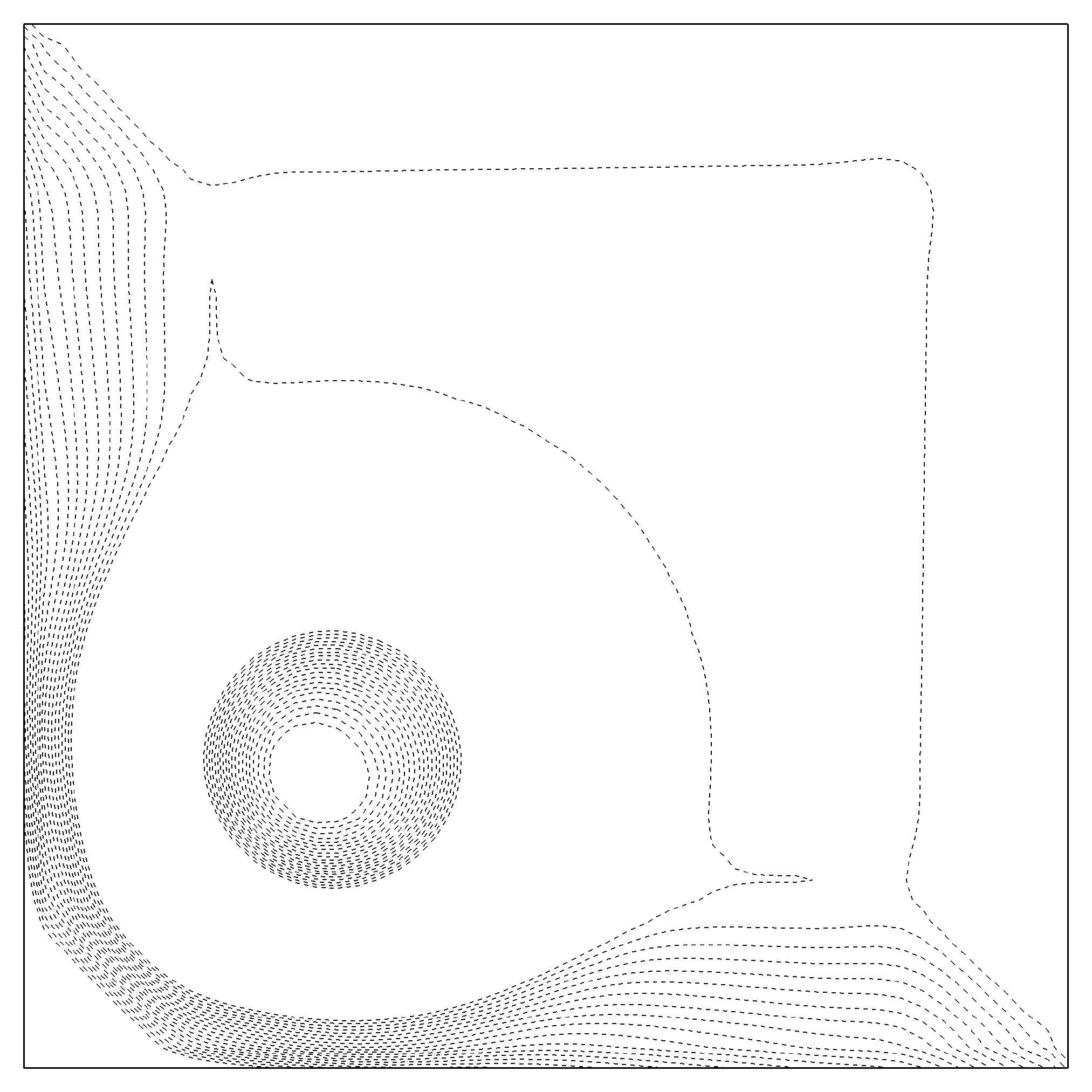} \\
		(a) Well-balanced scheme & (b) CRWENO5 & (c) Non well-balanced scheme
	\end{tabular}
	\caption{Pressure perturbations for test case defined in section (\ref{sec:2diso})  for $\eta$ = 0.001. (a) and (b) 20 equally spaced contour levels between -0.00026 and +0.00026 , (c) 20 equally spaced contour levels between -0.02 and +0.00026 }
	\label{fig:2disothermaleta001}
\end{center}
\end{figure}

\subsection{Perturbation from polytropic hydrostatic solution}
\label{sec:2dpolypert}
In this test case, we consider an initial condition in which a 2-D polytropic condition  is used to obtain a discrete hydrostatic state on a unit square domain. The boundary of the domain is taken to be transmissive and a gravitational potential of $\phi(x,y) = x+y$ is imposed on the domain. The hydrostatic state is defined as follows
\[
\rhoe(x,y) = (1 - \frac{\gamma - 1}{\gamma}\phi(x,y))^{\frac{1}{\gamma -1}} \qquad \pe(x,y) =  \rho^{\gamma}(x,y)
\]
Now, we add an initial perturbation to the hydrostatic pressure distribution as follows
\[
p(x,y,0) = \pe(x,y) + \eta \exp(-100((x-0.3)^2 + (y -0.3)^2))
\]
The simulation is performed on a grid of resolution 51$\times$51 upto a final time of $t$ = 0.15 and the results for two different values of amplitude of pressure perturbations ($\eta$ = 0.1, 0.001) are compared with a non well-balanced scheme in which the source terms are computed using central differences. We also compare the results with that obtained using the higher order CRWENO5 scheme. The scheme does not preserve the exact hydrostatic state as the gravitational force is not aligned with the Cartesian grid. However, from figures (\ref{fig:2dpolytropiceta1}) and (\ref{fig:2dpolytropiceta001}),  we can see that the well-balanced scheme is able to resolve the pressure perturbation reasonably accurately while the solutions obtained from non well-balanced scheme is distorted to a very large extent. We can also observe that the results are comparable with that obtained using the reference CRWENO5 scheme even for the smaller perturbations.
\begin{figure}
	\begin{center}
		\begin{tabular}{ccc}
			\includegraphics[width=0.30\textwidth]{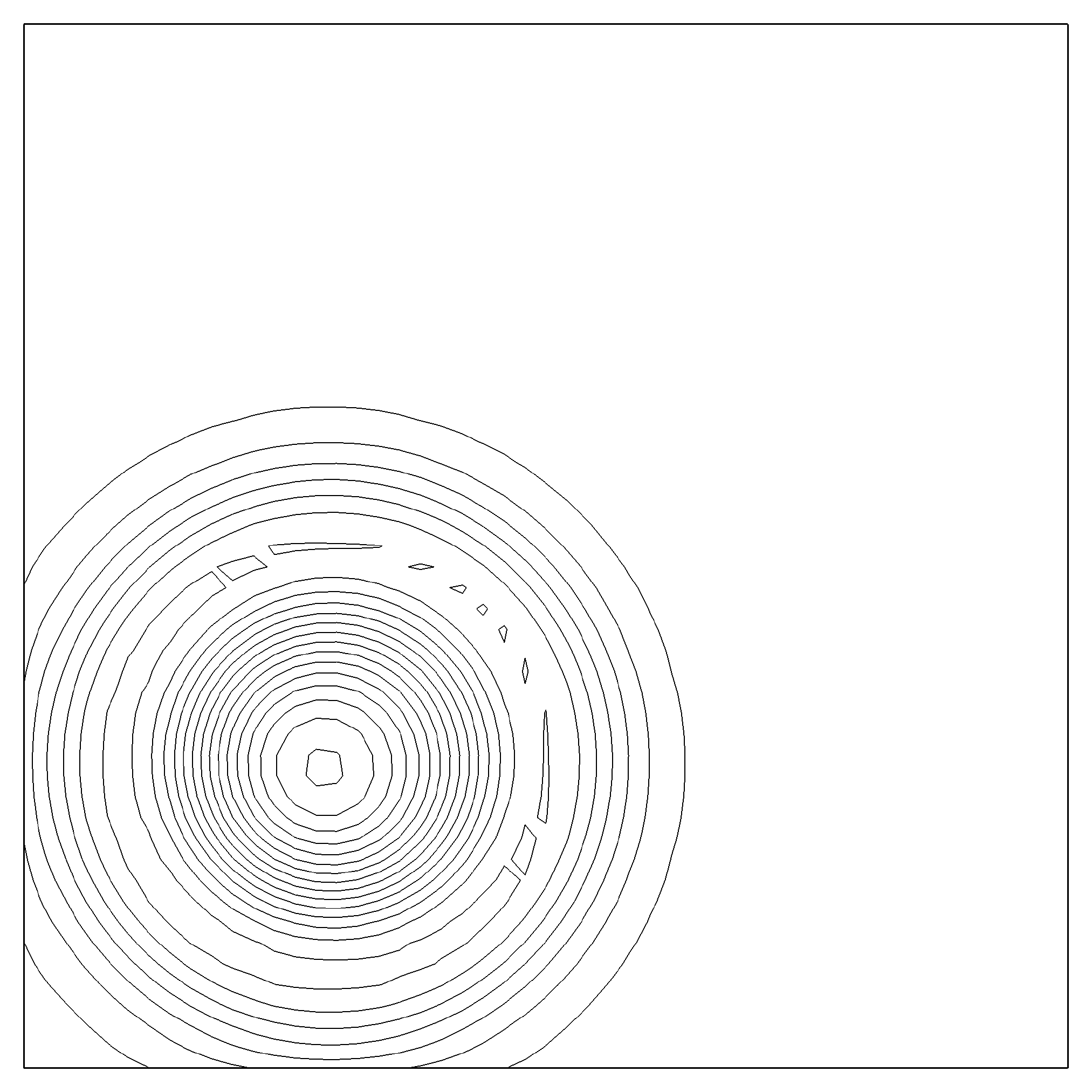} &
			\includegraphics[width=0.30\textwidth]{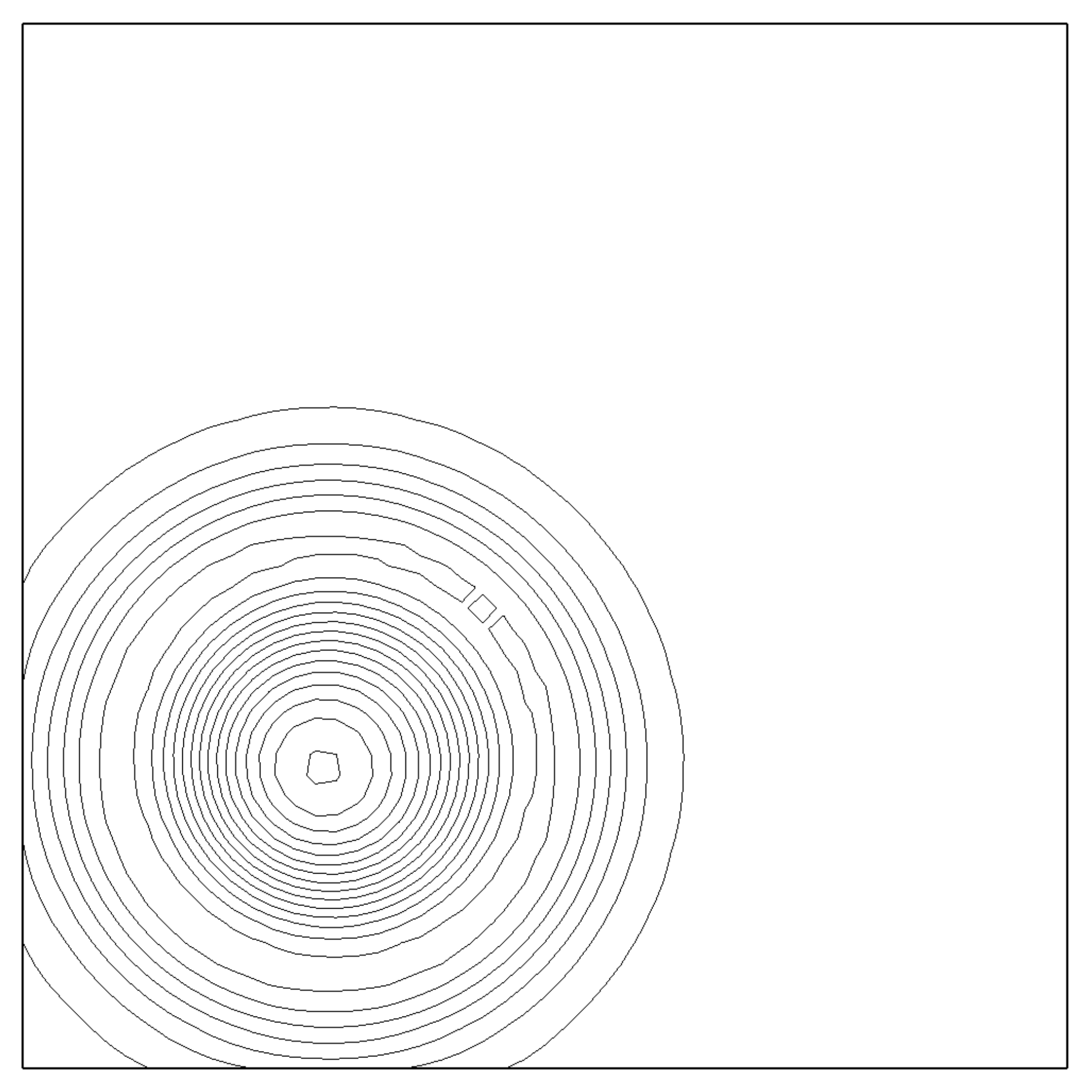} &
			\includegraphics[width=0.30\textwidth]{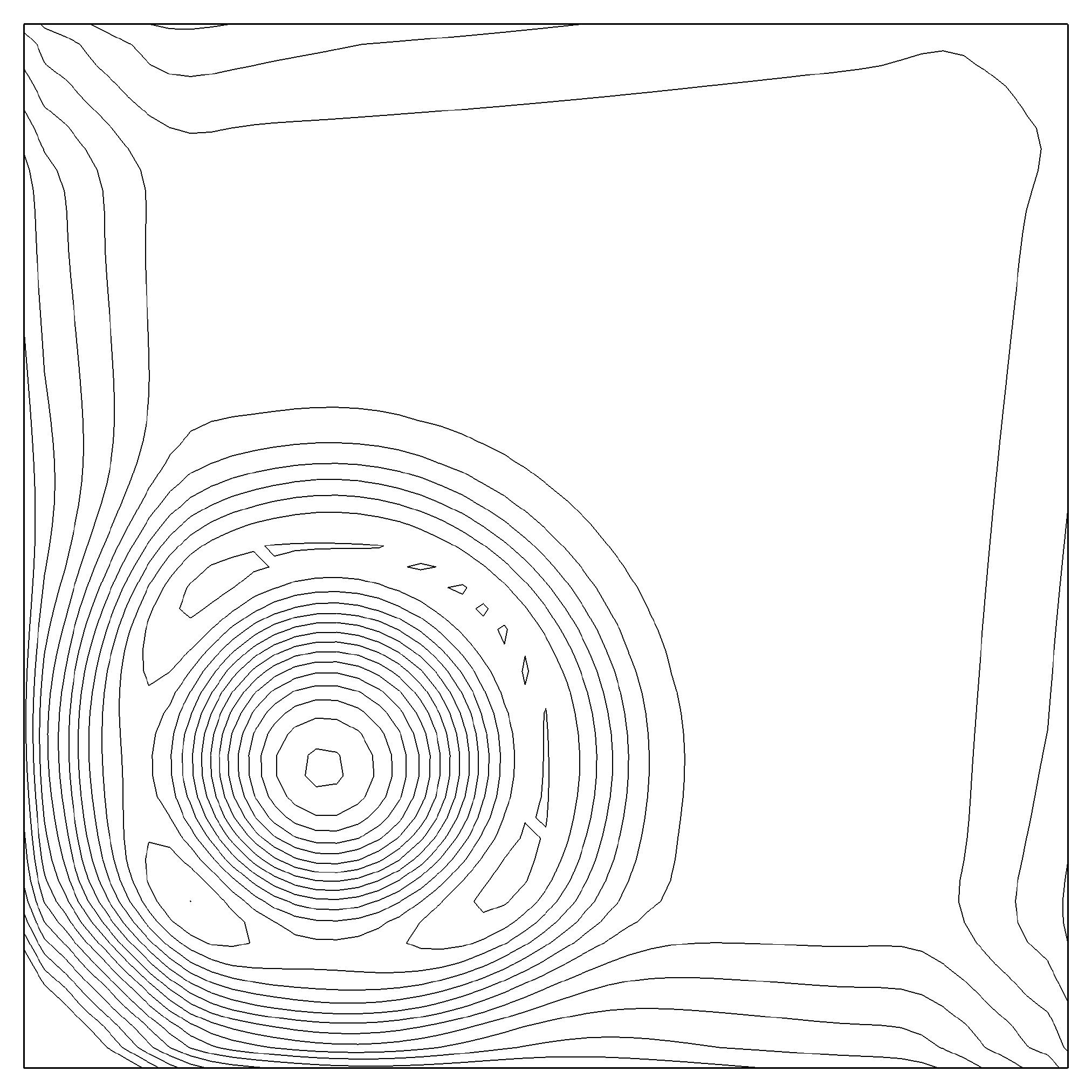} \\ 
			(a) Well-balanced scheme & (b) CRWENO5 & (c) Non well-balanced scheme
		\end{tabular}
		\caption{Pressure perturbations for test case defined in section (\ref{sec:2dpolypert})  for $\eta$ = 0.1. 20 equally spaced contour levels between -0.03 and +0.03}
		\label{fig:2dpolytropiceta1}
	\end{center}
\end{figure}

\begin{figure}
	\begin{center}
		\begin{tabular}{ccc}
			\includegraphics[width=0.30\textwidth]{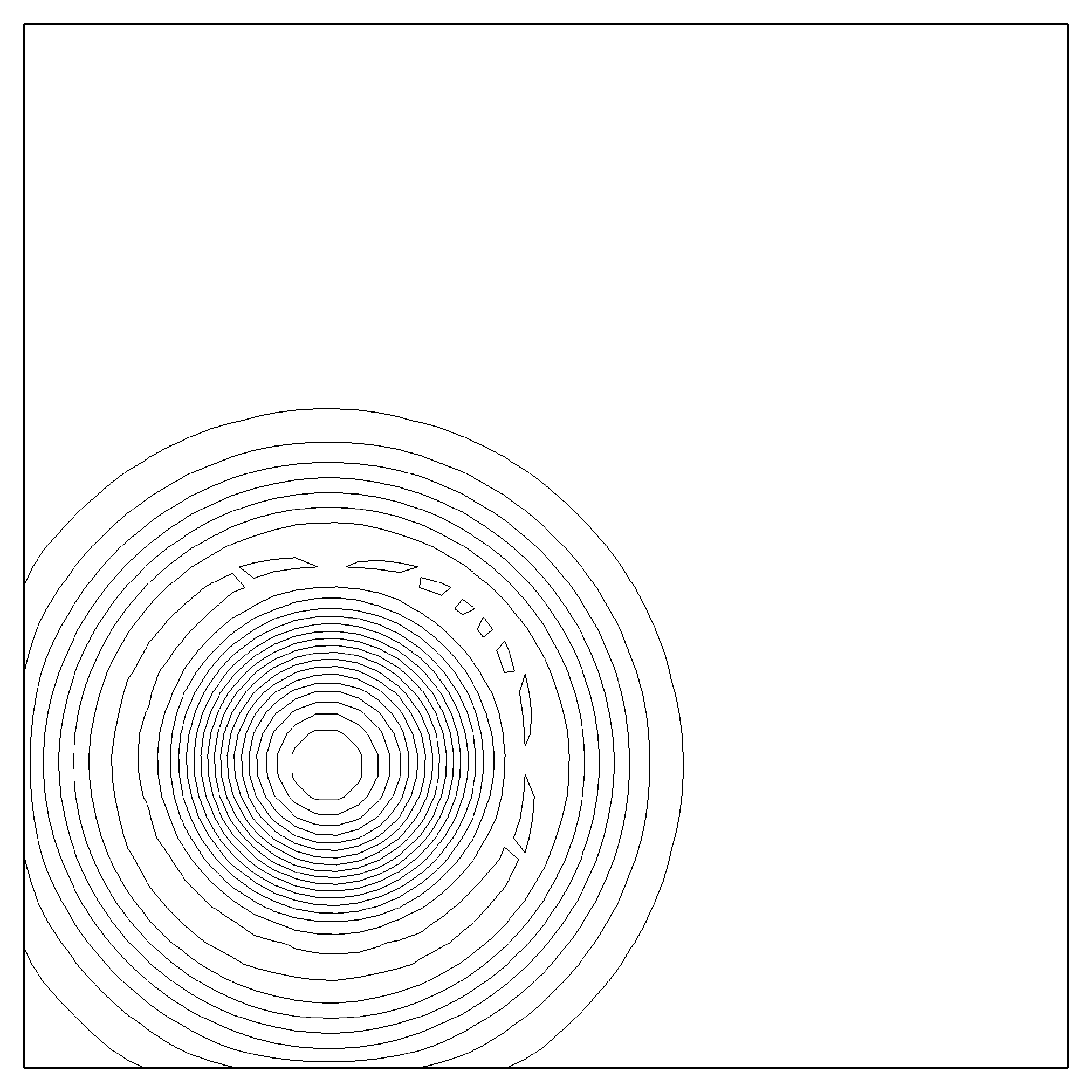} &
			\includegraphics[width=0.30\textwidth]{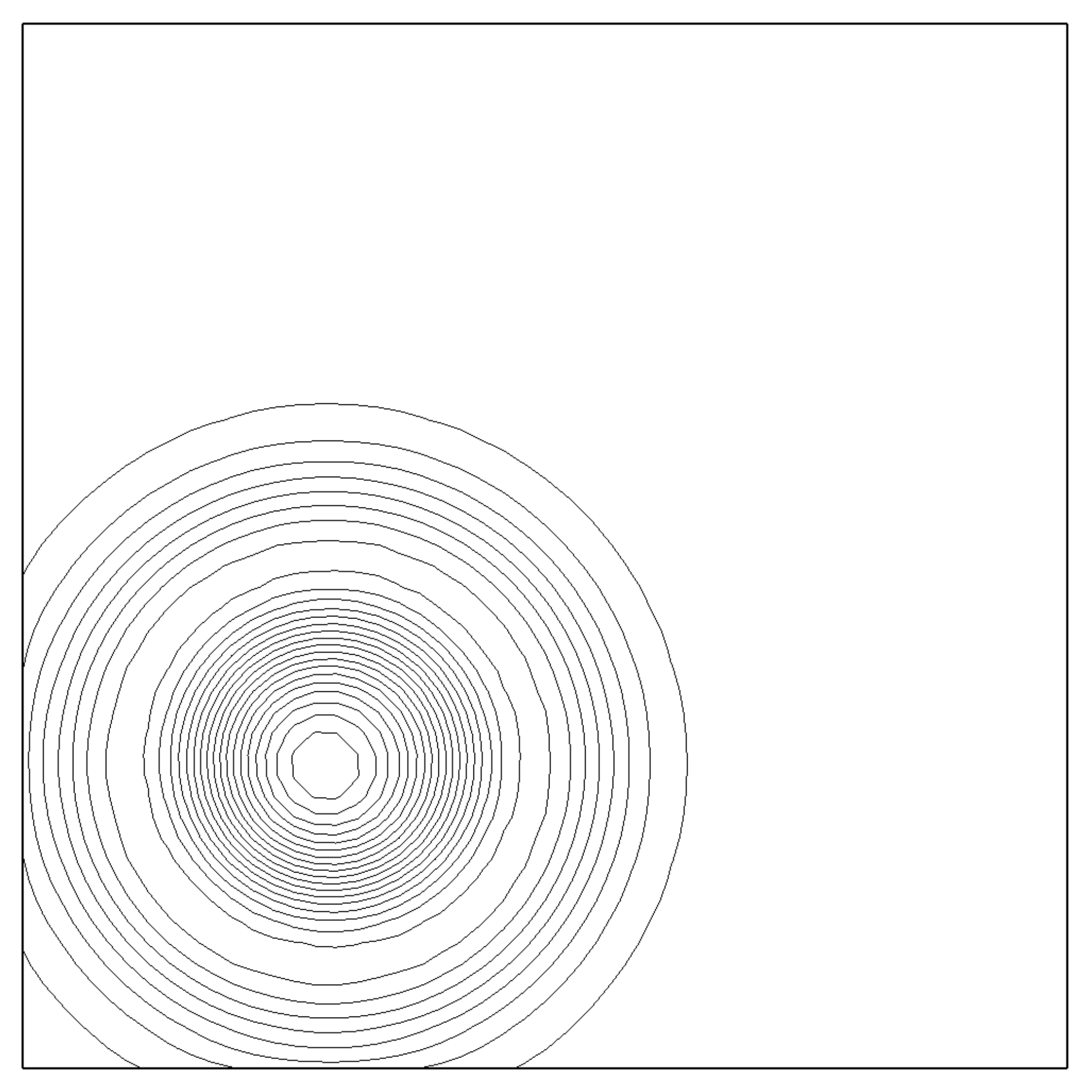} &
			\includegraphics[width=0.30\textwidth]{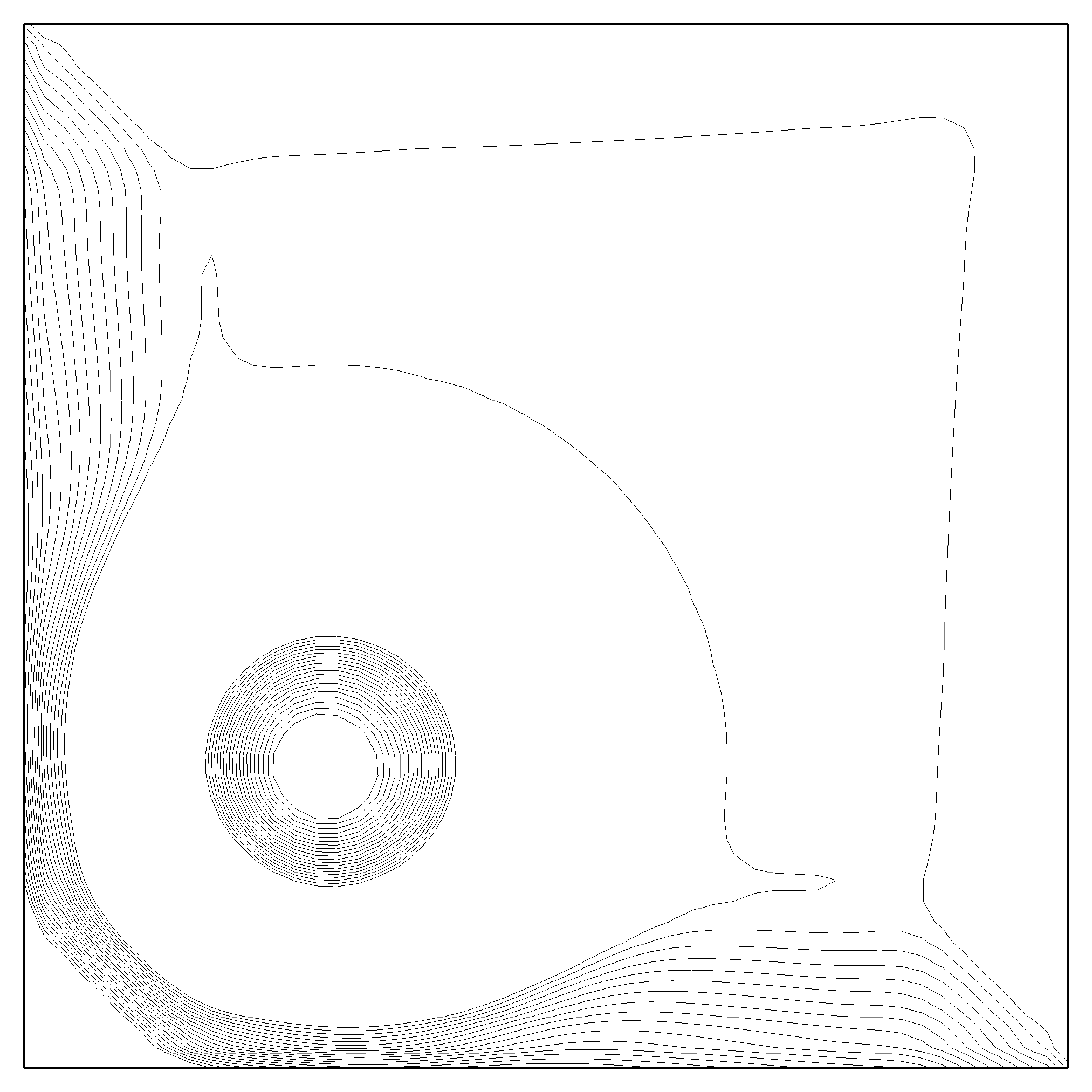} \\ 
			(a) Well-balanced scheme & (b) CRWENO5 & (c) Non well-balanced scheme
		\end{tabular}
		\caption{Pressure perturbations for test case defined in section (\ref{sec:2dpolypert})  for $\eta$ = 0.001. (a) and (b) 20 equally spaced contour levels between -0.00026 and +0.00026 , (b) 20 equally spaced contour levels between -0.02 and +0.00026 }
		\label{fig:2dpolytropiceta001}
	\end{center}
\end{figure}

\subsection{Order of accuracy}
\label{sec:2dorder}
This is a test case taken from~\cite{Li2016} and is used to determine the order of accuracy of the well-balanced scheme when applied to general solutions. An exact solution of the Euler equations with gravity is given by
\[
\rho(x, y, t) = 1 + 0.2\sin(\pi(\phi(x,y) - t(u_0 + v_0))), \qquad u(x,y,t) = u_0, \qquad v(x,y,t) = v_0
\]
\[
p(x,y,t) = p_0 + t(u_0 + v_0) - \phi(x,y) + 0.2\cos(\pi(\phi(x,y) - t(u_0 + v_0)))/\pi
\]
in the domain $[0,2] \times [0,2]$ with the gravitational potential taken as $\phi(x,y) = x + y$. The parameters for the flow are chosen to be $u_0 = v_0 = 1$ and $p_0 = 4.5$. Exact solution is applied at the boundary and ghost points for each time step and the $L_2$ norms of the error for the primitive variables are computed at final time of $t = 0.1$ for various grid resolutions. The error norms and convergence rates for each primitive variable is tabulated in table~(\ref{t:2dconvergence}). We observe that second order accuracy has been achieved for all the variables for the two-dimensional well-balanced scheme.
\begin{table}
\begin{center}
\begin{tabular}{|c|c|c|c|c|c|c|c|c|}
\hline
Cells & \multicolumn{2}{|c|}{$\rho$}& \multicolumn{2}{|c|}{$u$} & \multicolumn{2}{|c|}{$v$}  & \multicolumn{2}{|c|}{$p$}\\ \hline
 & Error & Rate & Error & Rate & Error & Rate & Error & Rate \\ \hline
100$\times$100 & 2.802E-6 & - & 4.464E-6 & - & 4.464E-6 & - & 9.183E-6 & - \\ \hline
200$\times$200 & 5.982E-7 & 2.23 &  1.114E-6 & 2.00 & 1.114E-6 & 2.00 &  2.308E-6 &  1.99 \\ \hline
400$\times$400 &  1.377E-7 & 2.12 &  2.779E-7 & 2.00 & 2.779E-7 & 2.00 & 5.781E-7 & 1.99 \\ \hline
800$\times$800 &  3.338E-8 & 2.05 & 6.939E-8 & 2.00 & 6.939E-8 & 2.00 & 1.447E-7 & 1.99 \\ \hline
\end{tabular}
\caption{Convergence of error for 2D case given in section (\ref{sec:2dorder})}
\label{t:2dconvergence}
\end{center}
\end{table}

\subsection{Radial gravitational field}
This test case consists of a radial gravitational field and an isothermal hydrostatic solution given by $\rhoe = \exp(-\phi(r))$ and $\pe = \exp(-\phi(r))$, 
where $\phi(r) = r$. We conduct a test to verify the well-balanced property of the scheme on a Cartesian mesh. The test domain is $[-1,+1] \times [-1,+1]$ with transmissive boundary conditions on all sides and the simulation is performed upto a final time of $t = 1.0$ units. The errors in $L_1$ norm relative to the initial condition is computed at final time and is presented in table~(\ref{t:radialwb}). We can clearly see that the $L_1$ norm of the error is approximately equal to machine precision for all the grid resolutions tested showing that the scheme is well-balanced.
\begin{table}
 \begin{center}
 \begin{tabular}{|c|c|c|c|c|}
 \hline
Cells & $\rho$  & $u$ & $v$ & $p$ \\\hline
50$\times$50 & 8.203E-16 &   6.179E-16  & 6.179E-16 &  1.164E-15\\ \hline
100$\times$100  & 1.721E-15  & 1.119E-15  & 1.119E-15  & 2.505E-15\\ \hline
 200$\times$200  & 3.610E-15  & 1.695E-15   &1.695E-15 &  5.289E-15  \\ \hline
 \end{tabular}
\caption{$L_1$ error norms for primitive variables for a radial isothermal hydrostatic solution}
 \label{t:radialwb}
 \end{center}
\end{table}
Figure~(\ref{fig:radialwbtest}) shows a comparison between the density contours at time $t$ = 1.5, obtained using the well-balanced scheme and a non well-balanced scheme in which the source terms are calculated using the exact value of the gravitational force $\nabla \phi$. We see that the well-balanced scheme preserves the initial condition, while large errors can be observed for the non well-balanced scheme. 
\begin{figure}
\begin{center}
\begin{tabular}{ccc}
\includegraphics[width=0.32\textwidth]{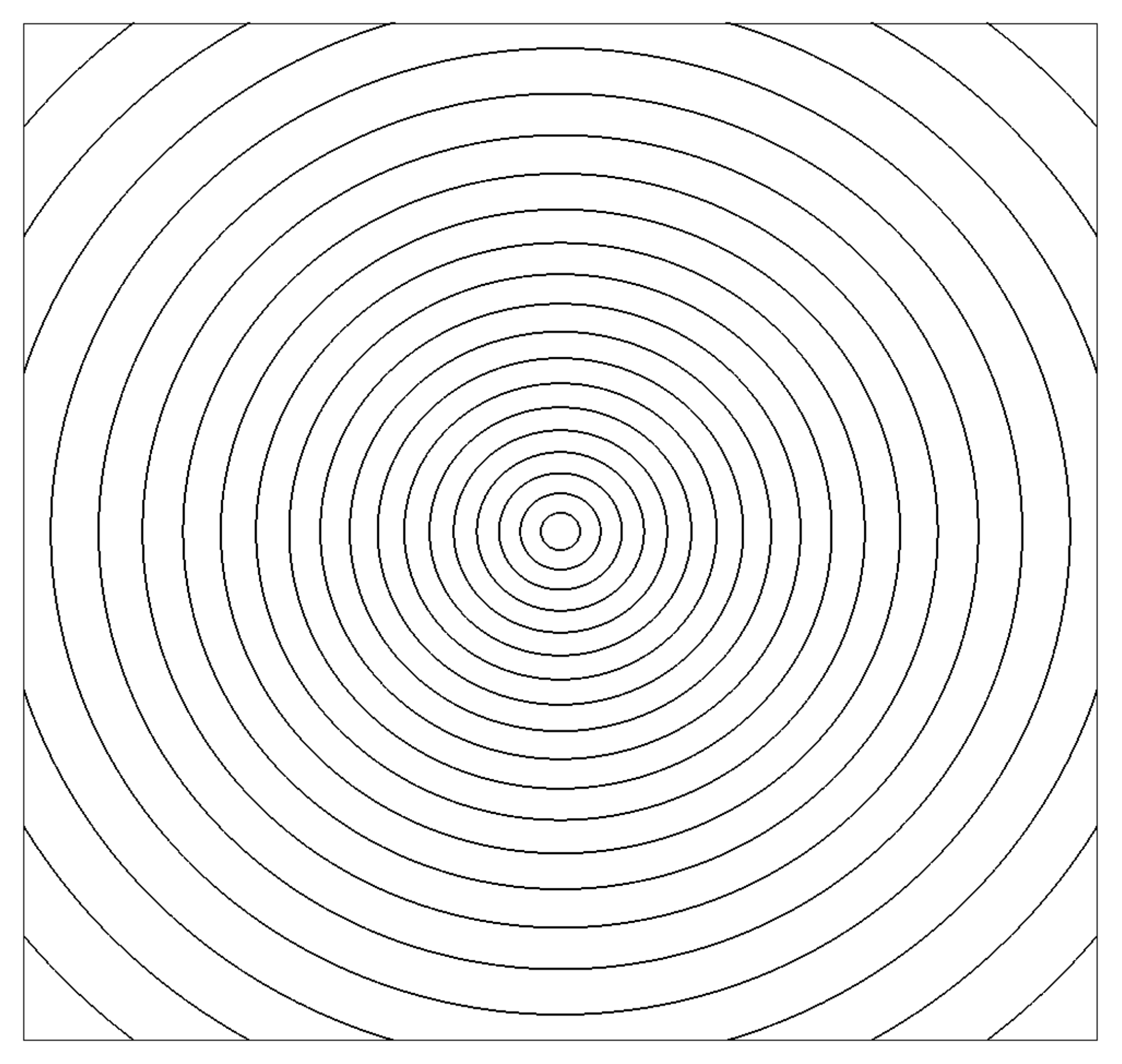} &
\includegraphics[width=0.32\textwidth]{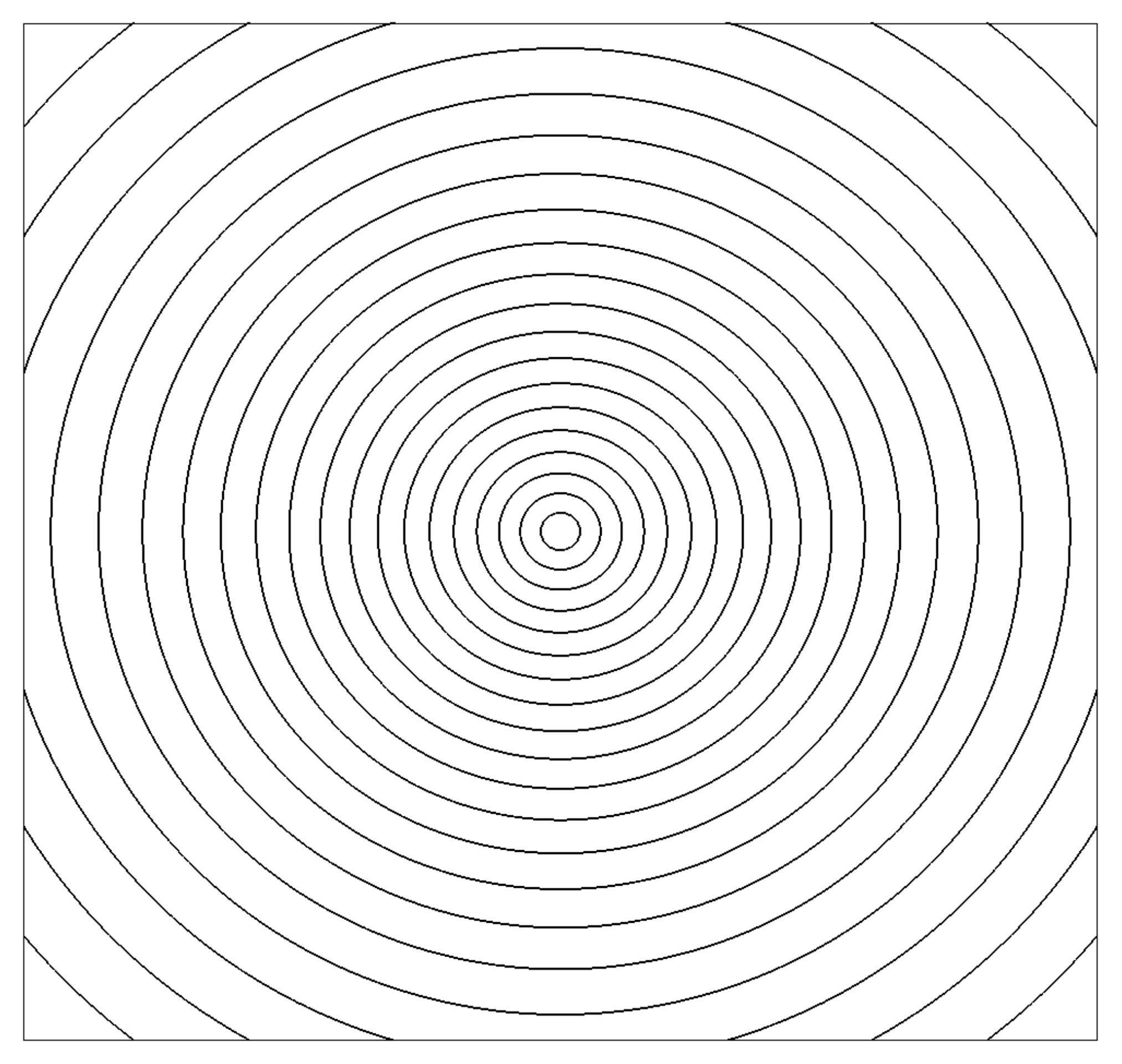} &
\includegraphics[width=0.32\textwidth]{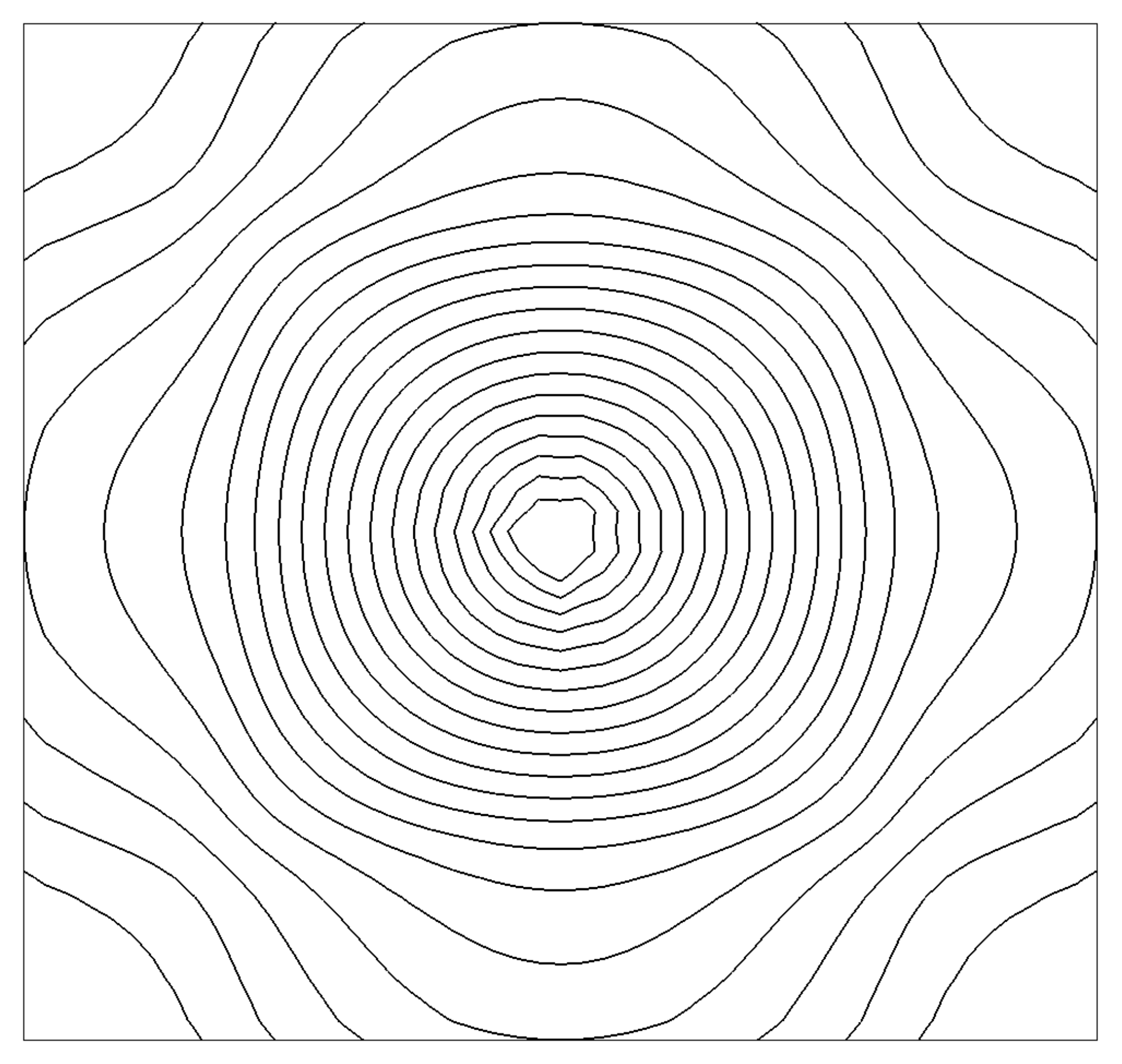} \\
(a) & (b)  & (c) 
\end{tabular}
\caption{Well-balanced test for radial hydrostatic problem on $51 \times 51$ grid; Density contours: (a) initial solution (b) well-balanced scheme: at $t = 1.5$, (c) non well-balanced scheme : at $t = 1.5$  }
\label{fig:radialwbtest}
\end{center}
\end{figure}

We next perform well-balanced tests for a polytropic hydrostatic solution given by
$
\rhoe = \left(1 - \frac{\nu - 1}{\nu}\phi\right)^{\frac{1}{\nu - 1}}$ and $\pe = \rhoe^{\nu}
$ where $\nu = 1.2$. The simulations are         performed upto a final time of $t=50$ units for grids of sizes $51\times51$ and $101\times101$. Transmissive conditions are imposed at the boundaries. Evolution of errors in    $L_1$ norm for the primitive variables are plotted in figure~(\ref{fig:radialpolyerrevolution}). Since the scheme is not exactly well-balanced for the  test case due to radial gravitational field and Cartesian mesh, errors of the order $\sim10^{-6}$ is observed for all the variables  at time $t=50$ units. However, using a non well-balanced scheme causes the simulation to become unstable at time $t=7$ units, suggesting that the present scheme is still useful in such simulations. \\
\begin{figure}
\begin{tabular}{cc}
\includegraphics[width=0.5\textwidth]{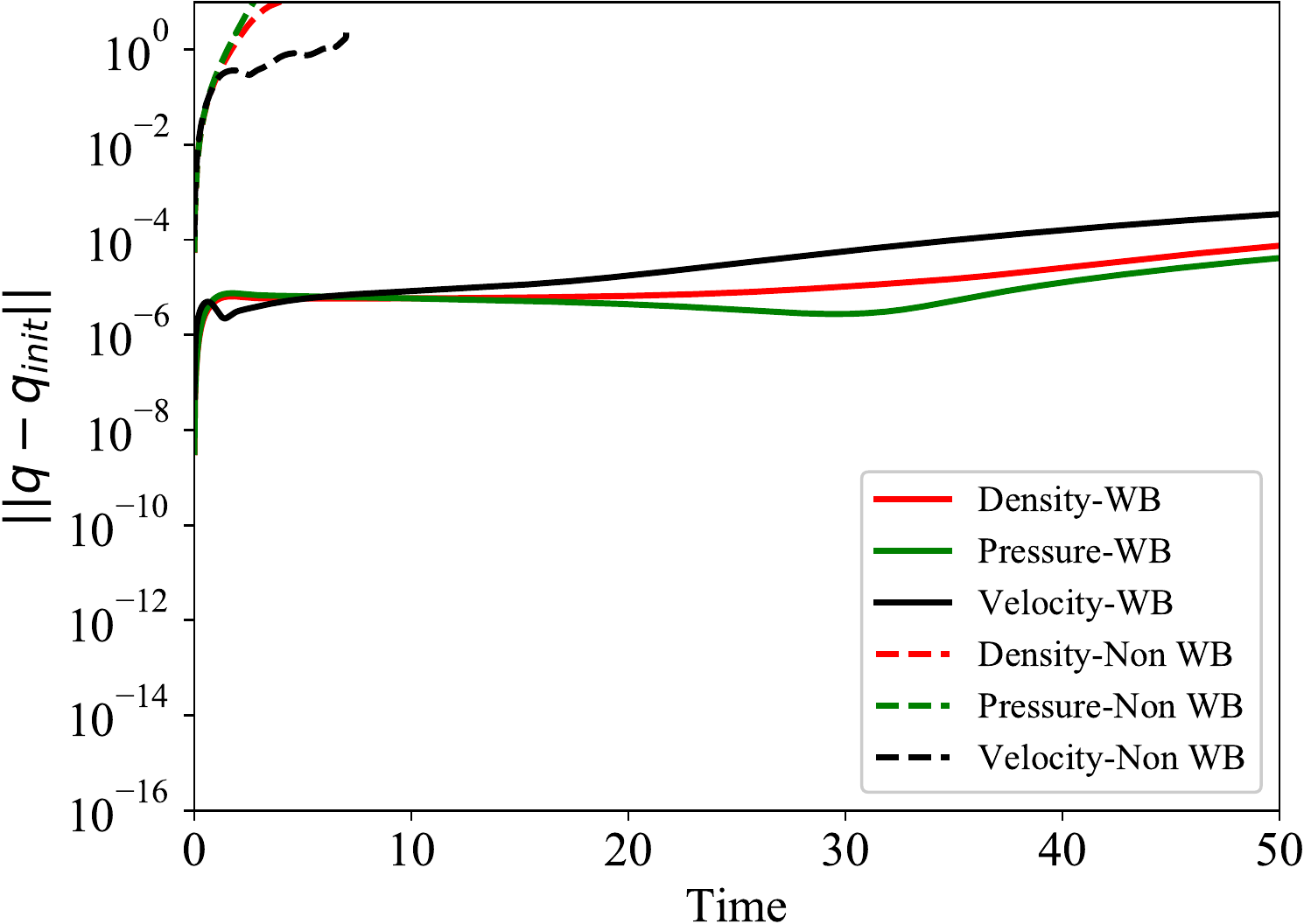} &
\includegraphics[width=0.5\textwidth]{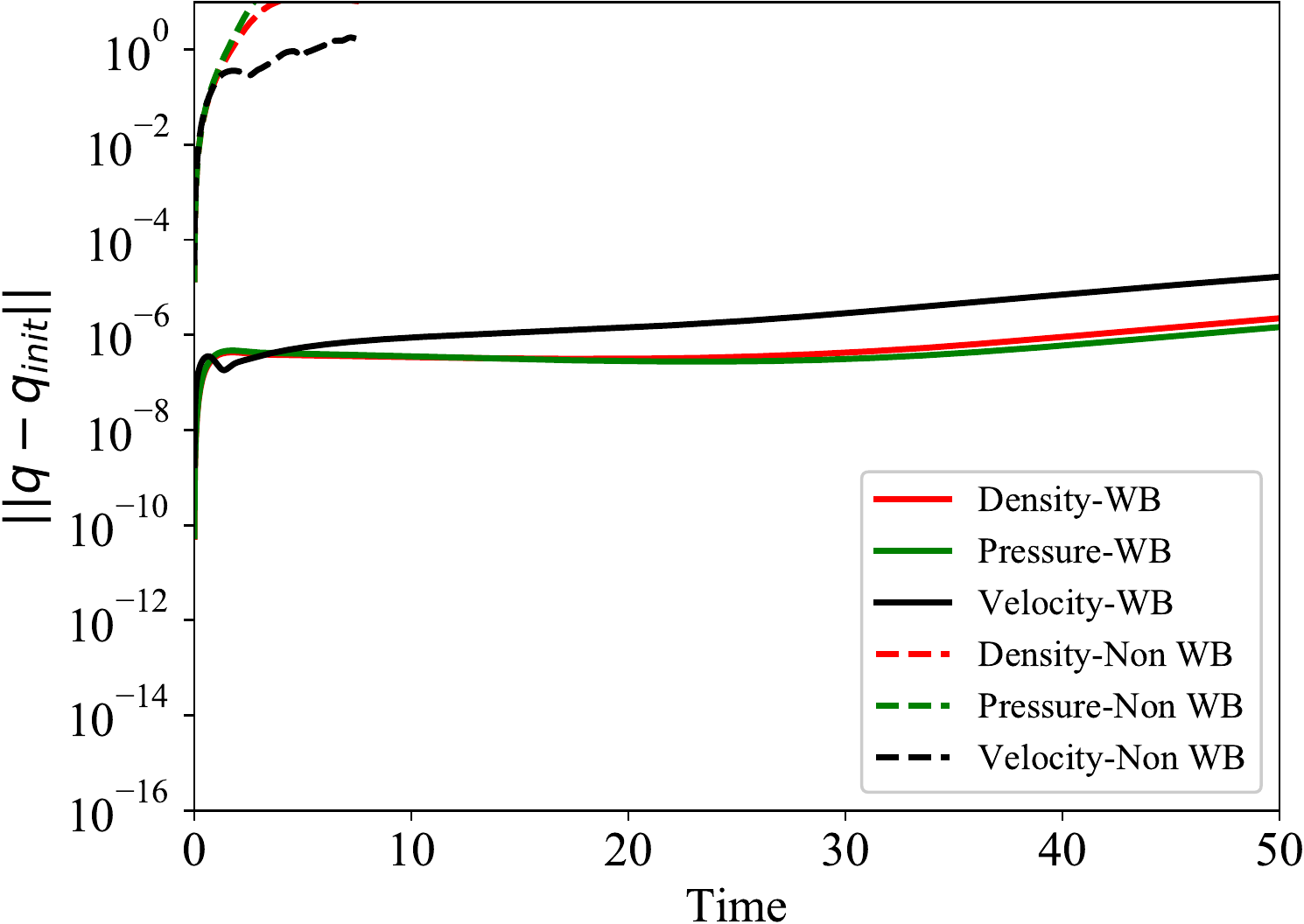} \\
(a)  & (b)  
\end{tabular}
\caption{Evolution of $L_1$ error norms for primitive variables for polytropic  initial conditions : (a) $51\times51$ (b) $101\times101$  }
\label{fig:radialpolyerrevolution}
\end{figure}
Now, we perform well-balanced tests for a van der Waals equation where the initial hydrostatic solution is obtained using the method described in section~\ref{sec:vdWhydro}. The equation~(\ref{eqn:hydrovanode}) is written for the radial case as
	\[
	\frac{d\rho}{dr} +   \frac{1}{f(\rho)}\frac{d\phi}{dr} = 0    \qquad \text{where} \qquad f(\rho) = \frac{MR_uT}{\rho(M - \rho b)^2} - \frac{2a}{M^2}
	\]
As with the 1-D case, we use $a = 0.4$ and $b = 0.001$ with the left boundary condition for $\rho$ taken as 1.0. The universal gas constant and molar mass are taken to be 1.0 and the above initial value problem is solved over a domain [0,$\sqrt{2}$]. This solution is interpolated onto a Cartesian mesh in a square domain [-1,1] to obtain our 2D problem. A grid size of 201$\times$201 is used with transmissive conditions imposed on all boundaries of the square domain. We perform numerical simulations for a large final time and compare the evolution of $L_1$ errors of the primitive variables with that obtained using a non well-balanced scheme. It should be noted that the current scheme would not be able to exactly preserve the well-balanced property even for the second-order discrete approximation of the exact solution due to misalignment of the grid with respect to the gravitational field. Therefore, $L_1$ errors of the order of $10^{-7}$ are observed for all the primitive variables using the well-balanced scheme even at a large time of $T$ = 250. However, from figure~(\ref{fig:radialvandererrevolution}a), we can observe that the non well-balanced scheme becomes unstable almost immediately after the start  of the simulation (i.e at $t$ = 0.03 units). Therefore, as with the 1-D case described in section~(\ref{sec:vdWhydro}), we observe that for van der Waals equation of state, the numerical solution is heavily reliant on the well-balanced property of the scheme.
	\begin{figure}
		\begin{tabular}{cc}
			\includegraphics[width=0.48\textwidth]{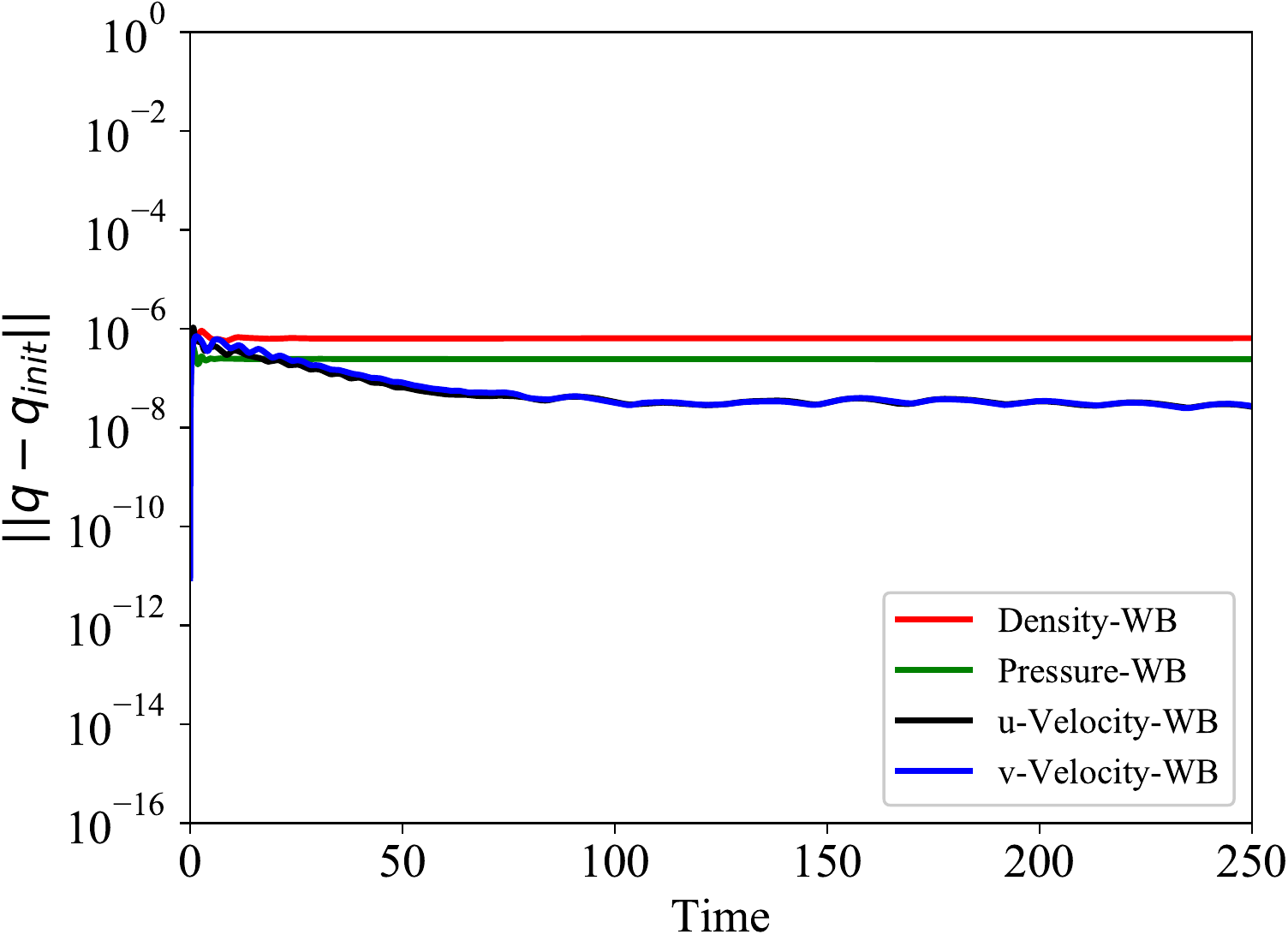} &
			\includegraphics[width=0.48\textwidth]{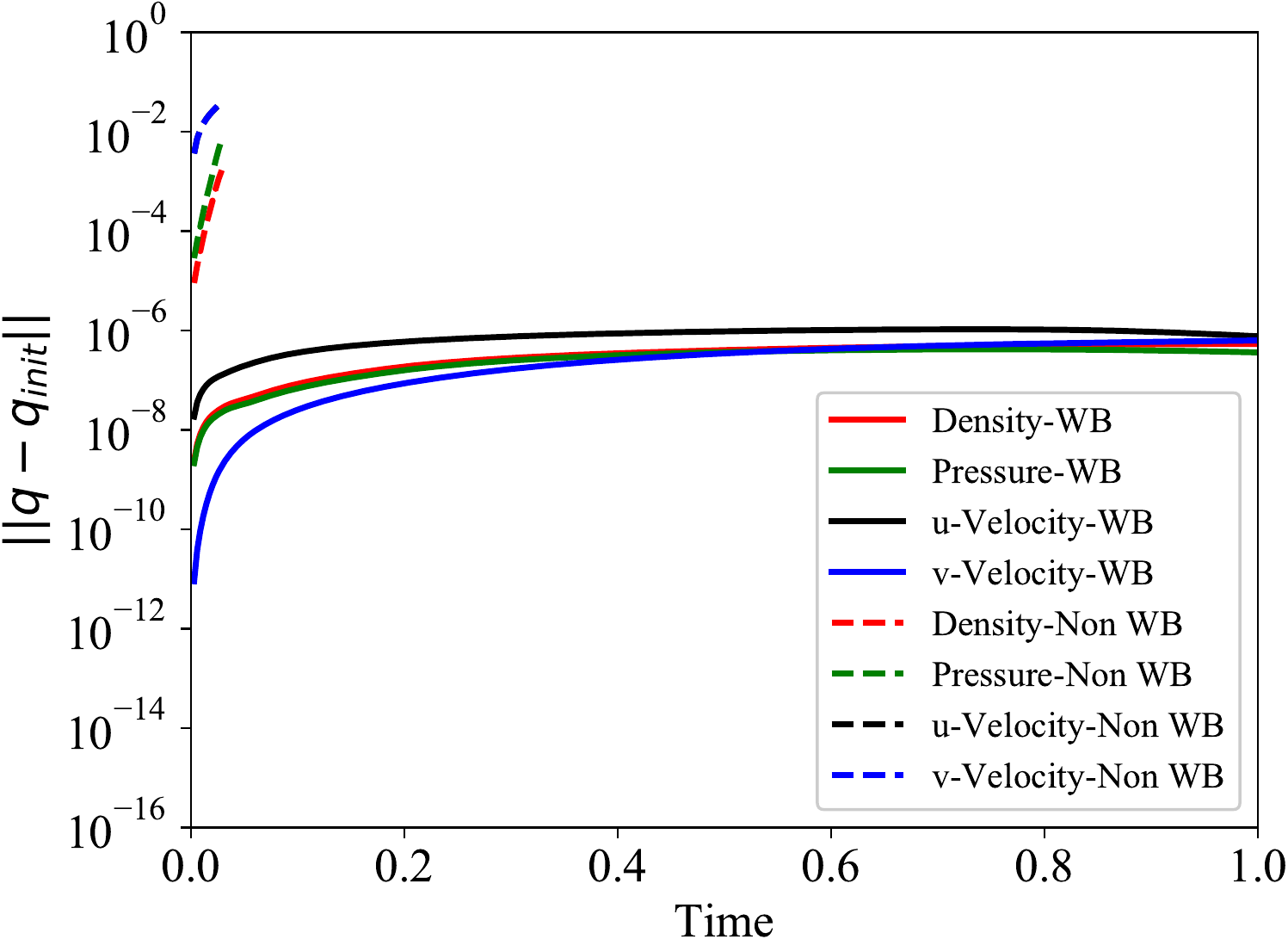} \\
			(a)  & (b)  
		\end{tabular}
		\caption{Evolution of $L_1$ error norms for primitive variables for 2D van der Waals case : (a)  WB, final time of $t$ = 250 units (b)WB v/s Non-WB }
		\label{fig:radialvandererrevolution}
	\end{figure}

\subsection{Radial Rayleigh-Taylor instability}
The initial condition consists of an isothermal hydrostatic solution which is perturbed along a wavy interface which then develops instabilities due to the radial gravitational field~\cite{LeVeque1999}. The gravitational potential is $\phi(r) = r$ and the initial condition is given by
\[
p = \begin{cases}
e^{-r}, &  r \leq r_0 \\
e^{-\frac{r}{\alpha} + r_0\frac{1 - \alpha}{\alpha}}, & r > r_0 
\end{cases}
\qquad
\rho = \begin{cases}
e^{-r}, &  r \leq r_i \\
\frac{1}{\alpha} e^{-\frac{r}{\alpha} + r_0\frac{1 - \alpha}{\alpha}}, & r > r_i 
\end{cases}
\] 
Here, $r_i = r_0 (1 + \eta \cos(k\theta))$ and $\alpha = \exp(-r_0)/(\exp(-r_0) + \Delta _{\rho})$ and the constants are $r_0 = 0.6$, $\eta = 0.02$, $\Delta _{\rho} = 0.1$ and $k = 20$. Simulations are performed up to a final time of $t = 5$ on the domain $[-1,+1] \times [-1,+1]$ with a mesh of $241 \times 241$ points.
\begin{figure}
\begin{center}
\begin{tabular}{cc}
\includegraphics[width=0.35\textwidth]{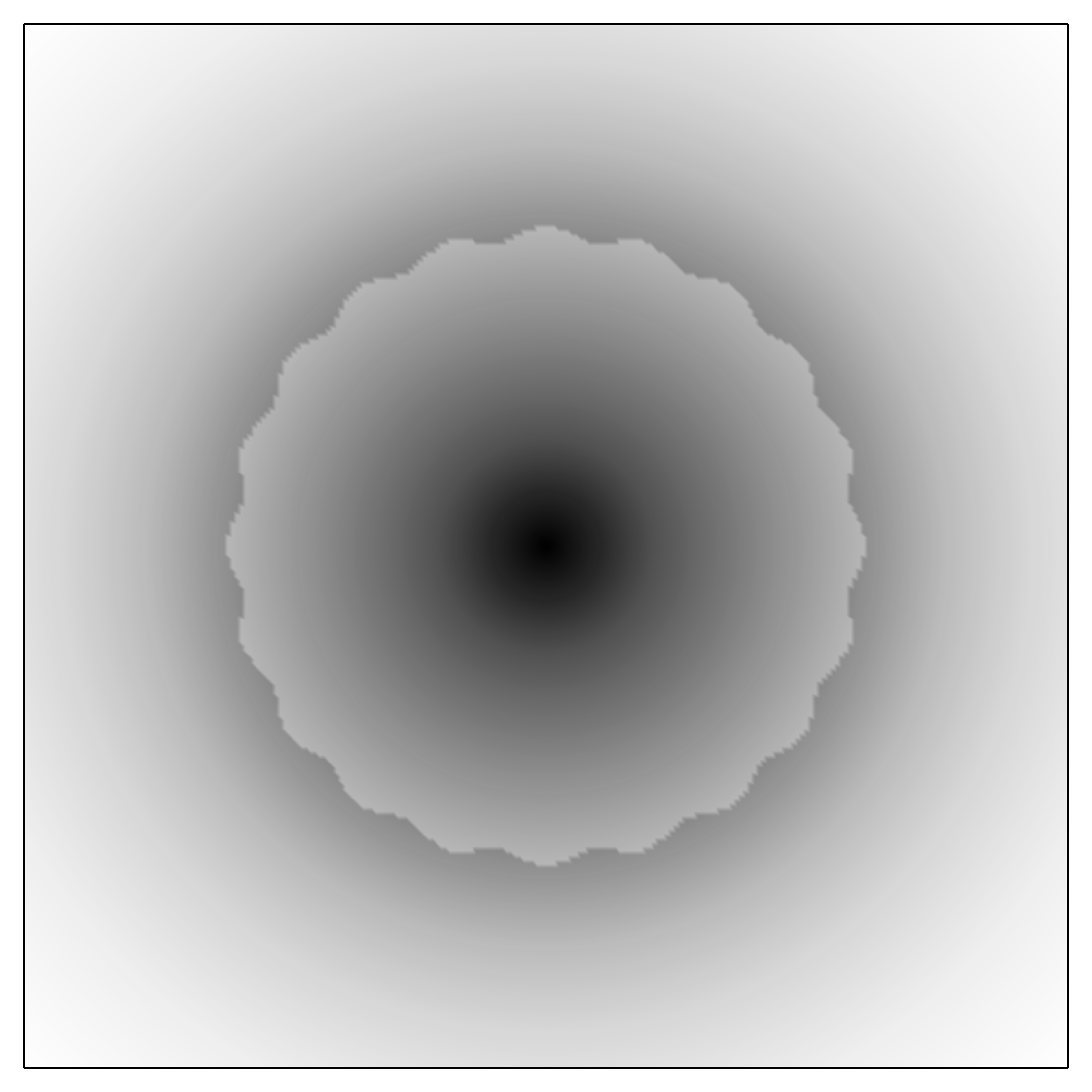} &
\includegraphics[width=0.35\textwidth]{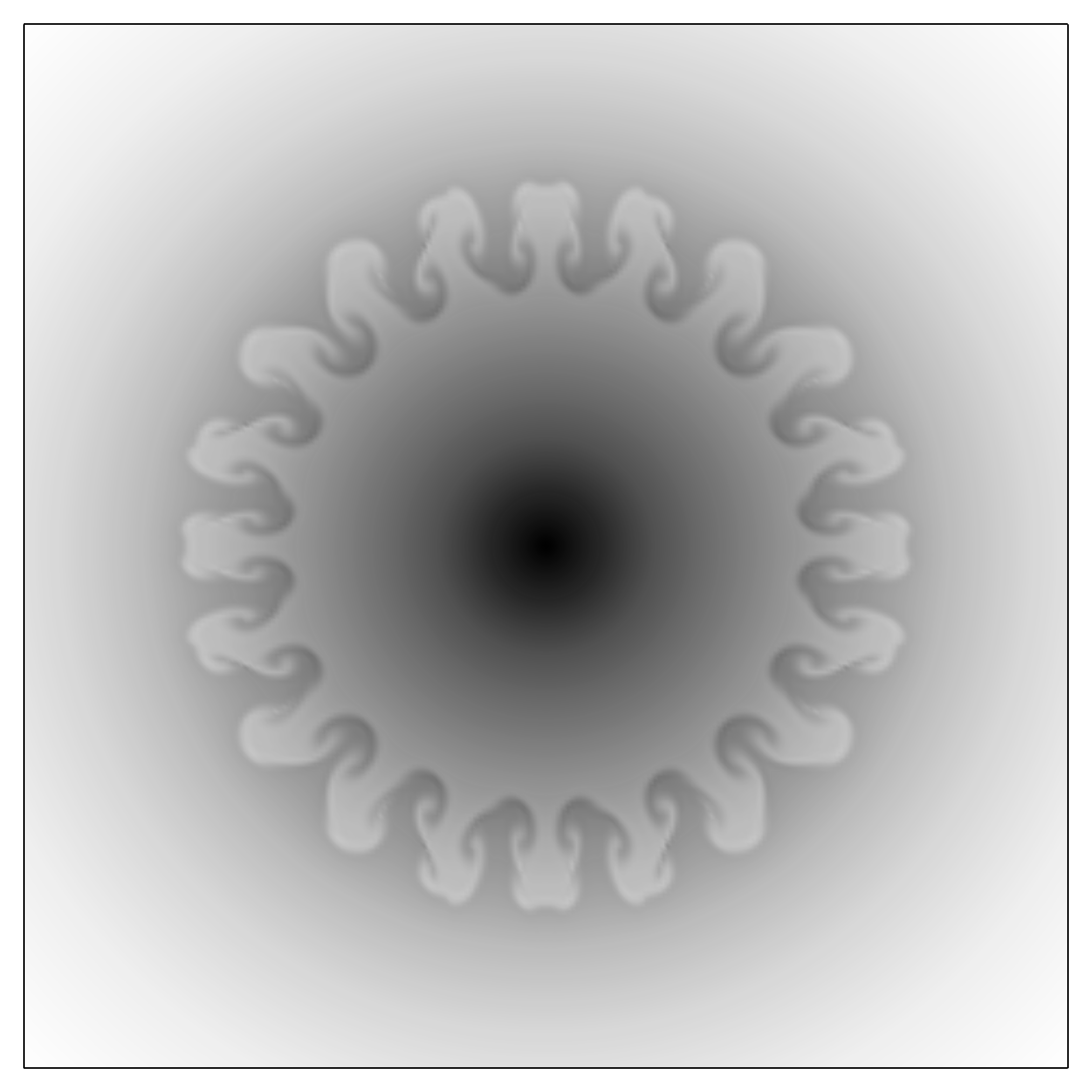} \\
(a)  & (b)  \\
\includegraphics[width=0.35\textwidth]{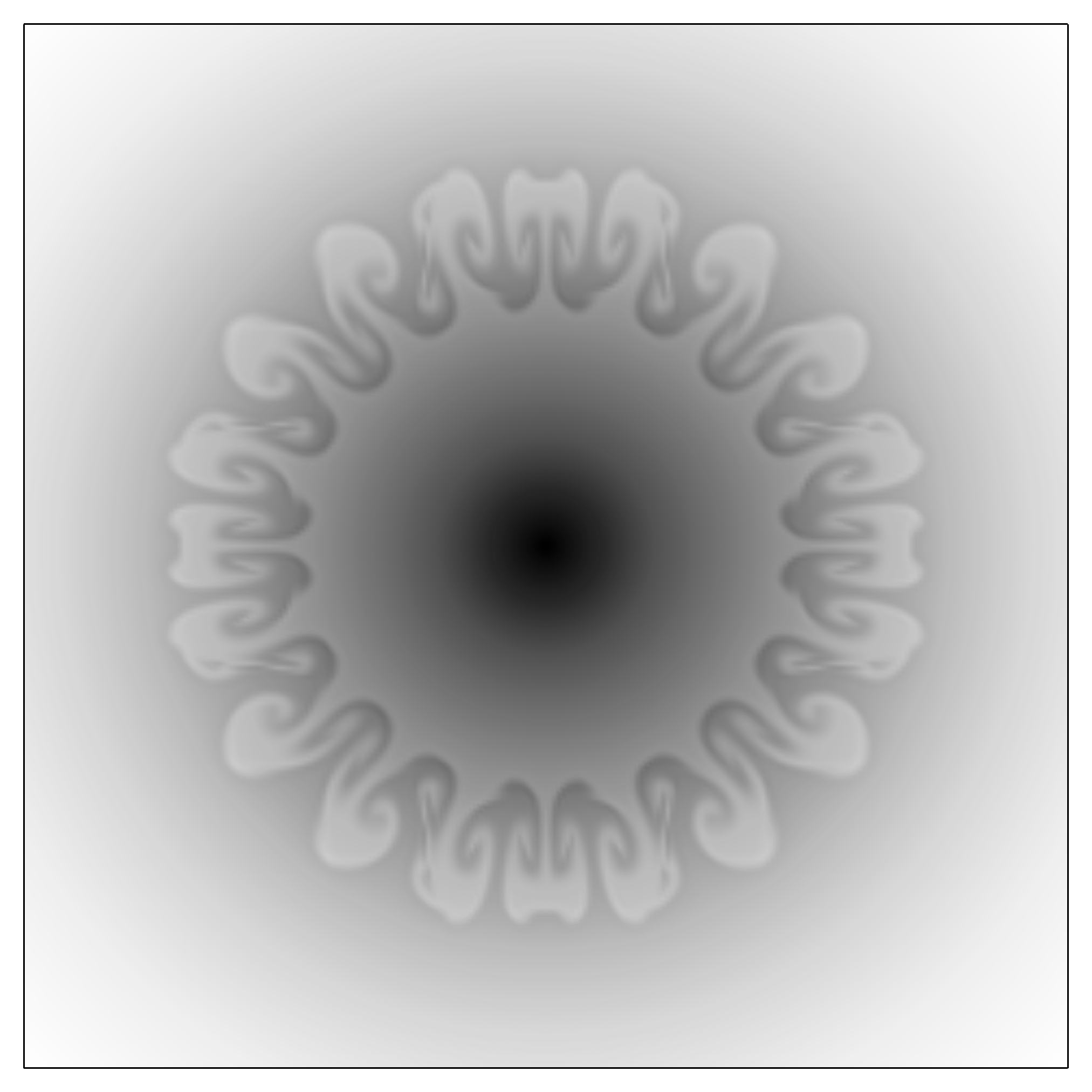} &
\includegraphics[width=0.35\textwidth]{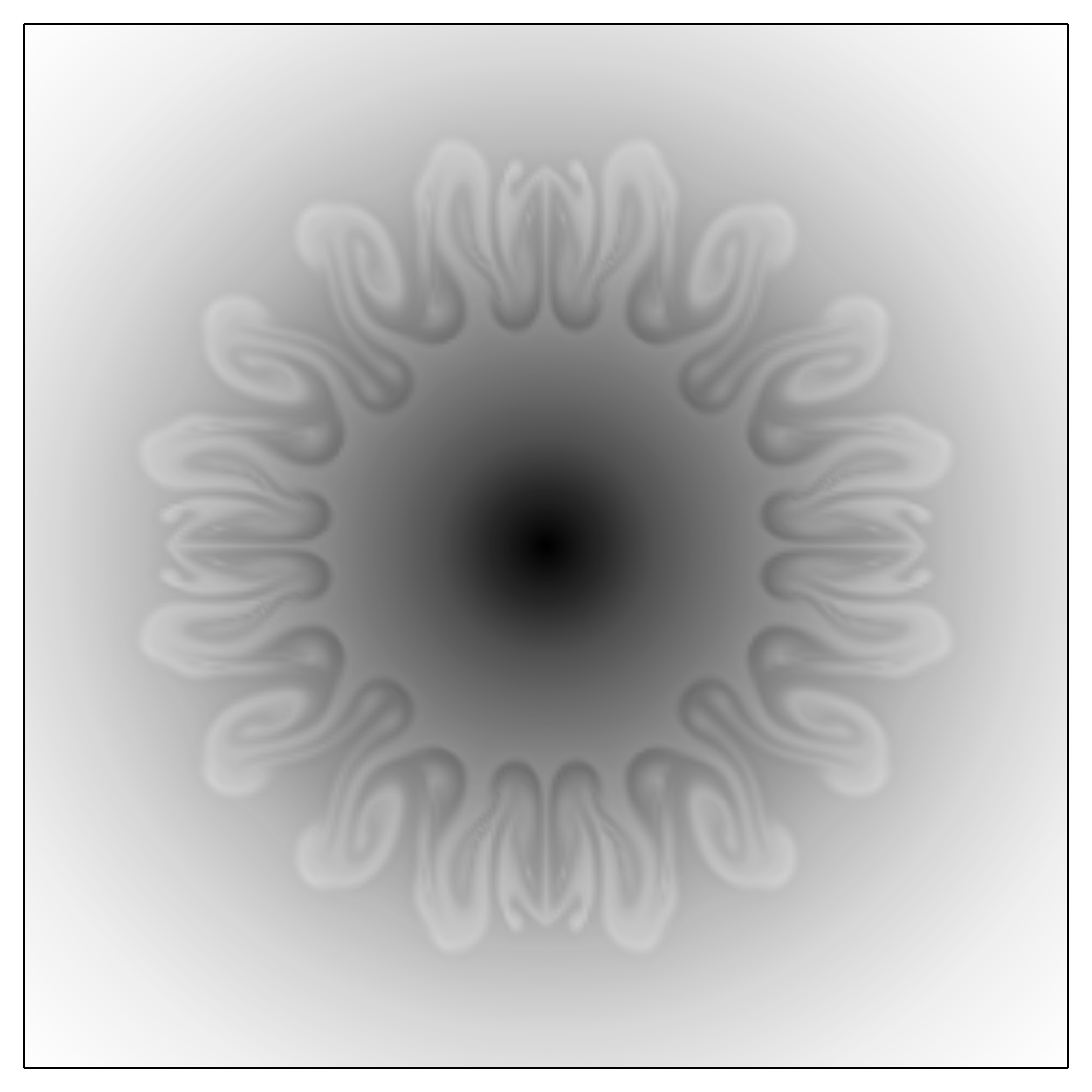} \\
(c)  & (d) \\
\end{tabular}
\caption{Radial Rayleigh-Taylor instability. Density plots at times (a) $t$ = $0$, (b) $t$ = $2.9$,  (c) $t$ = $3.8$, (d) $t$ = $5.0$. }
\label{fig:rayleightaylor}
\end{center}
\end{figure}
Figure~(\ref{fig:rayleightaylor}) shows the solutions for density at various times. A non well-balanced scheme where the source terms are computed using the exact values of the gravitational force is also used to perform the simulation. The solution (not shown here) disintegrates at around time $t$ = 2.5 due to the generation of large errors away from the discontinuous density interface at $r$ = $r_0$. In contrast, the well-balanced scheme produces relatively accurate results with the instabilities concentrated near the location of the interface while  retaining equilibrium solution away from it.

\subsection{Rising thermal bubble}

The test case~\cite{Giraldo2008} simulates the rising of a hot bubble due to buoyancy which is immersed in a hydrostatic solution given by
\[
\pe = p_0 \left(1 - \frac{(\gamma - 1)gy}{\gamma R \theta_0} \right)^{\frac{\gamma}{\gamma - 1}}, \qquad \rhoe = \rho _0 \left(1 - \frac{(\gamma - 1)gy}{\gamma R \theta_0}\right)^{\frac{1}{\gamma - 1}}
\]
where $p_0 = 10^5 \ N/m^2$, $\gamma=1.4$, $g = 9.8 \ m/s^2$ $R = 287.058 \ J/(kg \  K)$, $\theta_0 = 300 \ K$, $\rho_0 = p_0/(R \theta_0)$, and the gravitational potential is $\phi = gy$. The domain is $1000 \ m \times 1000 \ m$ with solid wall boundary conditions on all sides. The warm bubble is introduced as a potential temperature perturbation which involves replacing $\theta_0$ with $\theta_0 + \Delta \theta (x,y)$ in the density equation, where
\[
\Delta \theta (x,y) = \begin{cases}
0 & \textrm{if } r > r_c\\
\frac{\theta_c}{2}[1 + \cos(\frac{\pi r}{r_c})] & \textrm{if } r <  r_c
\end{cases} 
\quad \textrm{where} \quad
r = \sqrt{(x - x_c)^2 + (y - y_c)^2}
\]
Here, $\theta_c = 0.5 \ K$ is the perturbation strength, the initial location of the perturbation $(x_c, y_c) = (500 \ m, 350 \ m)$, and the radius of the bubble $r_c = 250 \ m$. The simulations are performed upto a final time of 700 s. Using a grid of $401\times401$ points, contour plots for the ascension of the warm bubble are drawn over the prescribed period of time and is shown in figure~(\ref{fig:bubblecontours}). The bubble rises due to buoyancy effects since the density inside the bubble is smaller than the ambient density and the interface deforms and rolls up due to physical instabilities. The solution is similar to that observed from previous studies~\cite{Ghosh2016} which use a fifth order accurate CRWENO5 scheme.  Figures~(\ref{fig:bubblecrosssection}) shows a comparison of line plots of potential temperature perturbations along cross-sectional lines at  $x = 500 \ m$ and $y = 720 \ m$ between CRWENO5 and the present scheme for two grids of sizes $201\times201$ and $401\times401$. We can observe that the solution with the present well-balanced scheme compares favourably well with the CRWENO5 scheme, considering that our scheme is only second order accurate.  
\begin{figure}
\begin{center}
\begin{tabular}{cc}
\includegraphics[width=0.48\textwidth]{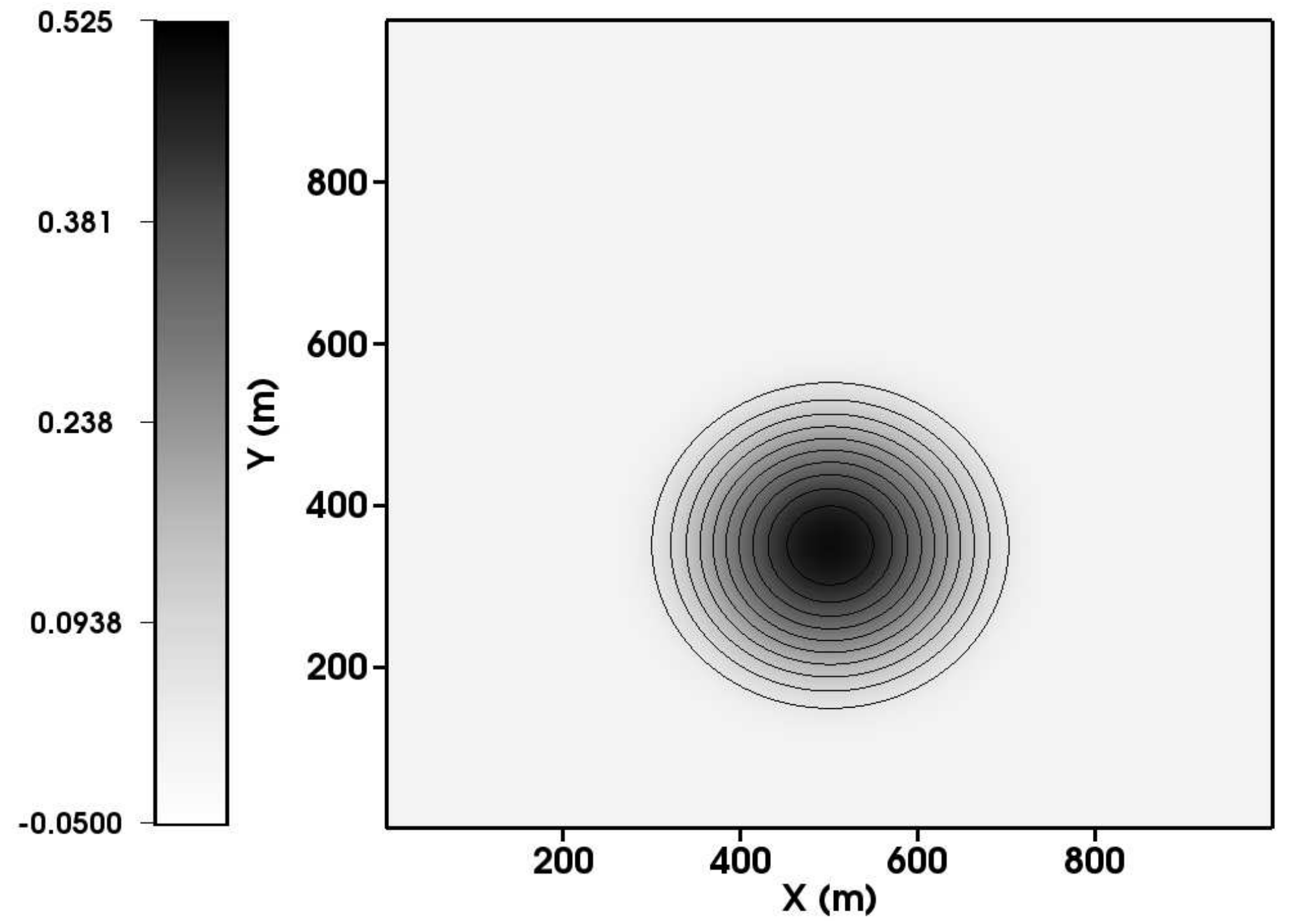} &
\includegraphics[width=0.48\textwidth]{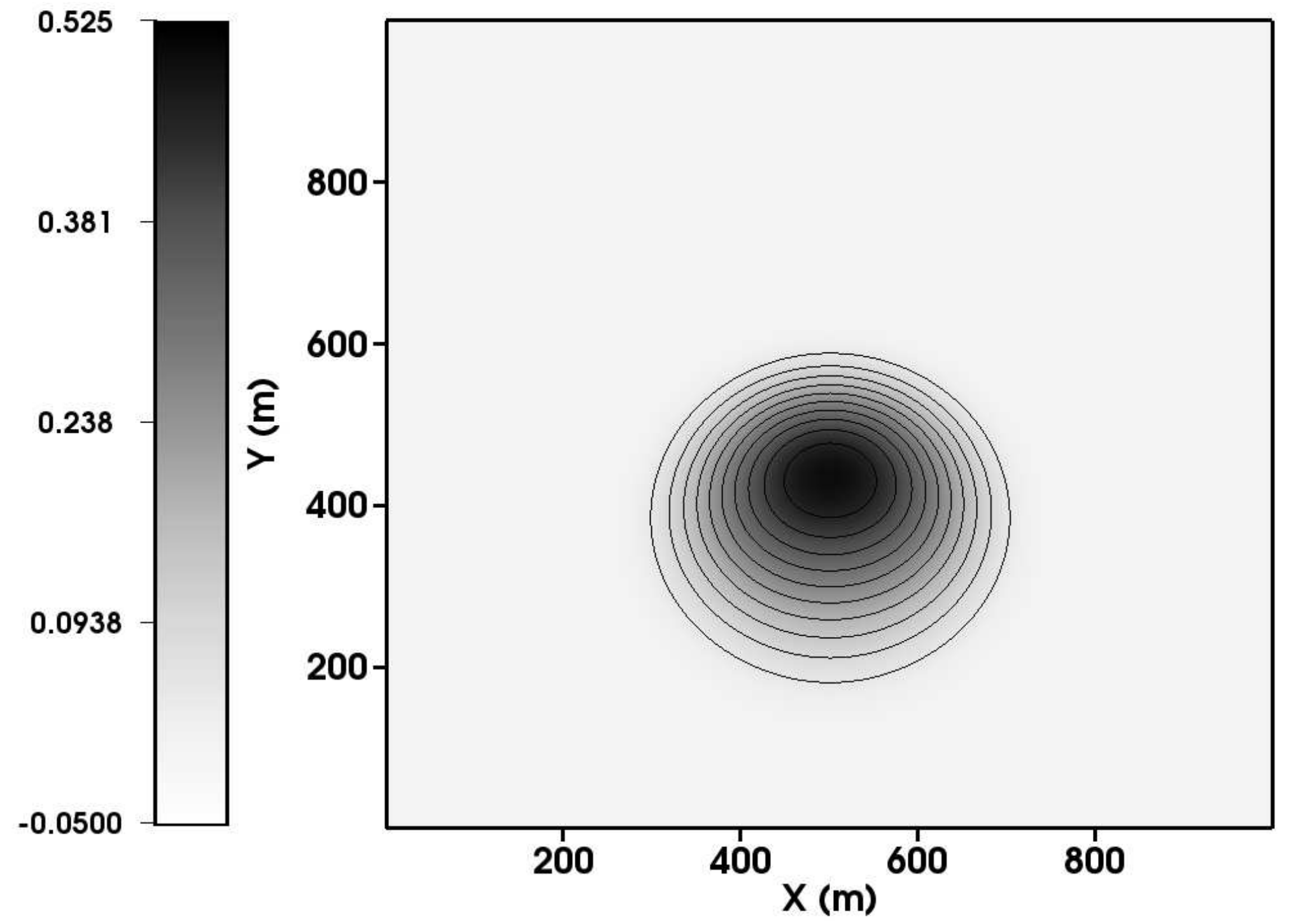} \\
(a) $ t = 0 $ $s$ & (b) $t = 150 $ $s$ \\
\includegraphics[width=0.48\textwidth]{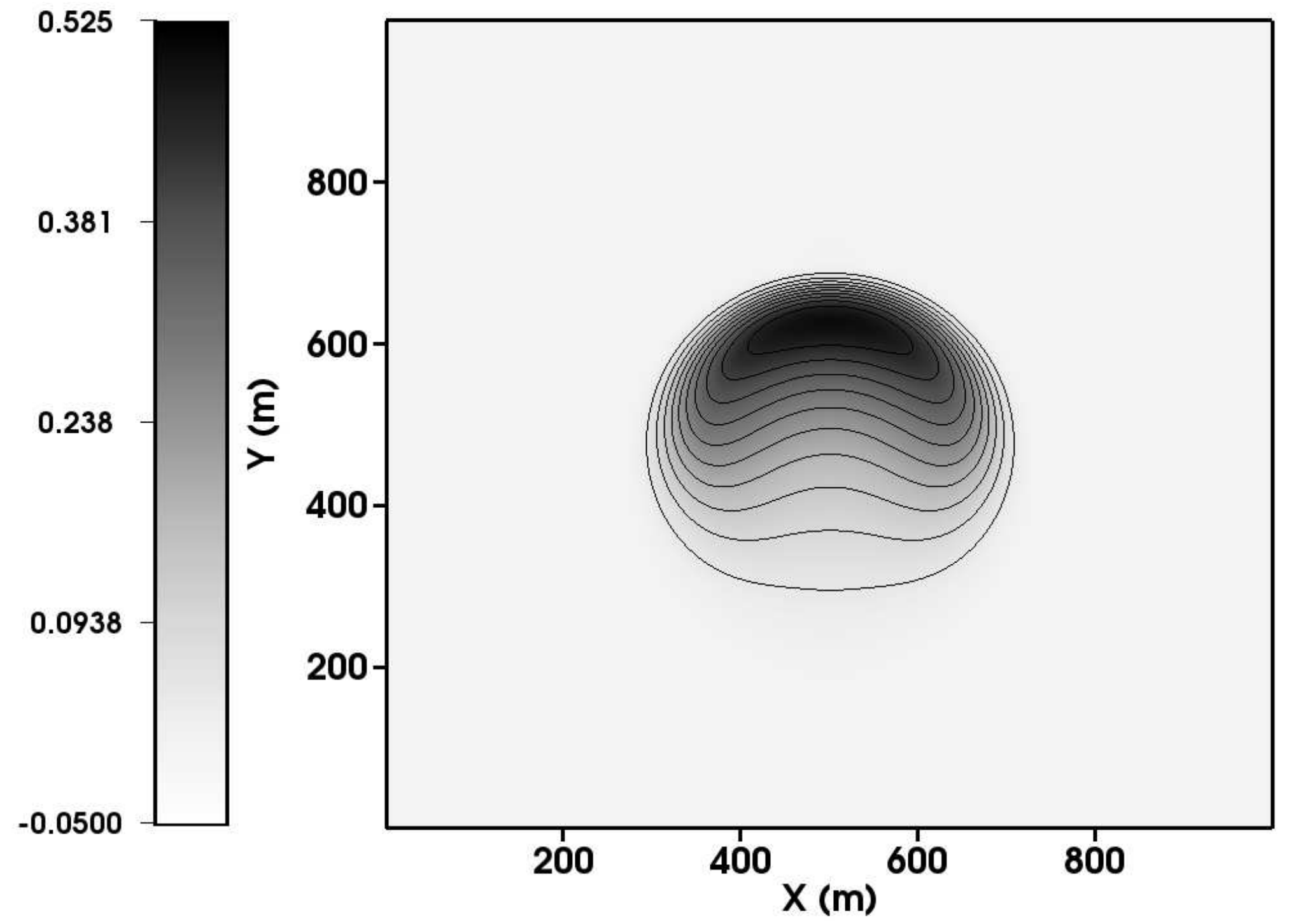} &
\includegraphics[width=0.48\textwidth]{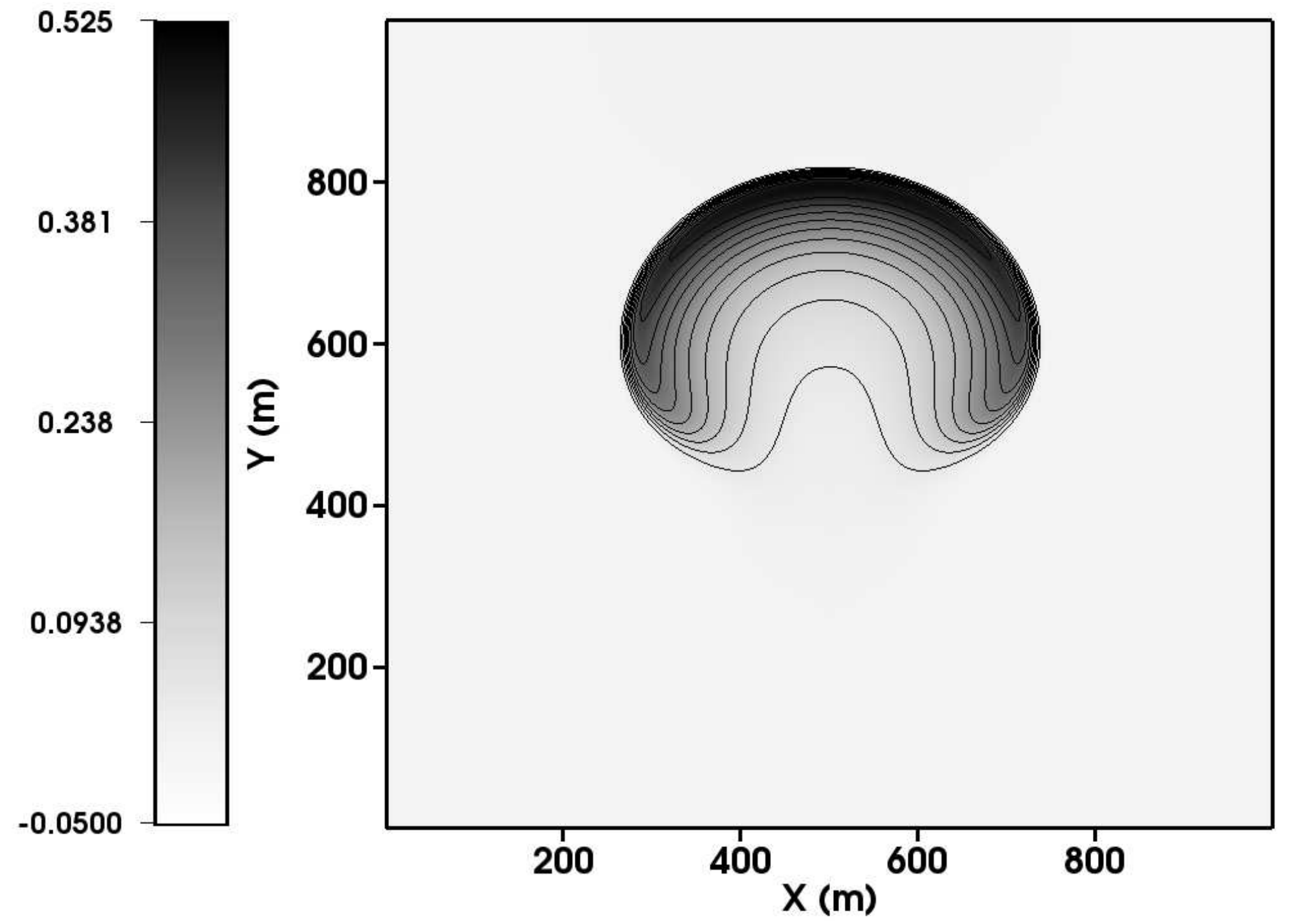} \\
(c) $t = 300 $ $s$ & (d) $t = 450 $ $s$\\
\includegraphics[width=0.48\textwidth]{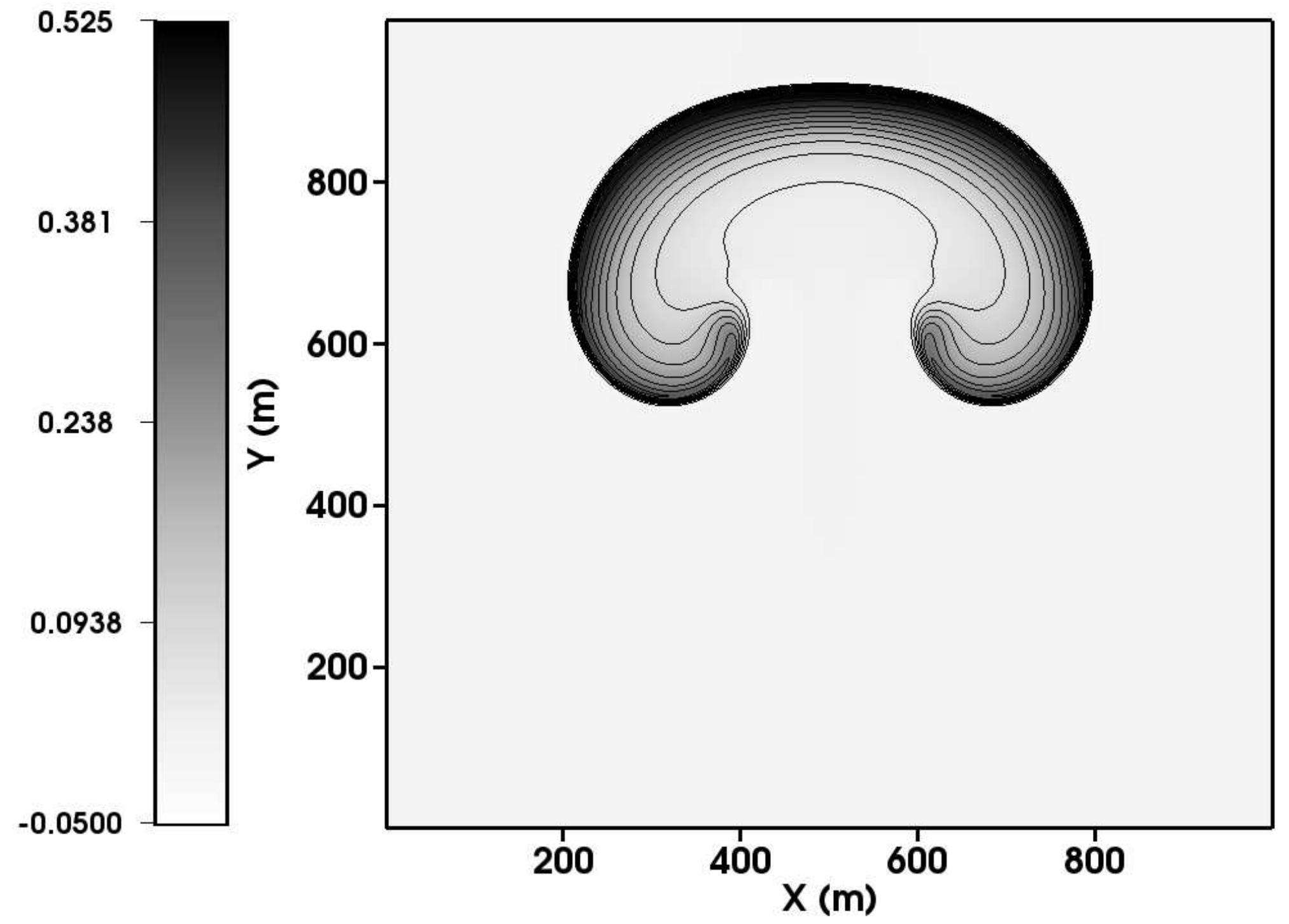} &
\includegraphics[width=0.48\textwidth]{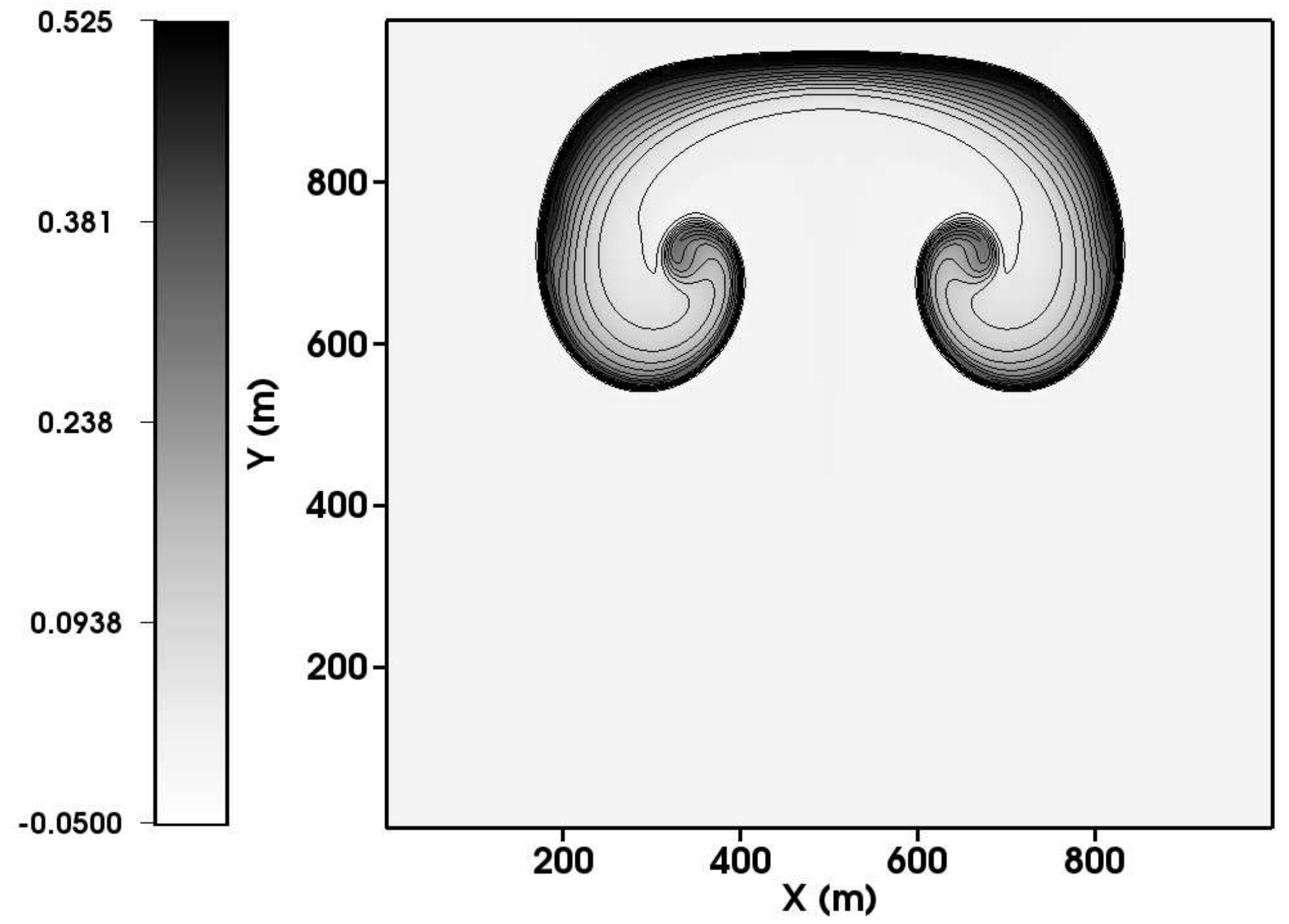} \\
(e) $t = 600 $ $s$ & (f) $t = 700 $ $s$
\end{tabular}
\caption{Rising thermal bubble, potential temperature perturbation ($\Delta \theta$) contours; 10 contour levels between -0.05 and 0.525 at different times.}
\label{fig:bubblecontours}
\end{center}
\end{figure}
\begin{figure}
\begin{center}
\begin{tabular}{cc}
\includegraphics[width=0.47\textwidth]{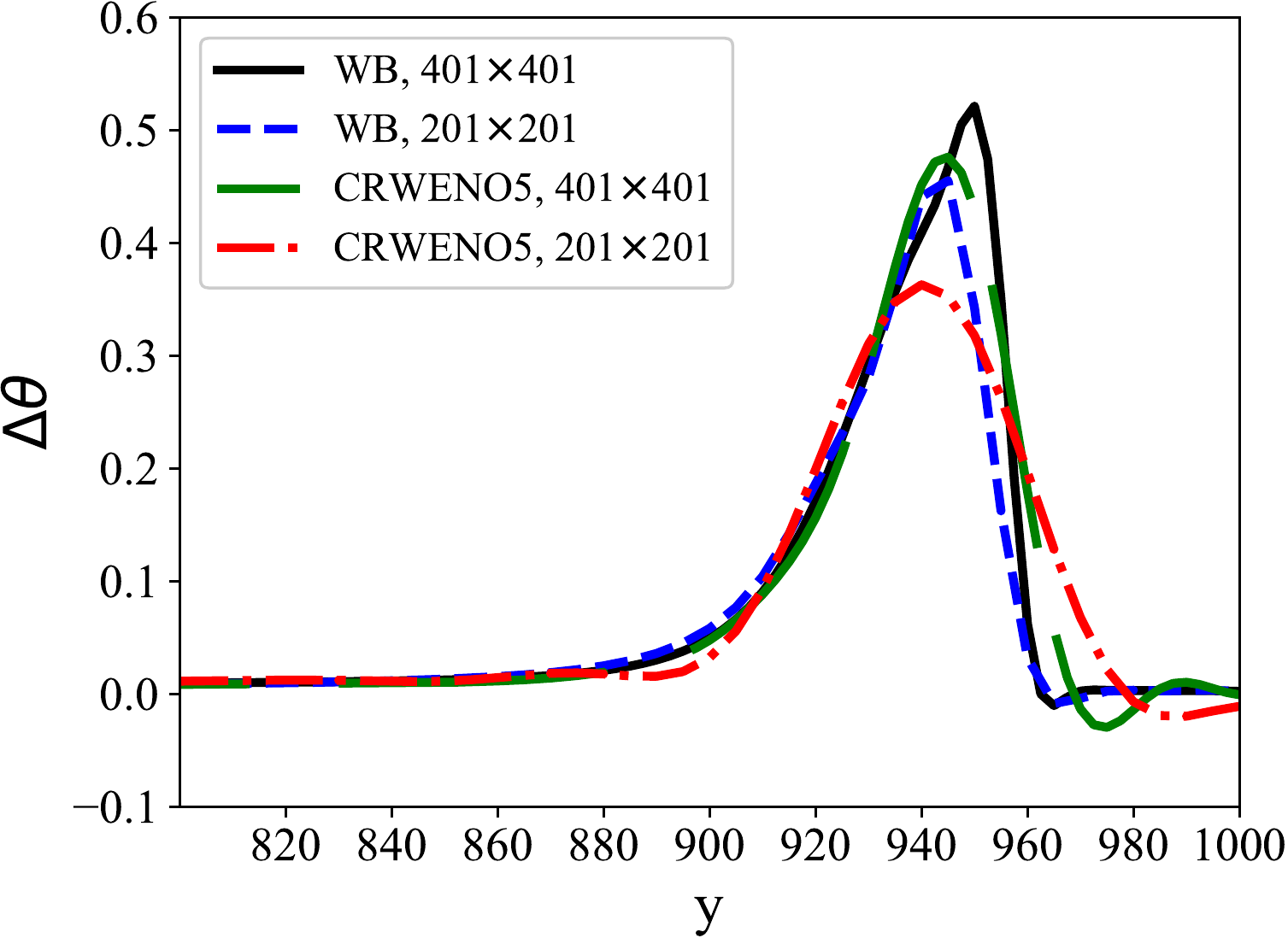} &
\includegraphics[width=0.47\textwidth]{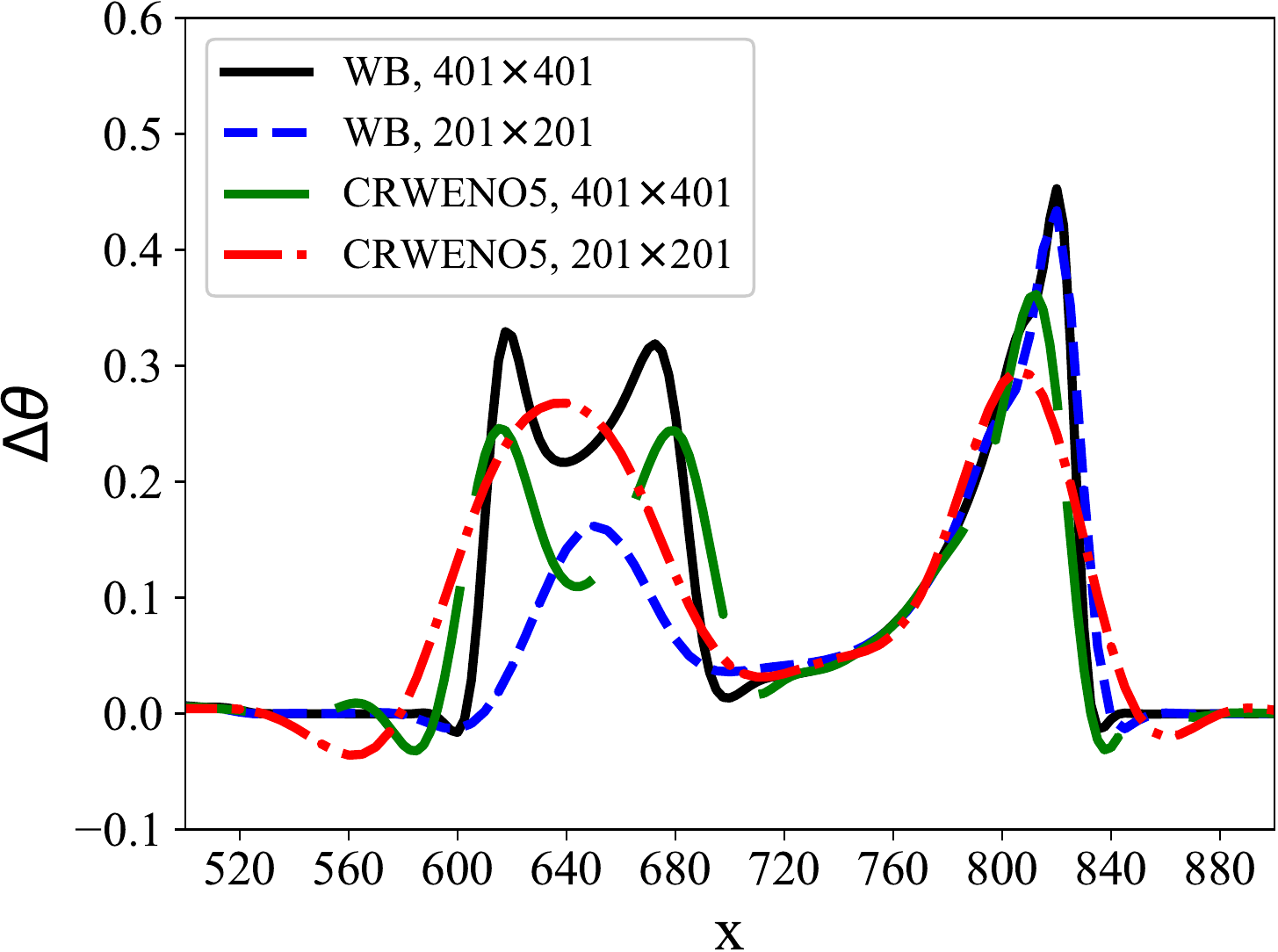} \\ 
(a) $x = 500 $ $m $& (b) $y = 720 $ $m$
\end{tabular}
\caption{Cross-sectional view of potential temperature perturbations for rising thermal bubble at time t = 700 s. }
\label{fig:bubblecrosssection}
\end{center}
\end{figure}
\subsection{Inertia-gravity wave}

This test case~\cite{Skamarock1994}, \cite{Giraldo2008} involves the evolution of a potential temperature perturbation in a channel which is of interest in validation of numerical weather prediction schemes. The hydrostatic state has a constant Brunt-V\"{a}is\"{a}l\"{a} frequency of $N = 0.01 /s$ which yields the potential temperature profile
\[
\thetae = \theta_0 \exp\left(\frac{N^2 y}{g} \right), \qquad \theta_0 = 300 \ K, \qquad g = 9.8 \ m/s^2
\]
The corresponding density and pressure can be written in terms of the Exner pressure $\bar{\Pi}$,
\[
\bar{\Pi} = 1 + \frac{(\gamma - 1) g^{2}}{\gamma R \theta_0 N^2}\left[\exp\left(-\frac{N^2 y}{g} \right) - 1 \right], \qquad \pe = p_0 \bar{\Pi}^{\frac{\gamma}{\gamma - 1}}, \qquad \rhoe = \rho_0 \exp\left(\frac{-N^2 y}{g}\right) \bar{\Pi}^{\frac{1}{\gamma - 1}}
\]
where $p_0 = 10^5 \ N/m^2$, $\rho_0 = p_0/(R \theta_0)$, $\gamma=1.4$ and $R = 287.058 J/(kg \ K)$.  A uniform initial velocity of $u$ = 20 $m/s$ is imposed in the $x$-direction . The initial condition for the hydrostatic density is perturbed by adding a deviation to the potential temperature given by
\[
\Delta \theta(x,y) = \theta_c \sin\left(\frac{\pi y}{h_c}\right) \left[1 + \left(\frac{x - x_c}{a_c}\right)^2\right]^{-1}
\]
where $\theta_c = 0.01 \ K$ is the perturbation strength, $h_c = 10,000 \ m$ is the total height of the domain, $x_c = 1,00,000 \ m$ is the initial horizontal location of the perturbation and $a_c = 5,000 \ m$ is the perturbation half-width. The domain is of size $300 \km \times 10 \km$ with periodic boundary condition along the $x$ axis, while the top and bottom boundaries are treated as solid walls. A grid resolution of $1201\times51$ is used.

Figure~(\ref{fig:gravitywavecontours}) shows the evolution of the initial perturbation due to the generation of horizontal gravity-wave at times $t$ = 0 s, 1000 s, 2000 s and 3000 s. We can see that the centre of imposed perturbation moves from its initial location of $x$ = 100,000 m to $x$ $\approx$ 160,000 m due to the presence of mean $u$-velocity. This shows good agreement with results found in literature~\cite{Giraldo2008}, \cite{Ghosh2016}. For further confirmation, a line plot along the cross-section of $y = 5000 \ m$ is taken and compared with results obtained from CRWENO5~\cite{Ghosh2016}. We can observe that the proposed 2'nd order scheme gives excellent agreement  with a 5'th order WENO scheme.
\begin{figure}
\begin{center}
\begin{tabular}{cc}
\includegraphics[width=0.49\textwidth]{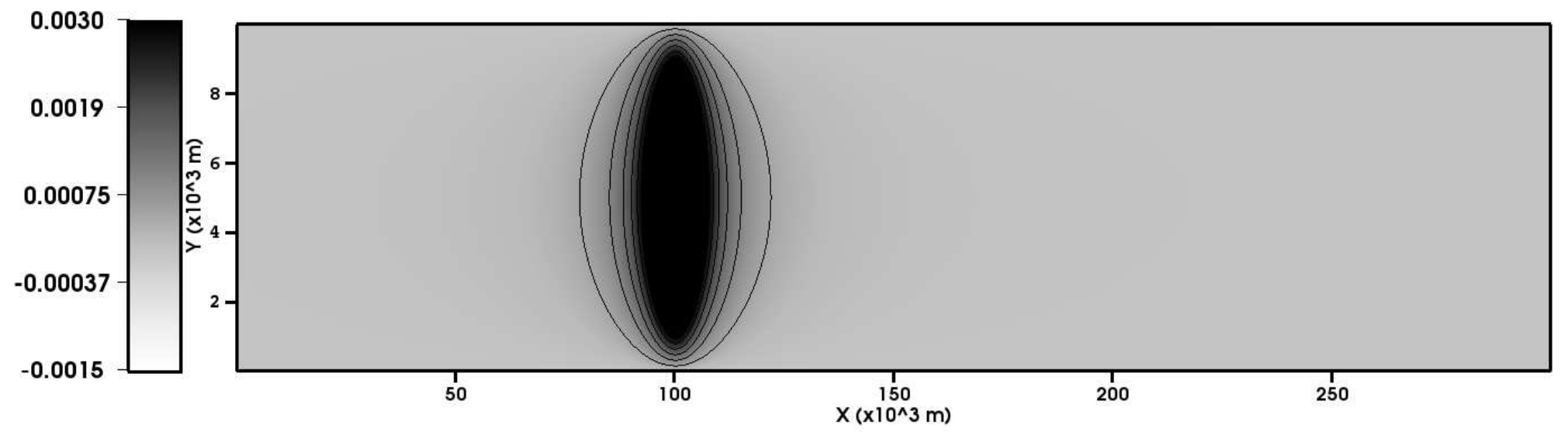} &
\includegraphics[width=0.49\textwidth]{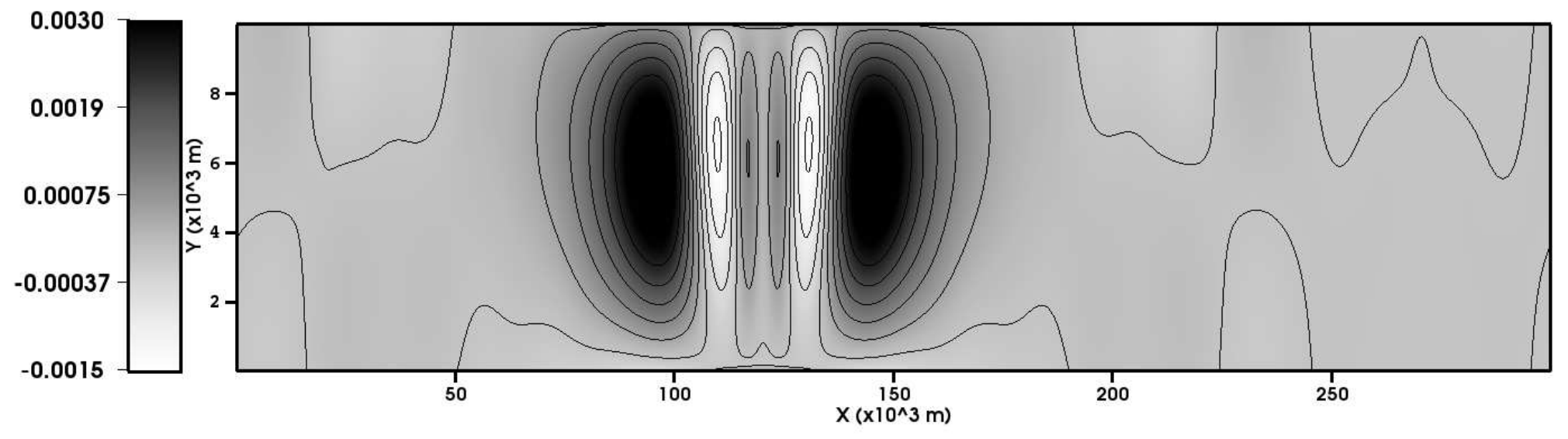} \\
(a) $t = 0 s$  & (b) $t = 1000 s$ \\
\includegraphics[width=0.49\textwidth]{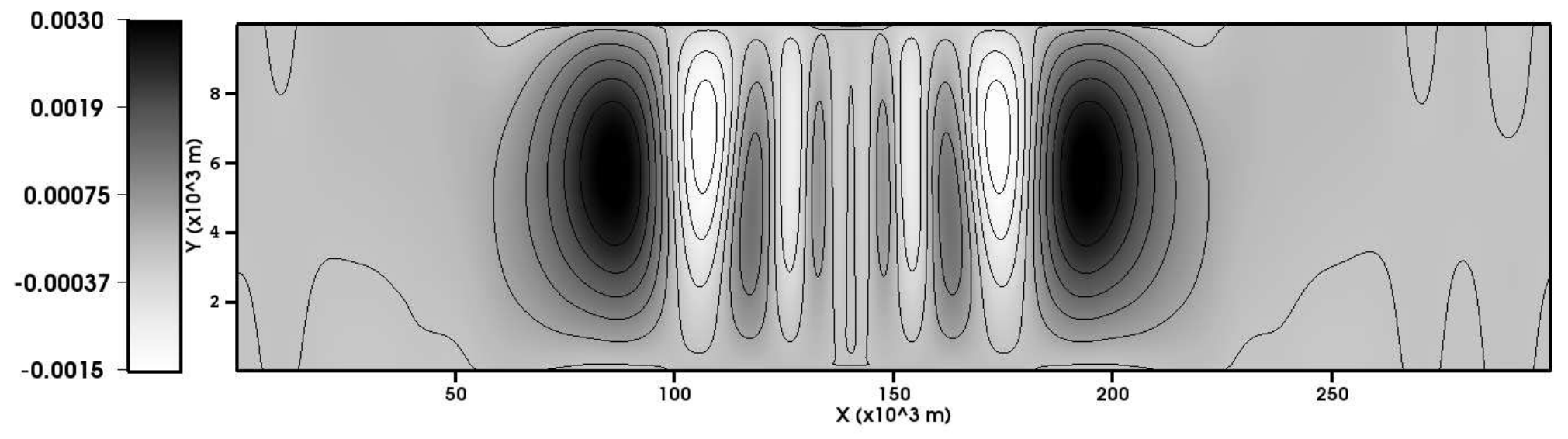} &
\includegraphics[width=0.49\textwidth]{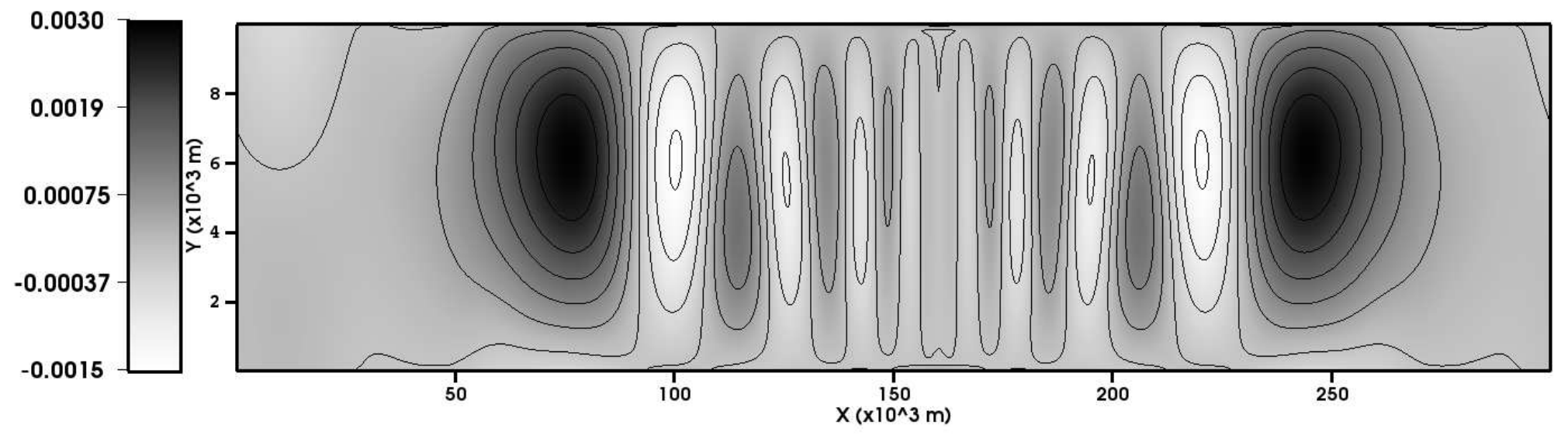} \\
(c) $t = 2000 s$ & (d) $t = 3000 s$
\end{tabular}
\caption{Inertial-gravity wave potential temperature perturbation($\Delta \theta$) contours: 10 contour levels between -0.0015 and 0.003 at different times.}
\label{fig:gravitywavecontours}
\end{center}
\end{figure} 
\begin{figure}
\begin{center}
\includegraphics[width = 0.5\textwidth]{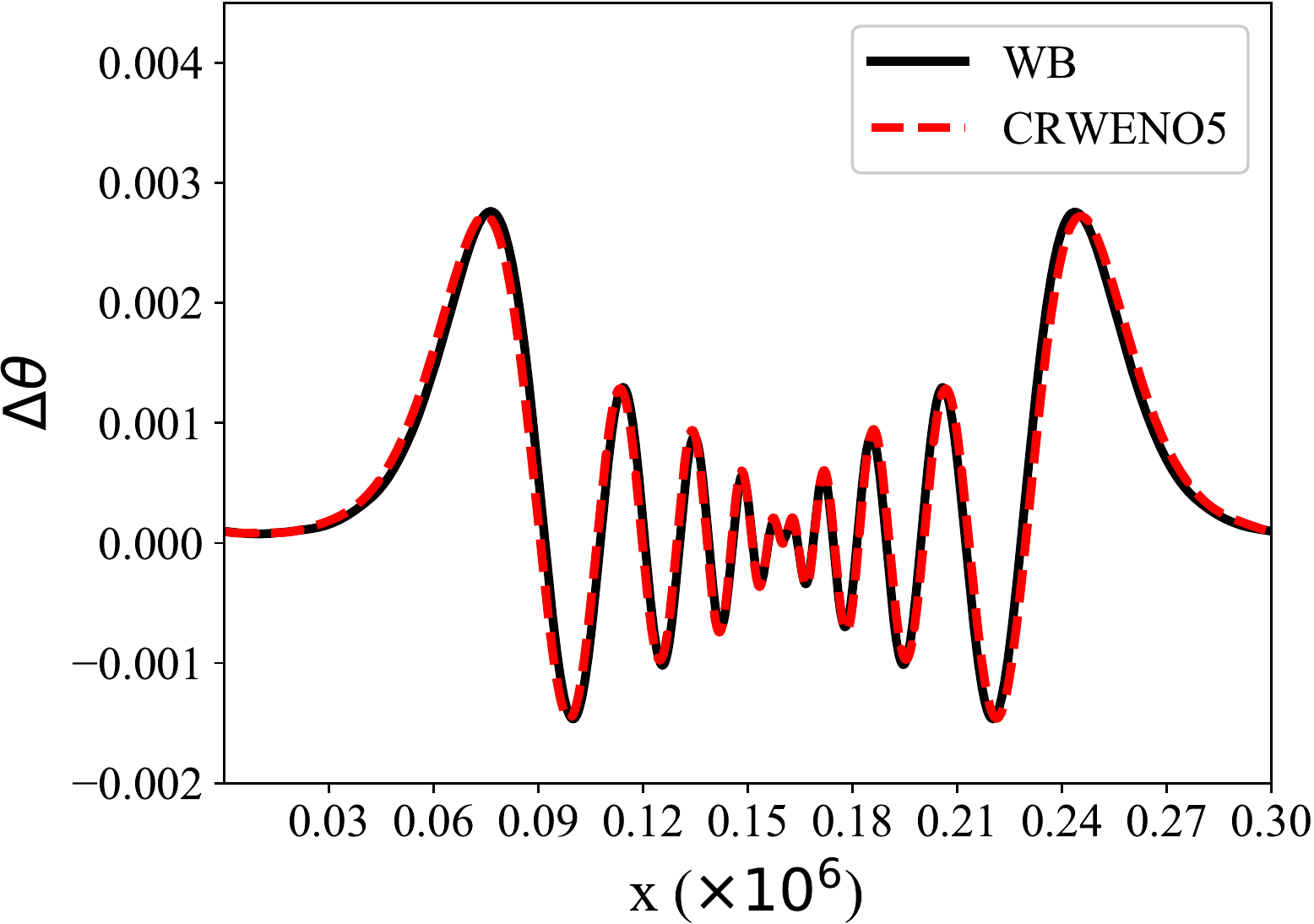}
\caption{Potential temperature perturbations at cross-section of  $y$ = 5000 m for Inertial-gravity wave at final time $t$ = 3000 s.}
\label{fig:wavecrosssection}
\end{center}
\end{figure}
\section{Summary and conclusions}
\label{sec:end}
Hydrostatic solutions depend on the equation of state and some additional factors that may change from problem to problem, and hence it is not possible to construct a well-balanced scheme for all possible situations. We have instead constructed a scheme that admits hydrostatic solutions and these are shown to be second order approximation of the exact hydrostatic solutions. The scheme is general in the sense that it works for any equation of state and does not require \'a priori knowledge of the hydrostatic solution. A restriction is that the well-balanced property holds only when the gravitational force field is aligned with the grid lines. However, even when this condition is not satisfied, we show numerically that the scheme is able to maintain the hydrostatic solution to good accuracy while a non well-balanced scheme would break down. This behaviour can be further improved using radial meshes in case of astrophysical simulations where the gravitational force is predominantly radial, and will form part of our future work. The scheme has been shown to accurately compute small perturbations around hydrostatic solutions and the extensive set of test cases show that it has a wide range of applicability in computing flows under gravitational field.
\section*{Acknowledgements}
This work was supported by the Airbus Foundation Chair on Mathematics of Complex Systems at TIFR-CAM, Bangalore, India. The authors also thank the two anonymous reviewers whose constructive criticism helped us to improve the paper by providing additional comparisons to reference solutions.
\appendix
\section{Importance of well-balanced property}
\label{sec:imp}
In order to explain the significance of well-balancing on the computation of small perturbations, let us consider a generic time dependent PDE
\begin{equation}
\df{u}{t} = R(u)
\label{eq:ode}
\end{equation}
which has a stationary solution $u_e$, i.e., $R(u_e) = 0$. Small perturbations $\up$ around the stationary solution $u_e$ are governed by the linear differential equation
\begin{equation}
\df{\up}{t} = R'(u_e) \up
\label{eq:pode}
\end{equation}
Suppose we approximate the PDE by a numerical scheme
\[
\df{u_h}{t} = R_h(u_h)
\]
where $u_h$ might denote the vector of grid point values. Assume that the scheme is not well-balanced, i.e., $R_h(I_h u_e) \ne 0$ where $I_h$ denotes the interpolation operator. From the consistency of the scheme we have $R_h(I_h u_e) = C h^\alpha \|u_e\|$ for some $\alpha \ge 1$ and some solution norm $\| \cdot \|$ which is usually some Sobolev norm, and $C$ is a constant independent of the mesh size $h$ and the solution $u_e$. Suppose we want to compute the solution for small perturbations around the stationary state, $u_h = I_h u_e + \up_h$ where $\|\up_h \| \ll \| u_e \|$. Then linearizing the numerical scheme around $u_h=I_h u_e$
\[
\df{\up_h}{t} = R_h(I_h u_e + \up_h) = R_h(I_h u_e) + R_h'(I_h u_e)\up_h + O(|\up_h|^2)
\]
To first order the perturbations are governed by the linear equation
\begin{equation}
\label{eq:tenwb}
\df{\up_h}{t} = R_h(I_h u_e) + R_h'(I_h u_e) \up_h = C h^\alpha \|u_e\| + R_h'(I_h u_e) \up_h
\end{equation}
The first term on the right hand side is the error due to non well-balancedness and this error can dominate the linear evolution of small perturbations since $\|u_e\| \gg \|\up_h\|$. If the scheme is well-balanced for the exact hydrostatic solution $u_e$, then $R_h(I_h u_e) = 0$, so that the perturbation equation is
\[
\df{\up_h}{t} = R_h'(I_h u_e) \up_h = R'(I_h u_e) \up_h + [R_h'(I_h u_e) - R'(I_h u_e)] \up_h = R'(I_h u_e) \up_h + C h^\alpha \| \up_h \|
\]
where in the last step, we have assumed the consistency of the scheme in the sense that $\| R'(I_h u_e) - R_h'(I_h u_e) \| = O(h^\alpha)$, i.e., the same order of consistency as the non-linear scheme. In this case, the computed perturbation solution $\up_h$ can be expected to be a good approximation to the true solution governed by equation~(\ref{eq:pode}). Note that $R_h'$ is the linearization of the numerical scheme and this can be expected to approximate the linearized differential operator $R'$ if $R_h$ is smooth at the hydrostatic solution.

In practice, we may not have explicitly available equilibrium solutions in terms of formulae for hydrostatic pressure and density, and we may not be able to construct a scheme for which $R_h(I_h u_e)=0$ holds. In this case, we may find an approximate equilibrium solution $u_{h,e}$ and let us assume that our scheme is well-balanced for this approximate equilibrium solution, i.e., $R_h(u_{h,e}) = 0$. Then the equation for linear perturbations is 
\begin{equation}
\df{\up_h}{t} = R_h(u_{h,e}) + R_h'(u_{h,e}) \up_h = R_h'(u_{h,e}) \up_h
\label{eq:dwb}
\end{equation}
Let us assume that the discrete hydrostatic solution has the same accuracy as the consistency of the scheme $R_h$, i.e., $\| I_h u_e - u_{h,e} \| = O(h^\alpha)$. Then the perturbation equation can be written as
\[
\df{\up_h}{t} = R'(I_h u_e) \up_h + [R_h'(u_{h,e}) - R'(I_h u_e)] \up_h
\]
We can bound the last term on the right hand side as
\begin{eqnarray*}
\| R_h'(u_{h,e}) - R'(I_h u_e) \| &\le& \| R_h'(u_{h,e}) - R_h'(I_h u_e) \| + \| R_h'(I_h u_e) - R'(I_h u_e) \| \\
&\le& \| R_h''(u_{h,e}) \| \cdot \| u_{h,e}-I_h u_e \| + O(h^\alpha) \\
&=& O(h^\alpha) + O(h^\alpha)
\end{eqnarray*}
so that the perturbation equation is of the form
\[
\df{\up_h}{t} = R'(I_h u_{e}) \up_h + C h^\alpha \| \up_h \|
\]
Hence the discretely well-balanced scheme is consistent with the exact perturbation equation~(\ref{eq:pode}) and more importantly, the truncation term is proportional to $\| \up_h \| \ll \| u_e \|$ unlike the truncation term in the non-well-balanced scheme~(\ref{eq:tenwb}) which is proportional to $\| u_e \|$. We can hence expect the numerical solutions from such a discretely well-balanced scheme to provide a good approximation for the computation of small perturbations around hydrostatic solutions. In the present work, we consider second order schemes so that $\alpha=2$, and we also show we can construct discrete hydrostatic solutions that are second order accurate approximations of the exact solutions, which are then well-balanced by our scheme.
\section*{References}
\bibliography{bbibtex}

\end{document}